\documentclass[amssymb,amsfonts]{amsart}
\usepackage[english]{babel}

\makeatletter
\numberwithin{equation}{section}
\numberwithin{figure}{section}
\theoremstyle{plain}
\usepackage{amsthm}
\usepackage{appendix}
\usepackage{amsmath,amssymb,amscd}
\usepackage[english]{babel}

\newtheorem{rem}{Remark}[section]
\newtheorem{thm}{Theorem}[section]
\newtheorem{lem}{Lemma}[section]

\newcommand{\setting}{In the setting of Lemma \ref{bootstrap}, the following  estimates hold  for $t\in [0,T]$}
\newcommand{\settingos}{In the setting of Lemma \ref{osboot}, the following estimates hold for $t\in [0,T]$}
\newcommand{\dzeta}{\nabla\tilde{\zeta}}
\newcommand{\zzeta}{\tilde{\zeta}}
\newcommand{\eps}{\varepsilon}
\newcommand{\tts}{\tilde{\sigma}}

\newcommand{\qbb}{\tilde{Q}_{b}}
\newcommand{\JJJ}{\mathcal{J}}
\newcommand{\RRR}{\mathbb{R}}
\newcommand{\tl}{\tilde{\lambda}}

\newcommand{\enj}{\int |\nabla \eps^{j}|^{2}}
\newcommand{\lmj}{\int |\eps^{j}|^{2}e^{-|y|}}
\newcommand{\qbj}{\tilde{Q}_{b_{j}}}

\newcommand{\epj}{\eps^{j}}
\newcommand{\epjo}{\eps^{j}_{1}}
\newcommand{\epjt}{\eps^{j}_{2}}

\newcommand{\xs}{\frac{dx_{j}}{ds_{j}}}
\newcommand{\ls}{\frac{d\lambda_{j}}{ds_{j}}}

\newcommand{\tailj}{\int |\nabla \tilde{\eps}^{j}|^{2}+\int |\tilde{\eps}^{j}|^{2}e^{-|y|}}

\newcommand{\tzj}{\tilde{\zeta}_{b_{j}}}

\newcommand{\du}{\nabla^{\nu}}
\newcommand{\tkj}{[\tau_{k}^{j},\tau_{k}^{j+1}]}
\newcommand{\DDD}{\mathcal{D}}
\newcommand{\fd}{u_{0,\boldsymbol{\lambda},\boldsymbol{x}}}
\newcommand{\llll}{\boldsymbol{\lambda}}
\newcommand{\xxxx}{\boldsymbol{x}}
\newcommand{\bbb}{\boldsymbol{\beta}}
\newcommand{\AAA}{\mathcal{A}}
\newcommand{\FFF}{\mathcal{F}}
\makeatother
\usepackage{fancyhdr}
\usepackage{babel}

\begin{document}

\title{log-log blow up solutions blow up at exactly m points}
\begin{center}
\author{Chenjie Fan $^{\ddagger} $ $^\dagger$}
\end{center}
\thanks{\quad \\$^{\ddagger} $Department of Mathematics,  Massachusetts Institute of Technology,  77 Massachusetts Ave,  Cambridge,  MA 02139-4307 USA. email:
cjfan@math.mit.edu.\\
$^\dagger$ The author is partially supported by NSF Grant DMS 1069225, DMS 1362509 and DMS 1462401.}
\maketitle
\begin{abstract}
We study the focusing mass-critical nonlinear Schr\"odinger equation, and construct certain solutions which blow up at exactly $m$ points according to the log-log law. 
\end{abstract}

\section{Introduction}
We consider the Cauchy Problem for the   mass-critical focusing nonlinear Schr\"odinger equation (NLS) on $\RRR^{d}$ for $d=1,2$:
\begin{equation}\label{nls}
(NLS)
\begin{cases}
iu_{t}=-\Delta u-|u|^{\frac{4}{d}}u,\\
u(0)=u_{0}\in H^{1}(\RRR^{d}).
\end{cases}
\end{equation}
Problem \eqref{nls} has three  conservation laws:
\begin{itemize}
\item Mass:
\begin{equation}
M(u(t,x)):=\int |u(t,x)|^{2}dx=M(u_{0}),
\end{equation}
\item Energy:
\begin{equation}
E(u(t,x):=\frac{1}{2}\int |\triangledown u(t,x)|^{2}dx-\frac{1}{2+\frac{4}{d}}\int|u(t,x)|^{2+\frac{4}{d}}dx=E(u_{0}),
\end{equation}
\item Momentum:
\begin{equation}
P(u(t,x)):=\Im(\int \triangledown u(t,x) \overline{u(t,x)})dx=P(u_{0}),
\end{equation}
\end{itemize}
and  the following symmetry:
\begin{enumerate}
\item Space-time translation: If $u(t,x)$ solves \eqref{nls}, then $\forall t_{0}\in \RRR, x_{0}\in \RRR^{d}$, we have  $u(t-t_{0}, x-x_{0})$ solves \eqref{nls}.
\item Phase transformation: If $u$ solves \eqref{nls}, then $\forall \theta_{0}\in \RRR$, we have $e^{i\theta_{0}}u$ solves \eqref{nls}. 
\item Galilean transformation: If $u(t,x)$ solves \eqref{nls}, then $\forall \beta\in \RRR^{d}$, we have  $u(t,x-\beta t)e^{i\frac{\beta}{2}(x-\frac{\beta}{2}t)}$ solves \eqref{nls}.
\item Scaling: If $u(t,x)$ solves \eqref{nls},  then $\forall \lambda \in \RRR_{+}$,  we have $u_{\lambda}(t,x):=\frac{1}{\lambda^{\frac{d}{2}}}u(\frac{t}{\lambda}, \frac{x}{\lambda})$ solves \eqref{nls}.
\item Pseudo-conformal transformation: If $u(t,x)$ solves \eqref{nls}, then $\frac{1}{t^{\frac{d}{2}}}\bar{u}(\frac{1}{t},\frac{x}{t})e^{i\frac{|x|^{2}}{4t}}$ solves \eqref{nls}.
\end{enumerate}
\subsection{Setting of the problem and statement of the main result}
The equation \eqref{nls} is mass-critical since the conserved quantity given by the mass is invariant under scaling symmetry. It is called focusing since the conserved quantity given by the energy is not coercive.\par
The NLS \eqref{nls} was proved to be  locally  well-posed (LWP) in $H^{1}$ by  Ginibre and Velo in \cite{ginibre1979class}. This means that  for any initial data in $H^{1}$, gives rise  to  a unique solution  $u\in C([0,T];H^{1})$, for some time $T$ and the solution depends continuously on the initial data. Since this equation is mass-critical, and $H^{1}$ is a norm subcritical with respect to the  $L^{2}$
 norm, one can take  $T=T(\|u_{0}\|_{H^{1}}) >0$.
By LWP  it is not hard to see that  if the solution is defined in $[0,T_{0})$ and cannot be extend beyond $T_{0}$, then it has to be that  $\lim_{t\rightarrow T_{0}}\|u(t)\|_{H^{1}}=\infty$.  In this case we say that $ u$  blows up in finite time in $T_{0}$. A solution that does not blow up in finite time is called global.\par
 It turns  out that not all solutions of  \eqref{nls} are global.  The classical virial identity (\cite{glassey1977blowing}) indicates the existence of solutions which blow up in finite time. By direct calculation one has:
\begin{equation}\label{classicalvirial}
\partial_{t}^{2}\int |x|^{2}|u|^{2}=4\partial_{t}\Im(\int x\nabla u \bar{u})=16E(u_{0}).
\end{equation}
and \eqref{classicalvirial} immediately indicates that if  $xu_{0}\in L^{2}$ and $E(u_{0})<0$, then the solution $u$ must blow up in finite time. \par
Some questions are then natural: If a solution blows up in finite time, what is the mechanism for  singularity formation, i.e. how does one describes the reason for a solution to blow up and how does one  describe the behavior of the solution when it approaches the blow up time?\par
Virial identity  \eqref{classicalvirial} on its own  does not give the answer to these questions, and up to the best of our knowledge, blow up solutions to \eqref{nls}  blow up way faster than the blow up time  predicted by \eqref{classicalvirial}.\par
Before we continue, let's first introduce the ground state, which is one typical object appearing in the study of {focusing} equations.  
The ground state $Q=Q(x)$ is the unique positive solution in $L^{2}(\RRR^d)$ that solves
\begin{equation}\label{ground}
-\Delta W+W=|W|^{\frac{4}{d}}W.
\end{equation}
 $Q$ is an explicit function when d=1, $Q(x)=\left(\frac{3}{ch2x}\right)^{\frac{1}{4}}$, and $Q$ is smooth and decays exponentially for $d=1,2$.  
It is a direct calculation to check that $Q(x)e^{it}$ is a solution of \eqref{nls}. In general we call solutions  to \eqref{nls}  of the  form $W(x)e^{it}$ {soliton solutions}. We have several quick remarks:
\begin{itemize}
\item $W(x)e^{it}$ solves \eqref{nls} if and only if $W(x)$ solves \eqref{ground}.
\item A pure variational argument (\cite{weinstein1983nonlinear}) shows that any $H^{1}$ solutions to \eqref{nls} which initial data $u_{0}$ has  mass below the mass of the ground state,  does not blow up. 
\item To compare with \eqref{classicalvirial}, one can use variational arguments (\cite{weinstein1983nonlinear}) to show that for any function $f\in H^{1}$ such that  $\|f\|_{2}\leq \|Q\|_{2}$, one has $E(f)\geq 0$, and in particular one can derive from \eqref{ground} that $E(Q)=0$.
\end{itemize}
\par
By applying the pseudo-conformal transformation to the global soliton solution $Q(x)e^{it}$, one obtains the explicit blow up solution  
\begin{equation}\label{stlike}
S(t,x):=\frac{1}{|t|^{\frac{d}{2}}}Q(\frac{x}{t})e^{-\frac{|x|^{2}}{4t}+\frac{i}{t}}.
\end{equation}
We remark that $\|S(t,x)\|_{2}=\|Q\|_{2}$, and  this is (up to symmetries)  the minimal mass blow up solution, \cite{merle1993determination}, (recall no $H^{1}$ solution with mass below $\|Q\|_{2}^{2}$ can blow up in finite time). We also remark by an  easy scaling argument, that if a solution $u$ to \eqref{nls} blows up in finite time T, then its blow up rate for $\|\nabla u(t)\|_{2}$ has a lower bound:
\begin{equation}
\|\nabla u(t)\|_{2}\gtrsim \frac{1}{\sqrt{T-t}},
\end{equation}  
and the explicit blow up solution $S(t,x)$ has blow up rate $\frac{1}{|t|}$.\par
Despite the fact that the  Schr\"{o}dinger equation has the infinite speed of propagation, $S(t,x)$ blows up locally in the physical space, and more precisely  by looking at its explicit formula \eqref{stlike}, we can say that the solution blows up at $x_{0}:=0\in \RRR^{d}$.
Relying on the fact that the blow up behavior of $S(t,x)$ is local, by a certain  compactness argument and using $S(t,x)$ as basic building blocks, Merle (\cite{merle1990construction}) constructed a solution $u$ to \eqref{nls} which blows up at $k$ points in finite time T,  and near blow up time $T$ it  has the following asymptotic:
\begin{equation}
u(t,x)\sim \sum_{i=1}^{k}S(\frac{t-T}{\lambda^{2}_{i}},\frac{x-x_{i}}{\lambda_{i}}), \lambda_{i}>0, x_{i}\neq x_{j}, \forall i\neq j.
\end{equation}
\par
The goal of this paper is to use the log-log blow up solutions, that we recall below, as basic building blocks instead of $S(t,x)$ in\eqref{stlike} to construct blow up solutions to \eqref{nls}.
The so-called log-log blow up solutions are solutions to \eqref{nls} which blow up in finite time $T$ with blow up rate  $\|\nabla u(t)\|_{2}\sim \left(\frac{\ln|\ln |T-t||}{T-t}\right)^{\frac{1}{2}}$. Such solutions had been suggested numerically by Landman, Papanocolaou, Sulem, Sulem , \cite{landman1988rate} and first constructed by Perelman, \cite{perelman2001blow}, and later intensively studied by Merle and Rapha\"el, \cite{merle2005blow}, \cite{merle2003sharp} , \cite{merle2006sharp}, \cite{merle2004universality}, \cite{raphael2005stability}, \cite{merle2005profiles}. 
Merle and Rapha\"el consider the solutions to \eqref{nls}  with initial data $u_{0}$ such that
\begin{equation}\label{supercritical}
\|Q\|_{2}<\|u_{0}\|_{2}<\|Q\|_{2}+\alpha,
\end{equation}
where $\alpha$ is a universal small positive constant.\\
Let's give a quick  summary of the Merle and Rapha\"el's results relevant for this paper.\\
When \eqref{supercritical} holds, they show that
all solution with strict negative energy blow up in finite time $ T $ according to the log-log law: $\|\nabla u(t)\|_{2}\sim \left(\frac{\ln|\ln |T-t||}{T-t}\right)^{\frac{1}{2}}$. And more precisely, near blow up time $T$, the solution $u$ can be decomposed as:
\begin{equation}\label{loglogblowup}
\begin{aligned}
&u(t,x)=\frac{1}{\lambda(t)^{\frac{d}{2}}}Q(t,\frac{x-x(t)}{\lambda(t)})+\Xi(t,x),\\
&\lambda(t)\in \RRR^{+}, \frac{1}{\lambda(t)}\sim \sqrt{\frac{\ln|\ln T-t |}{T-t}},\\
&\lim_{t\rightarrow T}\lambda(t)\|\Xi(t)\|_{H^{1}}=0, 
\end{aligned}
\end{equation}
They also show that the log-log blow up solution is stable in the sense that the initial data, which generate log-log blow up solutions, is an open set in $H^{1}$.
\begin{rem}
Merle and Rapha\"el's results are more general than what has been recalled above. For example the strict negative condition  can be relaxed.
\end{rem}
As mentioned above, the log-log blow up behavior is stable under $H^{1}$ perturbation. What's more, this log-log blow up behavior is local (in physical space) in a certain sense,  in spite of  infinity speed of propagation of the NLS equation. In \cite{planchon2007existence}, Planchon and Rapha\"el constructed a log-log blow up solutions on a bounded domain $\Omega$ of $\RRR^{d}$ with $u(t)\in H_{0}^{1}(\Omega)$. In particular, if one looks at solutions of the  form \eqref{loglogblowup}, they prove that  $x(t)$ has a limit as $t$ approaches the blow up time, hence it  makes sense to say that such solutions blow up at a certain point\footnote{Note that we already remarked this property for explicit blow up solutions $S(t,x)$.} $x_{\infty}:=\lim_{t\rightarrow T} x(t)$ in $\RRR^{d}$. The key element in their proof is a very robust bootstrap argument. It turns out that the analysis of log-log blow up solutions has a very suitable bootstrap structure.\\
In this work, as an analogue  of \cite{merle1990construction},  we  use log-log blow up solutions as building blocks to construct solutions which blow up at (exactly) $m$ points. 
To be precise, we  show:
\begin{thm}\label{thmmain}
For $d=1,2$, for each positive integer $m$,  and given any $m$ different points $x_{1,\infty},...x_{m,\infty}$ in $\RRR^{d}$, there exists a solution $u$ to \eqref{nls} such that $ u $ blows up in finite time $T$, and for $t$ close enough to $T$,
\begin{equation}\label{doubleloglog}
u(t,x)=\sum_{j=1}^{m}\frac{1}{\lambda_{j}^{\frac{d}{2}}(t)}Q(\frac{x-x_{j,\infty}}{\lambda_{j}(t)})e^{-i\gamma_{j}(t)}+\Xi(t,x).
\end{equation}
where, for $ j=1,\dots,m, $
\begin{equation}
 \frac{1}{\lambda_{j}(t)}\sim\sqrt{\frac{\ln|\ln T-t |}{T-t}}, \quad 
\text{ and  } \quad  \lambda_{j}\|\Xi(t)\|_{H^{1}} \xrightarrow{t\rightarrow T} 0, 
\end{equation}
i.e $\Xi $ can be viewed as an error term.  In particular, since the $m$ given points are arbitrary, the solutions do not necessarily have any symmetry restriction. 
\end{thm}
\begin{rem}
Same construction works on torus, $\mathbb{T}^{d}, d=1,2$.
\end{rem}
\begin{rem}
The result of \cite{planchon2007existence} already implies the existence of {symmetric}  solutions which  blow up at two points according to log-log law.  In fact one first constructs a solution to NLS on the half line/plane $H:=\RRR^{d}_{+}$, such that the solution blows up at one point according to log-log law and satisfies the Dirichlet condition $u\equiv 0$ on $\partial \RRR^{d+}$, then, by extending this solution symmetrically,  one easily derives the solution which blows up at two points in the whole line/plane.  Similarly, one can construct solutions that blow up at a even number of points  according to log-log law, but they will have a very strong symmetry. See Corollary 1 in \cite{planchon2007existence}.
\end{rem}
\begin{rem}
The work of \cite{merle1990construction} uses idea of "integrate from infinity", which means one needs to evolve the data backward.  This does not seem to direct work in this setting because of the remainder term $\Xi(t,x)$. In this work, we will evolve the data forward. 
\end{rem}
\begin{rem}
We point out two applications of Theorem \ref{thmmain}. First, it implies the existence of large mass log-log blow up solutions. More general results on this direction have been obtained in \cite{merle2005one}. Second, for those who are familiar with standing ring blow up solutions, \cite{raphael2006existence},\cite{raphael2009standing}, our construction in 1 D case implies the existence of multiple-standing- ring blow up solutions for quintic NLS on dimension $N\geq 2$. To be precise, one can construct a radial solution $u$ to the following Cauchy Problem:
\begin{equation}
\begin{cases}
iu_{t}+\Delta u=-|u|^{4}u,\\
u_{0}\in H^{N_{2}(N)}(\RRR^{N}).
\end{cases}
\end{equation}
such that $u$ blows up in finite time with  log-log blow up rate and near blow up time $T$, $u$ has the following asymptotic
\begin{eqnarray}
u(t,x)=u(t,r)=\sum_{j=1}^{m}\frac{1}{\lambda_{j}(t)^{1/2}}P(\frac{r-r_{j}(t)}{\lambda_{j}(t)})e^{-i\gamma_{j}(t)}+\Xi(t,x),\\
 \frac{1}{\lambda_{j}(t)}\sim\sqrt{\frac{\ln|\ln T-t |}{T-t}}, \quad 
\text{ and  } \quad  \lambda_{j}\|\Xi(t)\|_{H^{1}} \xrightarrow{t\rightarrow T} 0, \\
\lim_{t\rightarrow T}r_{j}(t)=r_{j,\infty}>0.
\end{eqnarray}
where $P$ is the unique positive $L^{2}$ solution which solves \eqref{ground} for $d=1$.
\end{rem}
From this point on, we fix $m\in \mathbb{N}$ as the number of blow up points. We will do proof only for $d=2$, the case $d=1$ follows similarly.

To construct a solution satisfying Theorem \ref{thmmain} without relying on certain symmetry property of the initial data, the  most intuitive process  to follow  is that 
one first prepares  $m$ log-log blow  solutions $$u^{j}(t,x)=\frac{1}{\lambda^{\frac{d}{2}}_{j}}Q(\frac{x-x_{j}(t)}{\lambda_{j}})e^{-i\gamma_{j}(t)}+\Xi_{j}, j=1,\dots,m$$
satisfying \eqref{loglogblowup}. Then one shows that the solution to \eqref{nls} with initial data $\sum_{j=1}^{m}u^{j}(0)$ evolves  approximately as $\sum_{j=1}^{m}u^{j}(t)$, if one  assumes that at the initial time all solutions are very close to blowing  up,  i.e. $\lambda_{j}(0), j=1,\dots,m$ is very small, and they are physically separated, i.e. $\min_{j\neq j'}|x_{j}(0)-x_{j'}(0)|$ is very large.
To achieve this, one needs some mechanism to decouple  $u^{1}$ , $u^{2},\dots u^{m}$. Our choice is to require extra smoothness outside the (potential) singular points. Roughly speaking, we  find neighborhoods $U_{j}$ of $x_{j,0}\equiv x_{i}(0)$, $j=1,\dots,m$, and show that the solutions  keep very high regularity outside these $m$ neighborhoods. 
This approach is motivated by the work of Rapha\"el and Szeftel, \cite{raphael2009standing}. They consider the focusing quintic NLS on $\RRR^{d}$ and they require their data to be radial and in $H^{d}$. By the  radial symmetry assumption, the problem can be understood in polar coordinates as  a perturbation of the 1D-focusing-quintic NLS.  The goal of Rapha\"el and Szeftel is to construct solutions that  blow up at a sphere (or a ring). The crucial point in their paper is to understand the propagation of singularity/regularity. They show that all the singularities  are   kept around the sphere where the solution is supposed to blow up, and the solution is kept bounded in $H^{\frac{d-1}{2}}$
outside the sphere. Note that, thanks to the radial assumption, the authors  are using  the  1D NLS to model their solutions  \cite{raphael2009standing}.  See also \cite{raphael2006existence}, which indeed has the same spirit as \cite{raphael2009standing}, but in the setting \cite{raphael2006existence}, one does not need to pursue high regularity.\par
We  show the following theorem.
\begin{thm}[propagation of regularity]\label{regularity}
For any given $K_{1}>1,K_{2}>2 $, {not necessarily integers}, if  $K_{1}<\frac{K_{2}}{2}$, then we construct a solution $u$ in $H^{K_{2}}$ 
to \eqref{nls} that blows up according to the log-log law as \eqref{loglogblowup} at finite time $T$ and such that
\begin{itemize}
\item  $sup_{t\in (0,T)}|x(t)-x(0)|<\frac{1}{1000}$.
\item $u(t,x)$ is bounded in $H^{K_{1}}$ when restricted in $|x-x_{0}|\geq 1/2$, 
\end{itemize}
\end{thm}
\begin{rem}
The choice of  special numbers $\frac{1}{2}$, $1/1000$ is (of course) just for concreteness and simplicity.
\end{rem}
\begin{rem}
When d=1, and $K_{2}$ is an integer, and one can prove Theorem \ref{regularity} for $K_{1}\leq \frac{K_{2}-1}{2}$ by sightly modifying the language  of \cite{raphael2009standing}. Their method is based on a  bootstrap  argument and  a certain pesudo-energy.. When d=2, Rapha\"el and Szeftel's method does not seem to directly work, and one should be able to use the argument in \cite{zwiers2010standing} to prove Theorem \ref{regularity} for $K_{1}\leq \frac{K_{2}-1}{2}$ when $K_{2}$ is integers.
 Our proof improves the previous results in the two aspects. We can take $K_{1}<\frac{K_{2}}{2}$ and we do not require $K_{2}$ to be integer.  Our proof is written  more in a  harmonic analysis style, relying on the  (upside-down) I-method, \cite{colliander2002almost},\cite{sohinger2010bounds}, interpolation and  Strichartz estimate. 
\end{rem}
\begin{rem}
When $K_{2}\geq K_{1}=1$, Theorem \ref{regularity} is implied by the work of  Holmer and Roudenko in \cite{holmer2012blow}.
\end{rem}
One should understand Theorem \ref{regularity} as a  proof for the fact that the log-log blow up behavior is local, in the sense that  it does not propagate  singularity outside the blow up point. This will help us to decouple the $m$"solitons" , (we sometimes also call them bubbles), in our construction of blow up solution. 

\begin{rem}
Because of the good localized property of log-log blow up, one can even work on manifolds, see \cite{godet2012blow} for work in this direction.
\end{rem}

\begin{rem}
It is not always true  in this kinds of problems, (if one doesn't put some restriction on the data), that the $m$ "solitons" or $m$ bubbles will be decoupled.  Different bubbles may interact with each other in a strong way.  See very recent work \cite{martel2015strongly} for this direction.
\end{rem}

Once one can somehow decouple the $m$ soliton $u^{1},...u^{m}$, then we will use some topological argument to construct initial data,  and  balance those $m$ bubbles and make them blow up in the same time. And, prescription of the blow up points is actually more subtle than making $m$ bubbles blow up at the same time. Fortunately, by taking advantage of the sharp dynamic of log-log blow up, it can still be achieved by certain topological argument. We remark, it is typical that one may rely on soft topological argument rather than pure analysis to prove things like this, see \cite{merle1992solution}, \cite{planchon2007existence},\cite{cote2011construction}, through one needs to find different topological argument in different settings.

\subsection{Notation}
Throughout this article, $\alpha$ is used to denote a universal small number, $\delta(\alpha)$ is  a small number depending on $\alpha$ such that $\lim_{\alpha\rightarrow 0}\delta(\alpha)=0$.
We use $\delta_{0}, \delta_{1},\dots$ to denote universal constant (they are usually small, but don't depend on $\alpha$). We use $C$ to denote a large constant, it usually changes line by line.
We also use $c$, $\eta$ and  $a$ to denote small constants. For any constant $r$, we use $r\pm$ to mean   $r\pm\delta$ where $\delta$ is a small positive constant.

We write  $A\lesssim B$ when 
 $A\leq CB$, for some universal constant $C$,  we write $A\gtrsim B$ if $B\lesssim A$. We write $A\sim B$ if $A\lesssim B$ and $B\lesssim A$. As usual, $A\lesssim_{\sigma} B$ means that $A\leq C_{\sigma} B$, where $C_\sigma$ is a constant depending on $\sigma$..

We use $\Lambda$ to denote the operator $\frac{d}{2}+y\nabla$ on $H^{1}(\RRR^{d})$. 
We  use the notation
\begin{equation}
\epsilon_{1}:=\Re \epsilon, \epsilon_{2}:=\Im \epsilon, \text{ i.e. } \epsilon=\epsilon_{1}+i\epsilon_{2},
\end{equation}
where $\Re$ is the real part and $\Im$ is the imaginary part.
 
We  use usual functional spaces $L^{p}$, $C^{1},...,C^{k}$ and $C^{\infty}$, we will also use Sobolev space $H^{s}$, $s\in \RRR$.\\
If not explicitly pointed out, $L^{p}$ means $L^{p}(\RRR^{d})$, so for the  other spaces.  We also use $L_{t}^{q}L_{x}^{p}$ to denote $L^{q}(\RRR; L^{p}(\RRR^{d}))$.  When a certain function is only defined on $I\times \RRR^{d}$, we  also use the notation $L^{q}(I;L^{p}(\RRR^{d}))$. Sometimes we use $\|f\|_{p}$ to denote $\|f\|_{L^{p}}$.

We use $(,)$ to denote the usual $ L^{2}$ (complex) inner product.

Finally, for a solution $u(t,x)$, we use $(T^{-}(u),T^{+}(u))$ to denote its lifespan.

\subsection{A quick review of Merle and Rapha\"el's work and heuristics for the localization of log-log blow up}\label{mrreview}
Let us quickly review the work of Merle and Rapha\"el and highlight the bootstrap structure related to it.
At the starting point of their series of work, in \cite{merle2005blow}, they consider a solution $u$ to \eqref{nls} with initial data $u_{0}\in H^{1}$ satisfies \eqref{supercritical}, with zero momentum and strictly negative energy. They rely on the following variational argument:
\begin{lem}[Lemma 1 in \cite{merle2005blow}]\label{lemvar}
For an arbitrary function $f\in H^{1}$, with energy  $E(f)\leq 0$, if also $f$ satisfies \eqref{supercritical}, then  one can find parameters $\lambda_{0}\in \RRR^{+},x_{0}\in \RRR^{d},\gamma_{0}\in \RRR$ and $\epsilon \in H^{1}$, such that
\begin{equation}
e^{i\gamma_{0}}\lambda_{0}^{\frac{d}{2}}f(\lambda_{0} x+x_{0})=Q+\epsilon,
\end{equation}
and
\begin{equation}
\|\epsilon\|_{H^{1}}\leq \delta(\alpha).
\end{equation}
\end{lem}
This lemma implies that  for the special solution $u(t)$ to \eqref{nls} considered by Merle and Rapha\"el,  one has the geometric decomposition
\begin{equation}\label{var1}
u(t)=\frac{1}{\lambda(t)^{\frac{d}{2}}}\left(Q+\epsilon(t)\right)\left(\frac{x-x(t)}{\lambda(t)}\right)e^{-i\gamma(t)},
\end{equation}
\begin{equation}\label{var2}
\|\epsilon(t)\|_{H^{1}}\leq \delta(\alpha).
\end{equation}
Note that one has some freedom in choosing  the three parameters $\lambda(t), x(t)$ and $\gamma(t)$. Because of this freedom   one can further use  the modulation theory to derive the next lemma.
\begin{lem}[Lemma 2 in \cite{merle2005blow}]\label{firstapprox}
Let $u(t)$ be the solution to \eqref{nls} with initial data $u_{0}$, which has zero momentum, strictly negative energy and let $u_{0}$ satisfy \eqref{supercritical}. Then within the lifespan of $u(t)$, there are three unique parameters $\lambda(t), x(t), \gamma(t)$ and $\epsilon=\epsilon_{1}+i\epsilon_{2}\in H^{1}$, such that
\begin{eqnarray}
u(t)=\frac{1}{\lambda(t)^{\frac{d}{2}}}(Q+\epsilon(t))\left(\frac{x-x(t)}{\lambda(t)}\right)e^{-i\gamma(t)},\\
\|\epsilon(t)\|_{H^{1}}\leq \delta(\alpha),\\
(\epsilon_{1},\Lambda Q)=(\epsilon_{1},yQ)=(\epsilon_{2},\Lambda^{2} Q)=0.
\end{eqnarray}
\end{lem}
Now the study of \eqref{nls} is transferred to the study of the evolution of system $\{\epsilon(t), x(t),\lambda(t),\gamma(t)\}$. We remark here that the blow up rate is determined by the parameter $\lambda(t)$.

 In this setting, Merle and Rapha\"el are using the ground state $Q$ to approximate the solution $u(t)$ (up to space translation, scaling and  phase transformation). It turns out that sharper results can be obtained by using $\qbb$, a modification of $Q$, \cite{merle2003sharp}, \cite{merle2006sharp}.
  Let's give a brief description of $\qbb$, see Proposition 1 in \cite{merle2006sharp}  for details.
  
   Let $b\in \RRR,\eta\in \RRR_{+}$ be small enough, $\eta$ is fixed. Let us define 
\begin{equation}\label{tempdefn}   
   R_{b}:=\frac{2}{|b|}\sqrt{1-\eta}, \quad R_{b}^{-}:=\sqrt{1-\eta}R_{b},
\end{equation}   
   
   and  let $\phi_{b}$ be a smooth cut-off function which equals 1 on  $|x|\leq R_{b}^{-}$ and vanishes for $|x|\geq R_{b}$. Then the modified profile $\qbb:=Q_{b}\phi_{b}$, where $Q_{b}$ solves the  equation
\begin{equation}
\begin{cases}
\Delta Q_{b}-Q_{b}+ib\Lambda Q_{b}+|Q_{b}|^{4/d}Q_{b}=0,\\
P_{b}\equiv Q_{b}e^{i\frac{b|y|^{2}}{4}}>0 \text{ in } B_{R_{b}},
Q_{b}(0)\in (Q(0)-\epsilon^{*}(\eta),Q(0)+\epsilon^{*}(\eta)), Q_{b}(R_{b})=0.
\end{cases}
\end{equation} 
Here we also define $\Psi_{b}$ 
\begin{equation}\label{qbbalmostselsimliar}
\Psi_{b}=-\Delta \qbb+\qbb-ib\Lambda\qbb-\qbb|\qbb|^{4/d},
\end{equation}
that will be used later.

We now  list  some useful estimates for $\qbb$:
\begin{enumerate}
\item $\qbb$ is uniformly close to $Q$ in the sense:
\begin{equation}\label{convergence}
\|e^{(1-\eta)\frac{\theta(|b||y|)}{|b|}}(\qbb-Q)\|_{C^{3}}+\|e^{(1-\eta)\frac{\theta(|b||y|)}{|b|}}(\frac{\partial}{\partial_{b}}\qbb+i\frac{|y|^{2}}{4}Q)\|\xrightarrow{b\rightarrow 0} 0,
\end{equation}
where 
\begin{equation}
\theta(r)=1_{0\leq r\leq 2}\int_{0}^{r}\sqrt{1-\frac{z^{2}}{4}}dz+1_{r>2}\frac{\theta(2)}{2}r.
\end{equation}
\item $\qbb$ is supported in $|y|\lesssim \frac{1}{|b|}$.
\item $\qbb$ has strictly super critical mass:
\begin{equation}
0<\frac{d}{d|b|^{2}}\|Q_{b}\|^{2}<\infty,
\end{equation}
i.e. $\|Q_{b}\|^{2}-\|Q\|_{2}^{2}\sim b^{2}$.
\item $\qbb$ is uniformly bound in $H^{s}, s\in R$ for all $b$ small enough. (Recall we only consider d=1,2)
\begin{equation}\label{uniform}
\|\qbb\|_{H^{s}}\lesssim_{s} 1.
\end{equation}
\end{enumerate}
\begin{rem}
$\theta(2)=\frac{\pi}{2}.$
\end{rem}
\begin{rem}
Estimate \eqref{uniform} is implied by estimate \eqref{convergence} when $s\leq 3$. However,  Merle and Rapha\"el  consider the $C^{3}$ rather than the general $C^{k}$ convergence in \eqref{convergence}  only due to the fact the nonlinearity $|Q|^{\frac{4}{d}}Q$ itself is not smooth enough when $d\geq 3$.  Thus for d=1,2, since the nonlinearity is algebraic, \eqref{uniform} holds for all $s$. And indeed it is not hard to directly use standard elliptic estimates to prove  \eqref{uniform} for $s\geq 3$ once we know this holds for $s\leq 3$. This fact is already implicitly used in \cite{raphael2009standing} for $d=1$.
\end{rem}
\begin{rem}
Later in this work many terms will involve $C\eta$ but we  will be able to fix $\eta$, such that $C\eta$ is as small as we want.
\end{rem}
\begin{rem}
Note that since $Q$ decays exponentially and $\qbb$ is uniformly close to Q, it is standard that  for given $N\in \mathbb N$, terms of the form $(f,\qbb), (\nabla^{N} f, \qbb),(y^{N}f, \qbb) $ are  controlled by $\int (|\nabla f|^{2}+|f|^{2}e^{-|y|})^{\frac{1}{2}}.$ This is widely used in \cite{merle2003sharp},\cite{merle2006sharp}.
\end{rem}
With  $\qbb$, Merle and Rapha\"el modify the lemma \ref{firstapprox} to the following:
\begin{lem}[Lemma 2 in \cite{merle2003sharp}]\label{secondapprox}
Let u(t) be the solution to \eqref{nls} with initial data $u_{0}$, which has zero momentum, strictly negative energy and  satisfies  \eqref{supercritical}. Then within the lifespan of the $ u(t)$, there are unique parameter $\{b(t), \lambda(t), x(t), \gamma(t)\}\in \RRR\times \RRR^{+}\times \RRR^{d} \times \RRR$ such that
\begin{eqnarray}
u(t,x)=\frac{1}{\lambda(t)^{\frac{1}{2}}}(\qbb+\epsilon)\left(\frac{x-x(t)}{\lambda(t)}\right)e^{-i\gamma(t)},\\  \label{geometric}
\|\epsilon\|_{H^{1}}+|b|\leq \delta(\alpha), \label{var3}
\end{eqnarray}
\begin{equation}\label{sharpo1}
(\epsilon_{1},|y|^{2}\Sigma_{b})+(\epsilon_{2}),|y|^{2}\Theta_{b(t)})=0,
\end{equation}
\begin{equation}\label{sharpo2}
(\epsilon_{1},y\Sigma_{b})+(\epsilon_{2}, y\Theta_{b})=0,
\end{equation}
\begin{equation}\label{sharpo3}
-(\epsilon_{1},\Lambda^{2}\Theta_{b})+(\epsilon_{2},\Lambda^{2}\Sigma_{b})=0,
\end{equation}
\begin{equation}\label{sharpo4}
-(\epsilon_{1},\Lambda\Theta_{b})+(\epsilon_{2}, \Lambda\Sigma_{b})=0.
\end{equation}
Here $\qbb:=\Sigma_{b}+i\Theta_{b}$.
\end{lem}
We note that \eqref{var3} a priori assures that the whole analysis is of perturbative  nature.
Again, the study of \eqref{nls} is transferred to the study of the system $\{\epsilon(t), \lambda(t),\gamma(t),b(t),x(t)\}$. To analyze  this system  it is essential that one considers the slowly varying time variable $s$ rather than the $t$:
\begin{equation}
\frac{dt}{ds}=\frac{1}{\lambda^{2}},
\end{equation}
Note that this change of variable changes the lifespan of u (in $t$ variable) to  the whole $\RRR$, (in $s$ variable), no matter if the original solution $u$ blows up in finite time or not.

Now $u$ satisfying \eqref{nls} is equivalent to $\epsilon_{1}, \epsilon_{2}, b(s),\lambda(s), x(s),\gamma(s)$ satisfying the system\footnote{Here we slightly abuse  notation.  For example, $x(s)$ actually means $x(t(s))$.}:
 \begin{equation}\label{sys1}
 \begin{aligned}
 &b_{s}\frac{\partial \Sigma}{\partial b}+\partial_{s}\epsilon_{1}-M_{-}(\epsilon)+b\Lambda\epsilon_{1}=\\
 &(\frac{\lambda_{s}}{\lambda}+b)\Lambda\Sigma+\tilde{\gamma}_{s}\Theta+\frac{x_{s}}{\lambda}\nabla\Sigma+(\frac{\lambda_{s}}{\lambda}+b)(\Lambda \epsilon_{1})
 +\tilde{\gamma_{s}}\epsilon_{2}+\frac{x_{s}}{\lambda}(\nabla \epsilon_{1})+\Im \Psi_{b}-R_{2}(\epsilon)
 \end{aligned}
 \end{equation}
 \begin{equation}\label{sys2}
 \begin{aligned}
 &b_{s}\frac{\partial{\Theta}}{\partial b}+\partial_{s}\epsilon_{2}+M_{+}(\epsilon)+b\Lambda\epsilon_{2}=\\
 &(\frac{\lambda_{s}}{\lambda}+b)\Lambda\Theta-\tilde{\gamma}_{s}\Sigma+\frac{x_{s}}{\lambda}\nabla\Theta+(\frac{\lambda_{s}}{\lambda}+b)\Lambda \epsilon_{2}
 -\tilde{\gamma}_{s}\epsilon_{1}+\frac{x_{s}}{\lambda}(\nabla \epsilon_{2})-\Re \Psi_{b}+R_{1}(\epsilon).
 \end{aligned}
 \end{equation}
Here \footnote{The evolution of $\gamma$ is of course equivalent to the evolution of  $\tilde{\gamma}$.} we have $\tilde{\gamma}(s)=-s-\gamma(s)$, $\qbb=\Sigma_{b}+i\Theta_{b}$ and $M=(M_{+},M_{-})$ is the linearized operator near the profile $\qbb$ and $R_{1}$, $R_{2}$ are  nonlinear terms. Interested readers may consult (2.31), (2.32) in \cite{merle2006sharp} for more details. 
 
Now if one plugs in the four orthogonality condition \eqref{sharpo1},\eqref{sharpo2},\eqref{sharpo3}, \eqref{sharpo4}, one can obtain a system for the parameters   $\{\lambda,x,\tilde{\gamma},b\}$, i.e. four ordinary equations involving  $\{\lambda(s),x(s),\tilde{\gamma}(s),b(s)\}$:\footnote{ Note $\frac{d}{ds}=\lambda^{2}\frac{d}{dt}$.}
\begin{eqnarray}\label{odesharpo1}
&&\frac{d}{ds}\{(\epsilon_{1}(t),|y|^{2}\Sigma_{b(t)})+(\epsilon_{2}(t),|y|^{2}\Theta_{b(t)})\}=0,\\
\label{odesharpo2}
&&\frac{d}{ds}\{(\epsilon_{1}(t),y\Sigma_{b(t)})+(\epsilon_{2}(t), y\Theta_{b(t)})\}=0,\\
\label{odesharpo3}
&&\frac{d}{ds}\{-(\epsilon_{1}(t),\Lambda^{2}\Theta_{b(t)})+(\epsilon_{2}(t),\Lambda^{2}\Sigma_{b(t)})\}=0,\\
\label{odesharpo4}
&&\frac{d}{ds}\{-(\epsilon_{1}(t),\Lambda\Theta_{b(t)})+(\epsilon_{2}(t), \Lambda\Sigma_{b(t)})\}=0.
\end{eqnarray} 
To write down the above ODE system explicitly it requires elementary but involved  algebraic computation (see (71), (72),(73), 
(74) in \cite{raphael2005stability}). A more compact way of writing   $\eqref{odesharpo1}, \eqref{odesharpo2}, \eqref{odesharpo3}, \eqref{odesharpo4}$ is 
\begin{equation}\label{modu}
(b_{s},\lambda_{s},x_{s}, \gamma_{s})=F(b_{s},\lambda_{s},x_{s},\gamma_{s},\epsilon_{1},\epsilon_{2}),
\end{equation}
which justifies the name of ODE system. We call \eqref{odesharpo1},\eqref{odesharpo2},\eqref{odesharpo3}, \eqref{odesharpo4} modulational ODE.

Now assume that  all conclusions in  Lemma \ref{secondapprox} hold. Then by applying them into the modulational ODE, one obtains the so-called modulational estimates in the following lemma.\\

\begin{lem}[Lemma 5 in \cite{merle2006sharp}]\label{lemmoduestimateinitial}
Let the assumption of Lemma \ref{secondapprox} hold, and let \eqref{odesharpo1},\eqref{odesharpo2},\eqref{odesharpo3},\eqref{odesharpo4} hold, then
\begin{equation}\label{modu1}
|\frac{\lambda_{s}}{\lambda}+b|+|b_{s}|\leq C\left(\int |\nabla\epsilon|^{2}+\int |\epsilon|^{2}e^{-|y|}\right)+\Gamma_{b}^{1-C\eta}+C\lambda^{2}|E_{0}|,\\
\end{equation}
\begin{equation}\label{modu2}
\begin{aligned}
|\tilde{\gamma}_{s}-\frac{1}{\|\Lambda Q\|_{2}^{2}}(\epsilon_{1},L_{+}\Lambda^{2}Q)|+|\frac{x_{s}}{\lambda}|\leq
&\delta(\alpha)(\int |\nabla \epsilon|^{2}e^{-2(1-\eta)\frac{\theta(b|y|)}{|b|}}+\int|\epsilon|^{2}e^{-|y|})^{\frac{1}{2}}\\
&+C\int |\nabla\epsilon|^{2}+\Gamma_{b}^{1-C\eta}+C\lambda^{2}|E_{0}|.
\end{aligned}
\end{equation}
\end{lem}
The $\Gamma_{b}$ term will naturally appear in  the definition of the linear radiation term $\tilde{\zeta_{b}}$, which we will discuss later. However, most of the time one only needs to know that
\begin{equation}\label{gamma}
e^{-(1+C\eta)\frac{\pi}{|b|}}\leq \Gamma_{b}\leq e^{-(1-C\eta)\frac{\pi}{|b|}}.
\end{equation}
We point out that \eqref{gamma} is (2.17) in \cite{merle2006sharp}.\par
By applying the conservation laws (Energy and Momentum), one can obtain two more crucial estimates in the following lemma.
\begin{lem}[Lemma 5 in \cite{merle2006sharp}]\label{lemenergyinitial}
The following two estimates hold:
\begin{eqnarray}\label{energy1}
|2(\epsilon_{1},\Sigma)+2(\epsilon_{2},\Theta)|\leq C(\int |\nabla\epsilon|^{2}+\int |\epsilon|^{2}e^{-|y|})+\Gamma_{b}^{1-C\eta}+C\lambda^{2}|E_{0}|,\\
\label{energy2}
|(\epsilon_{2},\nabla \Sigma)|\leq C\delta(\alpha)(\int |\nabla\epsilon|^{2}+\int |\epsilon|^{2}e^{-|y|})^{\frac{1}{2}}.
\end{eqnarray}
\end{lem}

To derive  the blow up rate for the blow up solution $u$, one performs  the following three steps:
\begin{enumerate}
\item Based on \eqref{modu1},  explore the fact that $b\sim -\frac{\lambda_{s}}{\lambda}$, and then transfer the evolution of $\lambda$, (which determines the blow up rate) to the evolution of $b$.
\item Obtain a  lower bound for $b_{s}$, which gives the upper bound for the blow up rate, \cite{merle2003sharp}.
\item Obtain  an upper bound for  $b_{s}$, which gives the lower bound for the blow up rate, \cite{merle2006sharp}.
\end{enumerate}
The lower bound of $b_{s}$ is given by the following lemma:
\begin{lem}[Proposition 2 in \cite{merle2006sharp}]
Let the results of Lemma \ref{secondapprox} hold, let \eqref{odesharpo1}, \eqref{odesharpo2}, \eqref{odesharpo3}, \eqref{odesharpo4}  hold, let \eqref{modu1}, \eqref{modu2}, \eqref{energy1},\eqref{energy2} hold, then one has the estimate
\begin{equation}\label{lv1}
b_{s}\geq \delta_{0}(\int |\nabla \epsilon|^{2}+\int |\epsilon|^{2}e^{-|y|})-C\lambda^{2}E_{0}-\Gamma_{b}^{1-C\eta}.
\end{equation}
\end{lem}
The inequality \eqref{lv1} is called {local virial estimate}, it is one of the key  estimates in the work of Merle and Rapha\"el.
The lower bound of $b_{s}$ involves the construction of a certain Lypounouv functional. For this construction one needs to introduce a certain tail term $\zeta_{b}$ or more precisely its cut-off version $\tilde{\zeta}_{b}$, \cite{merle2004universality}, \cite{merle2006sharp}.

Let us  first quickly describe $\zeta_{b}$, one may refer to Lemma 2 in \cite{merle2006sharp} for more details.
Let $b, \eta, R_{b}, R_{b}^{-}, \phi_{b}$ be  as in  \eqref{tempdefn}.  Let $\zeta_{b}$ be the unique radial solution to
\begin{equation}
\begin{cases}
\Delta \zeta_{b}-\zeta_{b}+ib\Lambda\zeta_{b}=\Psi_{b},\\
\int |\nabla \zeta_{b}|^{2}<\infty.
\end{cases}
\end{equation}
Here $\Psi_{b}$ is defined in \eqref{qbbalmostselsimliar}.
Note, as mentioned previously, that the $\Gamma_{b}$ term appears naturally when one construct $\tilde{\zeta}_{b}$, see (2.17) in \cite{merle2006sharp}.\\
What Merle and Rapha\"el actually use in \cite{merle2006sharp} is the cut-off version\footnote{The main reason for the introduction of this cut-off is that  $\zeta_{b}$ itself is not in $L^{2}$.} of $\zeta_{b}$, that here we denote  by $\tilde{\zeta_{b}}$. See their discussion before formula (3.4) in \cite{merle2006sharp}. 
Since later we will use  $\tilde{\zeta}_{b}$ in several places, we write here the precise definition.
Let $A=A(b)=e^{a\frac{\pi}{|b|}}$, where $a$ is some universal small constant, we  let $\chi_{A}$ be a smooth cut-off function that vanishes outside $|x|\geq A$. Then  one defines $\tilde{\zeta}_{b}:=\zeta_{b}\chi_{A}.$ Clearly $\tilde{\zeta}_{b}$ is supported in $\{|y|\lesssim A\}$. In the rest of this paper, the notation $A$ means $A(b)=e^{a\frac{\pi}{|b|}}$ and  note that $A\gg \frac{1}{|b|}$.

  The tail term $\tilde{\zeta}_{b}$ is introduced to improve the local virial inequality  \eqref{lv1}. Essentially, one wants to change the term $-\Gamma_{b}^{1-C\eta}$ in \eqref{lv1} to $c\Gamma_{b}$. Let
\begin{equation}
f_{1}(s):=\frac{b}{4}\|y\qbb\|_{2}^{2}+\frac{1}{2}\Im (\int y\dzeta\bar{\zzeta})+(\epsilon_{2},\Lambda(\zzeta_{re}))-(\epsilon_{1},\Lambda\zzeta_{im}).
\end{equation}
Then  (this is highly nontrivial, and is one of the key point in \cite{merle2006sharp})
\begin{equation}
\{f_{1}(s)\}_{s}\geq \delta_{1}(\int |\nabla \tilde{\epsilon}|^{2}+\int |\tilde{\epsilon}|^{2}e^{-|y|}+C\Gamma_{b}-C\lambda^{2}E_{0}-\frac{1}{\delta_{1}}\int_{A\leq |x|\leq 2A}|\epsilon|^{2}.
\end{equation}
Here $\tilde\epsilon=\epsilon-\tilde{\zeta}_{b}$.

With this, one constructs the Lyapounov functional $\JJJ$,  \cite{merle2006sharp} that we write explicitly as
\begin{equation}\label{lyp}
\begin{aligned}
\JJJ(s)=(\int |\qbb|^{2}-\int |Q|^{2})+2(\epsilon_{1},\Sigma)+2(\epsilon_{2},\Theta)+\int (1-\phi_{A})|\epsilon|^{2}\\-\frac{\delta_{1}}{800}\left(
b\tilde{f}_{1}(b)-\int_{0}^{b}\tilde{f_{1}}(v)dv+b\{(\epsilon_{2},(\zzeta_{re})_{1})-(\epsilon_{1},(\zzeta_{im})_{1})\}
\right),
\end{aligned}
\end{equation}
where
\begin{equation}
\tilde{f}_{1}(b)=\frac{b}{4}\|y\qbb\|_{2}^{2}+\frac{1}{2}\Im (\int y\dzeta \bar{\zzeta}).
\end{equation}
One has the following inequality:
\begin{equation}\label{criticalmono}
\frac{d}{ds}J(s)\leq -Cb\left(\Gamma_{b}+\int |\nabla \tilde{{\epsilon}}|^{2}+\int |\tilde{\epsilon}|^{2}e^{-|y|}+\int_{A}^{2A}|\epsilon|^{2}-\lambda^{2}E_{0}\right)+C\frac{\lambda^{2}}{|b|^{2}}E_{0}.
\end{equation}
Inequality \eqref{criticalmono}  finally leads to the lower bound of $b_{s}$.\\
  In   \cite{planchon2007existence} and other related works, one can see that the analysis of the log-log blow up solutions can be decomposed into two stages:
\begin{enumerate}
\item At a certain time $t_{0}$, the initial data evolves into some well prepared data.
\item The well-prepared data admits a  suitable bootstrap structure, and analysis can be significantly simplified.
\end{enumerate}
One can show for the solution $u(t)$ considered by Merle and Rapha\"el, that  there exists some $t_{0}$ such that $u(t_{0})$ satisfies the following:\\
\begin{equation}\label{firstdescrbitionofdata}
u(t_{0},x)=\frac{1}{\lambda_{0}^{\frac{d}{2}}}(\tilde{Q}_{b_{0}}+\epsilon_{0})(\frac{x-x_{0}}{\lambda_{0}})e^{-i\gamma_{0}}.
\end{equation}
Also $u(t_{0},x)$ satisfies the following:
\begin{enumerate}
\item orthogonality conditions:  \eqref{sharpo1}, \eqref{sharpo2}, \eqref{sharpo3}, \eqref{sharpo4},
\item the sign condition of $b$:
\begin{equation}\label{signimp}
b_{0}:=b(t_{0})>0,
\end{equation}
\item closeness to $Q$
\begin{equation}\label{oriclosetoq}
\|\epsilon_{0}\|_{H^{1}}+b_{0}<\alpha,
\end{equation}
\item smallness condition of the error $\epsilon_{0}$:
\begin{equation}\label{orismallcondition}
\int \lvert\nabla \epsilon_{0}\rvert^{2}+\lvert \epsilon_{0}\rvert^{2}e^{-|y|}<\Gamma_{b_{0}}^{\frac{6}{7}},
\end{equation}
\item renormalized energy/momentum control\footnote{In this case, this condition is actually implied by the log-log regime condition below, we still keep it to make the notation consistent.}:
\begin{equation}\label{orirenorm}
\lambda_{0}^{2}|E_{0}|+\lambda_{0}|P_{0}|<\Gamma_{b_{0}}^{100},
\end{equation}
\item log-log regime
\begin{equation}\label{oriloglog}
e^{-e^{\frac{2\pi}{b_{0}}}}<\lambda_{0}<e^{-e^{\frac{\pi}{2b_{0}}}}.
\end{equation}
\end{enumerate}
Without loss of generality (by translation in time), we can assume $t_{0}=0$.
Now let us focus on the initial data of the  form $u(t_{0})$, which from now on we denote with $u_0$.
It turns out that  the evolution of the data after $t_{0} $($=0$) is described by the following lemma.
\begin{lem}\label{originalboot}
Assume $u$ solves \eqref{nls} with initial data $u_{0}$ as \eqref{firstdescrbitionofdata}. For all $T<T^{+}(u), $ the following bootstrap argument holds:\\
Let the rescaled time $s$ be defines as $s=\int_{0}^{t} \frac{d\tau}{\lambda^{2}(\tau)}+s_{0}, s_{0}=e^{\frac{3\pi}{4b_{0}}}$, if one assumes the  bootstrap hypothesis for $t\in [0,T]$,
\begin{equation}\label{bhfirst}
\begin{aligned}
&b(t)>0, \quad b(t)+\|\epsilon(t)\|_{H^{1}}\leq 10\alpha^{1/2},  \quad e^{-e^{\frac{10\pi}{b(t)}}}\leq \lambda(t)\leq e^{-e^{\frac{\pi}{10b(t)}}},\\
&\frac{\pi}{10 \ln s}\leq b(t(s))\leq \frac{10\pi}{\ln s}, \quad \int \|\nabla \epsilon(t)\|^{2}+|\epsilon(t)|^{2}e^{-|y|}\leq \Gamma_{b(t)}^{\frac{3}{4}},\\
&\lambda(t_{2})\leq 3\lambda(t_{1}), \quad \forall T>t_{2}>t_{1}>0, \text{ (almost monotonicity)}, \\
&|x(t)|\leq 1/1000,
\end{aligned}
\end{equation}
then one has the  bootstrap estimate for $t\in [0,T]$:
\begin{equation}
\begin{aligned}\label{befirst}
&b(t)>0, \quad b(t)+\|\epsilon(t)\|_{H^{1}}\leq 5\alpha^{1/2},  \quad e^{-e^{\frac{5\pi}{b(t)}}}\leq \lambda(t)\leq e^{-e^{\frac{\pi}{5b(t)}}},\\
&\frac{\pi}{5 \ln s}\leq b(t(s))\leq \frac{5\pi}{\ln s}, \quad \int \|\nabla \epsilon(t)\|^{2}+|\epsilon(t)|^{2}e^{-|y|}\leq \Gamma_{b(t)}^{\frac{4}{5}},\\
&\lambda(t_{2})\leq 2\lambda(t_{1}), \quad \forall T>t_{2}>t_{1}>0,\text{ (almost monotonicity)},\\
&|x(t)|\leq  1/2000.
\end{aligned}
\end{equation}
\end{lem}
\begin{rem}\label{sharpen}
The special numbers $5$, $4/5$... appearing above, are of course only for technical reason, by sharpening  the initial conditions at $t_{0}$, one can push the bootstrap estimates to 
$$\int |\nabla \epsilon|^{2}+|\epsilon|^{2}e^{-|y|}\leq \Gamma_{b}^{1-c_{1}}, \quad e^{-e^{\frac{(1+c_{1})\pi}{b}}}\leq \lambda(t)\leq e^{-e^{\frac{(1-c_{1})\pi}{b}}}$$ for arbitrary $c_{1}>0$.
\end{rem}
We refer to \cite{raphael2006existence}, \cite{raphael2009standing}, \cite{planchon2007existence} for a proof.  See in particular Proposition 1 in \cite{raphael2009standing}.
 Under the bootstrap regime, the analysis is made easier   since one can simplify \eqref{modu1},\eqref{modu2}, \eqref{energy1},\eqref{energy2},\eqref{lv1},\eqref{criticalmono} following  the observation below:
\begin{equation}\label{feelingofscale}
\text{For any polynomial } P,\quad b>>P(\Gamma_{b}), \Gamma_{b}>>P(\lambda).
\end{equation}
Now,  the first step of the analysis leading to the log-log blow up solution  listed above, that is  $b\sim -\frac{\lambda_{s}}{\lambda}$,  is quite clear since by \eqref{modu1}, one has:
\begin{equation}\label{simple0}
|\frac{\lambda_{s}}{\lambda}+b|\lesssim \Gamma_{b}^{\frac{1}{10}}.
\end{equation}
The local virial inequality \eqref{lv1}, which is used to show the lower bound of $b_{s}$, is simplified further  to 
\begin{equation}\label{simple1}
b_{s}\geq -\Gamma_{b}^{1-C\eta},
\end{equation}
and the Lypounov functional \eqref{criticalmono}, which is used to show the upper bound of $b_{s}$, is simplified to
\begin{equation}\label{simple2}
\frac{d}{ds}\JJJ\leq -Cb\Gamma_{b}.
\end{equation}
Basically, all those terms involving   $\lambda^{2}E_{0}$ can be neglected since $\lambda$ is much smaller than $b$. This is actually one of the key observation in  \cite{planchon2007existence}, \cite{colliander2009rough}.

Recall from the geometric decomposition \eqref{geometric}
$$u(t,x)=\frac{1}{\lambda(t)^{\frac{d}{2}}}(\qbb+\epsilon)\left(\frac{x-x(t)}{\lambda(t)}\right)e^{-i\gamma(t)}.$$
Since $\qbb$ is supported in $|y|\lesssim \frac{2}{|b|}$, and the tail radiation term $\tilde{\zeta_{b}}$ is supported in $|y|\lesssim A\equiv \Gamma_{b}^{-a}$, then   the  analysis in work of Merle and Rapha\"el  is taking place  in $|x-x(t)|\lesssim \lambda(t)\Gamma_{b}^{-a}\ll 1$. The only part which connects the local dynamics and the information outside the potential blow up point $x(t)$ is given by  the conservation laws (energy and momentum).  Thus, if one considers the local momentum and local energy, one could localize the dynamics,  paying the prize that the local energy and momentum are no longer conserved. However, since in the analysis any term appearing with the energy or the momentum is multiplied  by a power of $\lambda$, and  $\lambda$ is so small as mentioned above, ultimately we do not  really need that the energy and momentum are conserved, but rather that they are bounded.
This is the reason we need to investigate whether  the  log-log blow up regime may keep the solution $u(t)$  very smooth away from the   blow up point. This would of course imply that the local momentum and local energy are varying in a bounded (not necessarily small) way, thus finally totally localizing the dynamics and decoupling  what happens in the region near the blow up point and away from it.

\subsection{Strategy and structure of the paper}
All arguments and results are valid for both dimensions $d=1$ and 2. For simplicity, from now on, we work on $d=2.$
The main result in Theorem \ref{thmmain} is proven by describing the dynamics for well-prepared initial data of the form 
\begin{equation}\label{datatemp}u_0(x)=\sum_{j=1}^{m}\frac{1}{\lambda_{j,0}}\tilde{Q}_{b_{j,0}}(\frac{x-x_{j,0}}{\lambda_{j,0}})e^{-\gamma_{j,0}}+\Xi.
\end{equation} We call this data a  ''multi-solitions model''. We sometimes call each ''soliton'' as bubble.

We  show that, under certain conditions,  $m$ bubbles $\frac{1}{\lambda_{j,0}}\tilde{Q}_{b_{j,0}}(\frac{x-x_{j,0}}{\lambda_{j,0}})+\Xi, j=1,\dots, m$ evolve as if they do not  feel the existence of each other, then we use a topological argument to show the existence of a blow up solution which blows up at $m$ prescribed points according to log-log law.  

As mentioned previously, one key element in the proof is to show that solutions generated by the well prepared data in 
\eqref{datatemp} will keep high regularity outside a small neighborhood of $x_{j,0}, \, j=1,\dots,m$.\\
To make the proof more accessible, we will first discuss an easier model (we call it "one soliton model" ), i.e. we show that
 $\frac{1}{\lambda_{0}}\tilde{Q}_{b_{0}}(\frac{x}{\lambda_{0}})+\Xi$   blows up according to the log-log law without propagating  singularity outside  a small neighborhood of the origin, which is essentially a restatement of Theorem \ref{regularity}.
 
We organize the rest  of this paper  as follows:

In Section \ref{os}, we  introduce   the well prepared data for the  "one soliton model" and describe its dynamics in  Lemma \ref{osboot}. 
In Section \ref{ts}, We introduce  the well prepared data for the  "multi-solitons model" and describe its dynamic in  Lemma \ref{bootstrap}.
In Section \ref{proofofos}, we prove Lemma \ref{osboot} for the "one soliton model".
In Section \ref{proofofts}, we prove  Lemma \ref{bootstrap} for the "multi-solitons model".
In Section \ref{proofofmain}, we use topological argument to prove the main theorem \ref{thmmain}.

\section{Description of initial data and dynamic/modification of system: one soliton model }\label{os}
Let us recall  that throughout the paper $\alpha$ will be used to denote a universal small number, though its exact value will be chosen at the very end of work. Also $\delta(\alpha)$ is used to denote small constant which depend on $\alpha$ and $\lim_{\alpha\rightarrow 0}\delta (\alpha)=0$.

\subsection{Description of initial data}\label{osdata}
We start with the "one soliton model". We define initial data $u_0$ in the following form:
\begin{equation}\label{onesolitondata}
u_{0}=\frac{1}{\lambda_{0}}(\tilde{Q}_{b_{0}}+\epsilon_{0})(\frac{x}{\lambda_{0}}),
\end{equation}
which satisfies  all the property of  data described in \eqref{firstdescrbitionofdata}, i.e. orthogonality condition \eqref{sharpo1},\eqref{sharpo2},\eqref{sharpo3},\eqref{sharpo4}, and the bounds \eqref{signimp}, \eqref{oriclosetoq},\eqref{orismallcondition},\eqref{orirenorm},\eqref{oriloglog}. Moreover   we assume that outside the origin  the data is smooth in the sense that for some $K_{2}>1$, we have
\begin{equation}\label{onesolitonoutsidesmooth}
\|u_{0}\|_{H^{K_{2}}(|x|\geq \frac{1}{3})}\lesssim 1,\\
\end{equation}
and $u_{0}$ in $H^{K_{2}}$.
\begin{rem}
For non-integer values $K_{2}$, formula \eqref{onesolitonoutsidesmooth} means that,  there exists a smooth cut-off function  $\chi$ equals $1$ in $|x|\geq \frac{1}{3}$ such that $\|\chi u_{0}\|_{H^{K_{2}}}\lesssim 1$.
\end{rem}
Now let us  restate Theorem \ref{regularity}.
\begin{lem}\label{oslem}
Let $u$ solve \eqref{nls},  with initial data  prepared as in \eqref{onesolitondata}. 
Then 
$$\forall K_{1}<\frac{K_{2}}{2}, \quad \|u(t)\|_{H^{K_{1}}(|x|\geq \frac{1}{2})}\lesssim 1$$  within the lifespan 
of $u$.
\end{lem}
\begin{rem}
One should understand  Theorem \ref{regularity} and Lemma \ref{oslem} in  the following way: if the initial data is smooth, then partial smoothness will be kept during the log-log blow up process.
\end{rem}
\subsection{Modification of system}\label{onesolitionmodisys}
Let us  point out here that Lemma \ref{originalboot} cannot directly be applied to $u$ that solves \eqref{nls} with  the prepared data $u_0$. In fact we cannot even say that  $u$  satisfies the geometric decomposition  in \eqref{geometric} for $t\neq 0$.

Recall that  previously a  geometric decomposition was obtained via  a variational argument, Lemma \ref{lemvar}, (and  a further modulation argument), all relying on a negative energy condition and  on \eqref{supercritical}.
In our  case, for data described as \eqref{onesolitondata} we don't even know if the have the negative energy condition.

This does not really matter in the "one soliton model", since one may modify Lemma \ref{var1} relying on the so-called orbital stability of the soliton and re-establish Lemma \ref{secondapprox} without negative energy condition.

We do not  use this approach here, since we will later deal with  "multi-solitons model" and general orbital stability for multi-solitons is a very hard open problem.

 Let us  consider now a system for  $\{u(t),b(t),\lambda(t),x(t),\gamma(t)\}$. We  define $\epsilon(t)=\epsilon_{1}+i\epsilon_{2}:=u(t)-\frac{1}{\lambda(t)}\tilde{Q}_{b(t)}(\frac{x-x(t)}{\lambda(t)})e^{i-\gamma(t)}$, and we  consider the system as
\begin{equation}\label{modifiedsystem}
\begin{cases}
iu_{t}=-\Delta u-|u|^{2}u,\\
\frac{d}{dt}\{(\epsilon_{1}(t),|y|^{2}\Sigma_{b(t)})+(\epsilon_{2}(t),|y|^{2}\Theta_{b(t)})\}=0,\\
\frac{d}{dt}\{(\epsilon_{1}(t),y\Sigma_{b(t)})+(\epsilon_{2}(t), y\Theta_{b(t)})\}=0,\\
\frac{d}{dt}\{-(\epsilon_{1}(t),\Lambda^{2}\Theta_{b(t)})+(\epsilon_{2}(t),\Lambda^{2}\Sigma_{b(t)})\}=0,\\
\frac{d}{dt}\{-(\epsilon_{1}(t),\Lambda\Theta_{b(t)})+(\epsilon_{2}(t), \Lambda\Sigma_{b(t)})\}=0.\\
\{u(t),b(t),\lambda(t),x(t),\gamma(t)\}\arrowvert_{t=0}=\{u_{0},b_{0},\lambda_{0},0,0\}.
\end{cases}
\end{equation}
The local well posedness of \eqref{modifiedsystem} is straightforward, see Appendix \ref{odelwp} for more details.
Note now the following geometric decomposition automatically hold due to the definition of $\epsilon$.
\begin{equation}\label{tempogeo}
u(t,x)=\frac{1}{\lambda(t)}(\qbb+\epsilon)\left(\frac{x-x(t)}{\lambda(t)}\right)e^{-i\gamma(t)}.
\end{equation}

As pointed out by Merle and Rapha\"el, such system should be studied in rescaled  time variable rather than in  its original time variable.
So we define the re-scaled time $s$ as $\frac{ds}{dt}=\frac{1}{\lambda^{2}(t)}$, and if one rewrites \eqref{modifiedsystem} in rescaled  time variable, $\epsilon$  automatically solves \eqref{sys1}, \eqref{sys2}.   Since the orthogonality condition  \eqref{sharpo1}, \eqref{sharpo2}, \eqref{sharpo3}, \eqref{sharpo4}  hold at $t=0$, by \eqref{modifiedsystem}, they  hold for all $t\in [0,T]$, the life span of \eqref{modifiedsystem}.

What's more, the life span of \eqref{modifiedsystem} is exactly the lifespan of the problem  \eqref{nls} with initial data $u_{0}$. This is not a trivial fact. Indeed \eqref{modifiedsystem} is an NLS coupled with  four ODEs involving  $\{b,\lambda,x,\gamma\}$ and  there is the  possibility that $b$ may  become large, and then this system no longer  makes sense since $\qbb$ is only defined for small $b$. However, as long as \eqref{modifiedsystem} holds, then the bootstrap lemma \ref{originalboot} works, and this will ensure  that $b$ stays bounded and small, and thus the lifespan of \eqref{modifiedsystem} coincides with that of the NLS  problem with initial data $u_{0}$.  There is of course another possibility, that  the coupled ODE breakdown, i.e. $\lambda$ becomes $0$ and  this of course, means  that the NLS  equation blows up.

To summarize, during the study of the NLS \eqref{nls} with initial data $u_{0}$  as in \eqref{onesolitondata}, for any $[0,T]$ in the lifespan of $u(t)$, $u(t)$ satisfies  geometric decomposition:
$$u(t)=\frac{1}{\lambda(t)}(\tilde{Q}_{b(t)+\epsilon(t)})\left(\frac{x-x(t)}{\lambda(t)}\right)e^{-i\gamma(t)},$$ and orthogonality conditions \eqref{sharpo1},\eqref{sharpo2}, \eqref{sharpo3},\eqref{sharpo4}. In particular the bootstrap lemma \ref{originalboot} works for these kinds of data.

\subsection{Description of the dynamic}
Now, we can (equivalently) restate Lemma \ref{oslem} in the  following way.
\begin{lem}\label{osboot}
Consider the system \eqref{modifiedsystem} with initial data $\{u_{0},b_{0},\lambda_{0},0,0\}$ described in subsection \ref{osdata}, then for  any $T<T^{+}(u_{0})$, we have the following bootstrap argument:\\
 Let the rescaled time $s$ be defined as $s=\int_{0}^{t} \frac{d\tau}{\lambda^{2}(\tau)}+s_{0}, s_{0}=e^{\frac{3\pi}{4b_{0}}}$, if one assumes bootstrap hypothesis  \eqref{bhfirst} for $t\in [0,T]$,
then one has bootstrap estimate \eqref{befirst} for $t\in [0,T]$, and  following regularity estimate holds for any fixed $K_{1}<\frac{K_{2}}{2}$,
\begin{equation}\label{K1smooth}
\|u(t)\|_{H^{K_{1}}(|x|\geq \frac{1}{2})}\lesssim 1, t\in [0,T].
\end{equation}
\end{lem}

\section{Description of initial data and dynamic: multi-soliton model}\label{ts}
Now we turn to "multi-solitons model".

We introduce some notations.
Let $\chi_{0,loc}$, $\chi_{1,loc}$ be two smooth cut-off functions such that
\begin{equation}\label{chi0}
\chi_{0,loc}(x)=
\begin{cases}
1, |x|\leq 3/4,\\
0, |x|\geq 1,
\end{cases}
\end{equation}
\begin{equation}
\chi_{1,loc}(x)=
\begin{cases}
1, |x|\leq 2/3,\\
0, |x|\geq 3/4.
\end{cases}
\end{equation}
Let's define the local version of energy and momentum.  For any $f\in H^{1}(\RRR^{2}), x_{0}\in \RRR^{2}$, let
\begin{eqnarray}
&&E_{loc}(f,x_{0}):=\int_{\RRR^{2}}\chi_{0,loc}(x-x_{0})\left(\frac{1}{2}|\nabla f|^{2}-\frac{1}{4}|f|^{4}\right),\\
&&P_{loc}(f,x_{0}):=\int_{\RRR^{2}}\nabla f \bar{f}\chi_{0,loc}(x-x_{0}).
\end{eqnarray}
\subsection{Description of the initial data}\label{datats}
We define an initial data $u_{0}$ of the following form:
\begin{equation}\label{setup}
u_{0}:=\sum_{j=1}^{m}u_{j,0}+\Xi_{0},
\end{equation}
where
\begin{eqnarray}
u_{j,0}=\frac{1}{\lambda_{j,0}}\tilde{Q}_{b_{j,0}}\left(\frac{x-x_{j,0}}{\lambda_{j,0}}\right)e^{-i\gamma_{j,0}}, j=1,\dots, m\\
\Xi_{0}\in H^{1}(\RRR^{2}),\end{eqnarray}
and the properties of $\Xi_{0}$ are implicitly encoded below.

We let 
\begin{eqnarray}
\begin{aligned}
&\epsilon^{j}_{0}:=\lambda_{j,0}\Xi_{0}(\lambda_{j,0}x+x_{j,0})e^{i\gamma_{j,0}},\\\label{notation1}
&\epsilon^{j}_{0}\equiv \epsilon^{j}_{1,0}+i\epsilon^{j}_{2,0}, j=1,\dots, m.
\end{aligned}
\end{eqnarray}
We require the orthogonality condition:
\begin{eqnarray}\label{initialo1}
&&(\epsilon_{1,0}^{j}, |y|^{2}\Sigma_{b_{j,0}})+(\epsilon^{j}_{2,0},|y|^{2}\Theta_{b_{j,0}})=0, j=1,\dots, m\\\label{initialo2}
&&(\epsilon^{j}_{1,0},y\Sigma_{b_{1,0}})+(\epsilon^{j}_{2,0},y\Theta_{b_{j,0}})=0, j=1,\dots, m\\\label{initialo3}
&&-(\epsilon_{1,0},\Lambda^{2}\Theta_{b_{j,0}})+(\epsilon^{j}_{2,0},\Lambda^{2}\Sigma_{b_{j,0}})=0, j=1,\dots, m\\\label{initialo4}
&&-(\epsilon_{1,0}^{j}, \Lambda \Theta_{b_{j,0}})+(\epsilon_{2,0}^{j}, \Lambda \Sigma_{b_{j,0}})=0, j=1,\dots, m,
\end{eqnarray}
where $\qbb\equiv\Sigma_{b}+i\Theta_{b}$.\\
We further require the following:
\begin{enumerate}
\item Non-interaction of any  two different solitons:
\begin{equation}\label{initalseparate}
 |x_{j,0}-x_{j',0}|\geq 10, \forall j\neq j', 
\end{equation}
\item Sign condition and smallness  condition of $b_{j,0}$:
\begin{equation}\label{initailsigh}
\alpha>b_{j,0}>0, j=1,\dots, m.
\end{equation}
\item Smallness condition:
\begin{equation}\label{initalsmallness}
\sum_{j=1}^{m}\lambda_{j,0}+\sum_{j=1}^{m}\|\epsilon^{j}\|_{H^{1}}\leq \alpha,
\end{equation}
\item Log-log regime condition:
\begin{equation}\label{initalloglog}
e^{-e^{\frac{2\pi}{b_{j,0}}}}<\lambda_{j,0}<e^{-e^{\frac{\pi}{2b_{j,0}}}}, j=1,\dots, m.
\end{equation}
\item Smallness of the local error
\begin{equation}\label{initalsmallnessoflocalerror}
\int_{|x|\leq \frac{10}{\lambda_{j,0}}}|\nabla \epsilon^{j}_{0}|^{2}+|\epsilon_{0}^{j}|^{2}e^{-|y|}\leq \Gamma_{b_{j,0}}^{\frac{6}{7}}, j=1,\dots, m.
\end{equation}
\item Tameness of local energy and local momentum
\begin{equation}\label{initialtamenessofenergy}
\lambda_{j}^{2}|E_{loc}(u_{0},x_{j,0})|\leq \Gamma_{b_{j}}^{1000}, j=1, \dots, m.
\end{equation}
\begin{equation}\label{initaltamenessofmomentum}
\lambda_{j}|E_{loc}(u_{0},x_{j,0})|\leq \Gamma_{b_{j}}^{1000}, j=1,\dots, m.
\end{equation}
\item Smoothness outside the singularity, here we fix a large $N_{2}$, which would be chosen later.
\begin{equation}\label{initialoutsmooth1}
|u_{0}|_{H^{N_{2}}(min_{1\leq j\leq m}\{|x-x_{j,0}|\}\geq \frac{1}{3})}\leq \alpha.
\end{equation}
\item $u_{0}$ in $H^{N_{2}}$.
\end{enumerate}
\begin{rem}
Such data can be constructed similarly as those  constructed in \cite{raphael2009standing}.
\end{rem}
\subsection{Modification of system}
Now let $u$ be the solution to \eqref{nls} with initial data $u_{0}$ described as  in Subsection \ref{datats}. We are expecting that throughout the lifespan of the evolution, we can find parameters $\{b_{j}(t),\lambda_{j}(t),x_{j}(t),\gamma_{j}(t)\}_{j=1}^{m}$ and a function $\Xi(t,x)$ such that the following  geometric decomposition holds:
\begin{eqnarray}\label{basicsettinggeo}
&&u(t,x)=\sum_{j=1}^{m}\frac{1}{\lambda_{j}(t)}\qbj(\frac{x-x_{j}(t)}{\lambda_{j}(t)})e^{-i\gamma_{j}(t)}+\Xi(t,x),\\
&&|b_{j}(t)|\leq \delta(\alpha), 
\|\epsilon^{j}\|_{H^{1}}\leq \delta(\alpha), 1\leq j\leq m,  \epsilon^{j}=\lambda_{j}\xi(\lambda_{j}x+x_{j})e^{i\gamma_{j}},
\end{eqnarray}
as well as the orthogonality condition:
\begin{eqnarray}\label{sharpoth1}
&&(\epsilon_{1}^{j}, |y|^{2}\Sigma_{b_{j}})+(\epsilon^{j}_{2},|y|^{2}\Theta_{b_{j}})=0, j=1,\dots, m,\\\label{sharpoth2}
&&(\epsilon^{j}_{1},y\Sigma_{b_{j}})+(\epsilon^{j}_{2},y\Theta_{b_{j}})=0, j=1,\dots, m,\\\label{sharpoth3}
&&-(\epsilon_{1},\Lambda^{2}\Theta_{b_{j}})+(\epsilon^{j}_{2},\Lambda^{2}\Sigma_{b_{j}})=0,j=1,\dots, m,\\\label{sharpoth4}
&&-(\epsilon_{1}^{j}, \Lambda \Theta_{b_{j}})+(\epsilon_{2}^{j}, \Lambda \Sigma_{b_{j}})=0, j=1,\dots, m.
\end{eqnarray}
Again $\qbb=\Sigma_{b}+i\Theta_{b}$, $\epsilon^{j}=\epsilon^{j}_{1}+i\epsilon_{2}^{j}$.

Since at this point we do not know that the general multi-solitons orbital stability holds,  we consider  the system for $\{u(t,x),\{b_{j}(t),\lambda_{j}(t),x_{j}(t), \gamma_{j}(t)\}_{j=1}^{m}\}$ as in Subsection \ref{onesolitionmodisys}:
\begin{equation}\label{tsnls}
\begin{cases}
iu_{t}+\Delta u=-|u|^{2}u,\\
\frac{d}{dt}\{(\epsilon_{1}^{j}, |y|^{2}\Sigma_{b_{j}})+(\epsilon^{j}_{2},|y|^{2}\Theta_{b_{j}})\}=0, j=1,\dots, m,\\
\frac{d}{dt}\{(\epsilon^{j}_{1},y\Sigma_{b_{j}})+(\epsilon^{j}_{2},y\Theta_{b_{j}})\}=0, j=1,\dots, m,\\
\frac{d}{dt}\{-(\epsilon_{1},\Lambda^{2}\Theta_{b_{j}})+(\epsilon^{j}_{2},\Lambda^{2}\Sigma_{b_{j}})\}=0, j=1,\dots, m,\\
\frac{d}{dt}\{-(\epsilon_{1}^{j}, \Lambda \Theta_{b_{j}})+(\epsilon_{2}^{j}, \Lambda \Sigma_{b_{j}})\}=0, j=1,\dots, m,\\
u(0,x)=u_{0}, \lambda_{j}(0)=\lambda_{j,0}, b_{j}(0)=b_{j,0}, x_{j}(0)=x_{j,0}, \gamma_{j}(0)=\gamma_{j,0}, j=1,\dots, m,
\end{cases}
\end{equation}
where $\epsilon^{j}, \epsilon_{1}^{j}, \epsilon_{2}^{j}$ is defined by $\Xi$, as previously noted, and $\Xi$ is defined by 
\begin{equation}\label{geodef}
\Xi(t,x)\equiv u(t,x)-\sum_{j=1}^{m}\frac{1}{\lambda_{j}}\qbj(\frac{x-x_{j}}{\lambda_{j}})e^{-i\gamma_{j}(t)}.
\end{equation}
By the orthogonality conditions  \eqref{sharpoth1},\eqref{sharpoth2},\eqref{sharpoth3}, \eqref{sharpoth4} of the initial data,   one has the orthogonality condition within the lifespan of \eqref{tsnls}:
\begin{eqnarray}\label{ot1}
&&(\epsilon_{1}^{j}, |y|^{2}\Sigma_{b_{j}})+(\epsilon^{j}_{2},|y|^{2}\Theta_{b_{j}})=0, j=1,\dots, m,\\\label{ot2}
&&(\epsilon^{j}_{1},y\Sigma_{b_{1}})+(\epsilon^{j}_{2},y\Theta_{b_{j}})=0, j=1,\dots, m,\\\label{ot3}
&&-(\epsilon_{1},\Lambda^{2}\Theta_{b_{j}})+(\epsilon^{j}_{2},\Lambda^{2}\Sigma_{b_{j}})=0, j=1,\dots, m,\\\label{ot4}
&&-(\epsilon_{1}^{j}, \Lambda \Theta_{b_{j}})+(\epsilon_{2}^{j}, \Lambda \Sigma_{b_{j}})=0, j=1,\dots, m.
\end{eqnarray}
 We  discussed previously that the lifespan of  \eqref{modifiedsystem} is the same as the lifespan of $u$. With the same argument one also has that  the lifespan of \eqref{tsnls} is the same as  lifespan of $u$.\\
We perform  one final simplification. Let us define\footnote{Please note $\eps$ and $\epsilon$ are different notations.}:
\begin{equation}
\begin{aligned}
&\Xi^{j}(x):=\Xi(x)\chi_{1,loc}(x-x_{j}), \eps^{j}(t,y):=\lambda_{j}\Xi^{j}(\lambda_{j}y+x_{j})e^{-i\gamma_{j}},\\
&\eps^{j}(t,y)\equiv \eps^{j}_{1}+i\eps^{j}_{2}, j=1,\dots, m.\\
\end{aligned}
\end{equation}
we point out that our analysis will be performed  under the condition $\lambda_{j}\ll b_{j}$, (see Lemma \ref{bootstrap} below). Now, by definition, one has $\eps^{j}=\epsilon^{j}\chi_{1,loc}(\frac{x}{\lambda_{j}})$, and recall $\qbb$ is supported in $|y|\lesssim \frac{1}{b}$. It is not hard to see the following orthogonality condition for $\epj$.
\begin{eqnarray}\label{oth1}
&&(\eps_{1}^{j}, |y|^{2}\Sigma_{b_{j}})+(\eps^{j}_{2},|y|^{2}\Theta_{b_{j}})=0, j=1,\dots, m,\\\label{oth2}
&&(\eps^{j}_{1},y\Sigma_{b_{1}})+(\eps^{j}_{2},y\Theta_{b_{j}})=0, j=1,\dots, m\\\label{oth3}
&&-(\eps_{1},\Lambda^{2}\Theta_{b_{j}})+(\eps^{j}_{2},\Lambda^{2}\Sigma_{b_{j}})=0, j=1,\dots, m,\\\label{oth4}
&&-(\eps_{1}^{j}, \Lambda \Theta_{b_{j}})+(\eps_{2}^{j}, \Lambda \Sigma_{b_{j}})=0, j=1, \dots, m.
\end{eqnarray}
Under the  hypothesis $\lambda_{j}\ll b_{j}$, the system \eqref{tsnls} is exactly:
\begin{equation}\label{tsnls2}
\begin{cases}
iu_{t}+\Delta u=-|u|^{2}u,\\
\frac{d}{dt}\{(\eps_{1}^{j}, |y|^{2}\Sigma_{b_{j}})+(\eps^{j}_{2},|y|^{2}\Theta_{b_{j}})\}=0, j=1,\dots, m,\\
\frac{d}{dt}\{(\eps^{j}_{1},y\Sigma_{b_{1}})+(\eps^{j}_{2},y\Theta_{b_{j}})\}=0, j=1,2, \dots, m,\\
\frac{d}{dt}\{-(\eps_{1}^{j},\Lambda^{2}\Theta_{b_{j}})+(\epsilon^{j}_{2},\Lambda^{2}\Sigma_{b_{j}})\}=0, j=1,\dots, m,\\
\frac{d}{dt}\{-(\eps_{1}^{j}, \Lambda \Theta_{b_{j}})+(\eps_{2}^{j}, \Lambda \Sigma_{b_{j}})\}=0, j=1,\dots, m\\
u(0,x)=u_{0}, \lambda_{j}(0)=\lambda_{j,0}, b_{j}(0)=b_{j,0}, x_{j}(0)=x_{j,0}, \gamma_{j}(0)=\gamma_{j,0}, j=1,\dots, m.
\end{cases}
\end{equation}
\subsection{Description of the dynamic}
We are now ready to present the main lemma which contains a  bootstrap argument used to describe the dynamics of \eqref{tsnls2}, where the initial data 
$\{u_{0},\{b_{j,0},\lambda_{j,0},x_{j,0},\gamma_{j,0}\}_{j=1}^{m}\}$ is described as in Subsection \ref{datats}.

\begin{lem}\label{bootstrap}
Consider the  system \eqref{tsnls2} with initial data described as in Subsection \ref{datats}, (with  the universal constant $\alpha$ is small enough), then
$\forall T<T^{+}(u)$, the following bootstrap argument holds (and all the estimates are independent of $T$):\\
Let $s_{j,0}:=e^{\frac{3\pi}{4b_{j,0}}}$, and define the re-scaled time  $s_{j}:=\int_{0}^{t}\frac{1}{\lambda_{j}^{2}}+s_{j,0}$, $j=1,\dots, m$.
If for  $t\in[0,T)$  the bootstrap hypothesis hold:
\begin{itemize}
\item smallness of $b_{j}$ and $\eps$:
\begin{equation}\label{assumuepertbation}
\sum_{j=1}^{m}|b_{j}|+\sum_{j=1}^{m}\|\eps^{j}\|_{H^{1}}\leq 10\alpha^{1/2},
\end{equation}
\item log-log regime, part I:
\begin{equation}\label{assumeloglog1}
e^{-e^{\frac{10\pi}{b_{j}(t)}}}\leq \lambda_{j}(t)\leq e^{-e^{\frac{\pi}{10b_{j}(t)}}}, j=1,\dots, m,
\end{equation}
\item log-log regime, part II:\\
\begin{equation}\label{assumeloglog2}
\frac{\pi}{10\ln s_{j}}\leq b_{j}\leq \frac{10\pi}{\ln s_{j}}, j=1,\dots, m,
\end{equation}
\item smallness of the local error:
\begin{equation}\label{assumesmallnessoflocalerror}
\int|\nabla \eps^{j}|^{2}+
|\eps^{j}|^{2}e^{-|y|}\leq \Gamma_{b_{j}}^{\frac{3}{4}}, j=1,\dots, m,
\end{equation}
\item almost monotonicity:
\begin{equation}\label{assumealmostmono}
\forall T>t_{2}\geq t_{1}\geq 0, \lambda_{j}(t_{2})\leq 3\lambda_{j}(t_{1}), j=1,\dots, m,
\end{equation}
\item control of translation parameters:
\begin{equation}\label{assumesmalltranslation}
\forall 0<t<T, |x_{j}(t)-x_{j,0}|\leq \frac{1}{1000}, j=1,\dots, m,
\end{equation}
\item local control of conserved quantity:
for all $0<t<T$, $j=1,\dots, m$:
\begin{eqnarray}\label{assumelocalcontrolofconserveredquantity}
&&|E_{loc}(x_{j}(t),u(t))-E_{loc}(x_{j,0}, u_{0})|\leq 1000,\\\label{assumelocalcontrolofmomentum}
&&|P_{loc}(x_{j}(t),u(t))-P_{loc}(x_{j,0}, u_{0})|\leq 1000,
\end{eqnarray}
\item outside-smoothness:
\begin{equation}\label{assumeoutsidesmoothness}
\|u\|_{H^{N_{1}}(\min_{j}\{|x-x_{j,0}|\geq \frac{1}{2}\})}\leq \frac{1}{\max_{j}\{\lambda_{j,0}\}},
\end{equation}
where  $N_{1}$ is some fixed large constant, $N_{1}<\frac{N_{2}}{2}.$
\end{itemize} 
Then for $t\in [0,T)$, the bootstrap estimates  hold:
\begin{itemize}
\item smallness of $b_{j}$ and $\epsilon$:
\begin{equation}\label{estimateperturbation}
\sum_{j=1}^{m} |b_{j}|+\sum_{j=1}^{m}\|\epsilon^{j}\|_{H^{1}}\leq 5\alpha^{1/2},
\end{equation}
\item log-log regime, part I:
\begin{equation}\label{estimateloglog1}
e^{-e^{\frac{5\pi}{b_{j}(t)}}}\leq \lambda(t)\leq e^{-e^{\frac{\pi}{5b_{j}(t)}}}, j=1,\dots, m, 
\end{equation}
\item log-log regime, part II:
\begin{equation}\label{estimateloglog2}
\frac{\pi}{5\ln s_{j}}\leq b_{j}\leq \frac{5\pi}{\ln s_{j}}, j=1,\dots, m,
\end{equation}
\item smallness of the local error:
\begin{equation}\label{estimateofsmalllocalerror}
\int|\nabla \eps^{j}|^{2}+
|\eps^{j}|^{2}e^{-|y|}\leq \Gamma_{b_{j}}^{\frac{4}{5}}, j=1,\dots, m,
\end{equation}
\item almost monotonicity: 
\begin{equation}\label{estimatealmostmono}
\forall T>t_{2}\geq t_{1}\geq 0, \lambda_{j}(t_{2})\leq 2\lambda_{j}(t_{1}), j=1,\dots, m,
\end{equation}
\item control of translation parameter:
\begin{equation}\label{estimatesmalltranslation}
\forall 0<t<T, |x_{j,t}-x_{j,0}|\leq \frac{1}{2000}, j=1,\dots, m,
\end{equation}
\item local control of conserved quantity:
for all $0<t<T$, $j=1,\dots, m$:
\begin{eqnarray}\label{estimatelocalcontrolofconservedquantity}
&&|E_{loc}(x_{j}(t),u(t))-E_{loc}(x_{j,0},u_{0})|\leq 500,\\\label{estimatelocalcontrolofmomentum}
&&|P_{loc}((x_{j}(t), u(t))-P_{loc}(x_{j,0}, u_{0})|\leq 500,\\
\end{eqnarray}
\item outside smoothness:
\begin{equation}\label{estimateoutsidesmoothness}
\|u\|_{H^{N_{1}}(\min_{j}\{|x-x_{j,0}|\geq \frac{1}{2}\})}\lesssim 1,
\end{equation}
\end{itemize}
\end{lem}
\begin{rem} 
According to our notation, \eqref{estimateoutsidesmoothness} means that 
$$\|u\|_{H^{N_{1}}(\min_{j}\{|x-x_{j,0}|\geq \frac{1}{2}\})}\leq C$$ for some universal constant C.
Also by \eqref{initalsmallness} we have  $lim_{\alpha\rightarrow 0}\frac{1}{\max_{j}\{\lambda_{j,0}\}}=\infty$. Thus, when $\alpha$ is small enough, \eqref{estimateoutsidesmoothness} is stronger than \eqref{assumeoutsidesmoothness}, i.e. this is a bootstrap estimate.\
\end{rem}
\begin{rem}\label{finite}
The  Lemma \ref{bootstrap} itself implies that  $T^{+}(u)\leq \delta_{1}(\alpha)$. Indeed, Lemma \ref{bootstrap} implies that the solution blows up in finite time and according to the log-log law. Similarly for the solutions in  Lemma \ref{originalboot} and  Lemma \ref{osboot}. This is a  standard argument, for details see for example \cite{planchon2007existence}. One may also directly look at Subsection \ref{sectiones}, see Remark \ref{finite222}.
\end{rem}
\begin{rem}
Note that Lemma  \ref{bootstrap}  implies that $\min_{j\neq j'}|x_{j}(t)-x_{j'}(t)|>5$ for all $t<T^{+}(u)$.
\end{rem}

\subsection{Further remarks on Lemma \ref{bootstrap}}\label{keyremarks}
Lemma \ref{bootstrap} means that all the bootstrap estimates hold within the lifespan of $u$ generated by the initial data described in Subsection \ref{datats}. In particular for initial data of the form 
$$u(t,x)=\sum_{j=1}^{m}\frac{1}{\lambda_{j}}\tilde{Q}_{b_{j,0}}(\frac{x-x_{j,0}}{\lambda_{j}})e^{-i\gamma_{j,0}}+\Xi_{0}(t,x)$$ with orthogonality condition \eqref{initialo1},\eqref{initialo2},\eqref{initialo3},\eqref{initialo4} and with bounds \eqref{initalseparate},\eqref{initailsigh}...\eqref{initialoutsmooth1},  the associated solution is smooth in the area $\min_{j}{|x-x_{j,0}|}\geq \frac{1}{2}$ (i.e. estimate \eqref{estimateoutsidesmoothness} holds). 

Further more, if one looks\footnote{We note that by the definition of $\chi_{1,loc}$ and the  bootstrap estimate $|x_{j}(t)-x_{j,0}|\leq \frac{1}{1000}$, then $u(t,x)=\chi_{1,loc}(x-x_{j}(t))u(t,x)$ in the region $|x-x_{j,0}|\leq \frac{1}{2}$.} at $\chi_{1,loc}(x-x_{j}(t))u(t,x), j=1,\dots, m$,
then 
\begin{equation}
\chi_{1,loc}(x-x_{j}(t))u(t,x):=\frac{1}{\lambda_{j}}(\tilde{Q}_{b_{j}}+\epj) (\frac{x-x_{j}}{\lambda_{j}})e^{-i\gamma_{j}}, j=1,\dots, m
\end{equation}
and  the bootstrap estimates stated in Lemma \ref{bootstrap} simply means that $\chi(x-x_{j}(t))u(t,x)$ evolves according to the log-log law described in the series of work of Merle and Rapahel  (until at least one soliton blows up.)

\section{Proof of Lemma \ref{osboot}: One Soliton Model}\label{proofofos}
This section is devoted to the proof of Lemma \ref{onesolitonoutsidesmooth}, and we need only to show \eqref{K1smooth} since we already have \eqref{originalboot}.  Fix  $[0,T]$, all the estimates below are independent of the choice of $T$.

\subsection{Setting up}\label{onesolitonsetting up}
Recall that $u$ is the solution to \eqref{nls} with initial data $u_{0}$ as described in  \eqref{onesolitondata}. Recall as discussed in Subsection \ref{onesolitionmodisys}, that we consider the system \eqref{modifiedsystem} for  $\{u(t),b(t),\lambda(t),x(t),\gamma(t)\}$, and the geometric decomposition   \eqref{tempogeo} holds. Also  Lemma \ref{originalboot} holds. We finally point out that  the bootstrap hypothesis \eqref{bhfirst} itself implies
\begin{equation}\label{scalefromblowupspeed}
\int_{0}^{t}\frac{1}{\lambda(\tau)^{\mu}}d\tau\leq 
\begin{cases}
C(\mu) \text{ for } \mu<2,\\
\frac{|\ln \lambda(t)|^{101}}{\lambda(t)^{\mu-2}} \text{ for } \mu\geq 2,
\end{cases}
\end{equation}
for any $t<T$.
See (51) in \cite{raphael2009standing}.

\begin{rem}
Note here that we may also use the  bootstrap estimate \eqref{befirst} to show \eqref{scalefromblowupspeed}, since we know that \eqref{bhfirst} implies \eqref{befirst}. However,  all the arguments in this section only rely on the bootstrap hypothesis \eqref{bhfirst}, this will become important when we work on multi-solitons model.\\
\end{rem}
\subsection{An overview  of the proof}
Recall that for all $t\in [0,T^{+}(u))$, the geometric decomposition holds
$$u(t)=\frac{1}{\lambda(t)}(\qbb+\epsilon)(\frac{x-x(t)}{\lambda(t)})e^{-i\gamma(t)}, $$ and by the bootstrap hypothesis \eqref{bhfirst}, we know $|x(t)|\leq \frac{1}{1000}$ for $t\in[0,T]$.
Now we  fix $\chi_{0}$, such
\begin{equation}
\chi_{0}(x):=
\begin{cases}
0, |x|\leq a_{0}\equiv\frac{1}{500},\\
1, |x|\geq d_{0}\equiv\frac{1}{250}. 
\end{cases}
\end{equation}
 The key to proving that  a log-log blow up solution can propagate some regularity outside its potential blow up point $x(t)$ is the following control:
\begin{equation}\label{origingain}
\forall T<T^{+}(u), \int_{0}^{T} \|\nabla( \chi_{0} u(t))\|_{2}^{2}\lesssim 1.
\end{equation}
This is pointed out in \cite{raphael2009standing}, see formula (63) in \cite{merle2005profiles} for a proof.

We  first use the  I-method to show the rough control,
\begin{lem}\label{lemosroughcontrol}
\settingos, for any $\sigma>0$,
\begin{equation}\label{ineqroughcontrol}
\|u(t)\|_{H^{K_{2}}}\lesssim_{\sigma} (\frac{1}{\lambda(t)})^{K_{2}+\sigma}.
\end{equation}   
\end{lem}
For a proof see Subsection \ref{roughcontrol}.

Next, we introduce a sequence of cut off functions $\{\chi_{l}\}_{l=0}^{L}$, where $L$ is a large but fixed number which will be chosen later:
\begin{equation}
\chi_{l}=
\begin{cases}
0, |x|\leq a_{l},\\
1,|x|\geq b_{l}.
\end{cases}
\end{equation}
Here we require $a_{l}<b_{l}<a_{l-1}$, $b_{L}\leq \frac{1}{2}$.\\
The idea is  to retreat from $\chi_{l-1}u$ to $\chi_{l}u$ for each $l$, showing that $\chi_{l}u$ has higher regularity than $\chi_{l-1}u$. For convenience of notation we set
$
v_{i}:=\chi_{i}u.
$
We  use the crucial control \eqref{origingain}, Strichartz estimates and interpolation techniques to show:
\begin{lem}\label{lemosfirstgain}
For all $0<\nu<1$ and  for $t\in [0,T]$, we have
\begin{equation}
\forall \sigma>0, \|\nabla^{\nu}v_{2}(t)\|_{L^{2}}\lesssim_{\sigma} \frac{1}{\lambda(t)^{\sigma}}.
\end{equation}
\end{lem}
For a proof see  Subsubsection \ref{proofosfirstgain}.

This lemma  gives a better $L^{\infty}$ estimate of $v_{2}$, which of course implies a better $L^{\infty}$ estimate of $v_{i}, i\geq 2$.

We  also show:
\begin{lem}\label{lemosimprovelinfty}
For any $0<\nu<1$ and for $t\in [0,T],$ we have 
\begin{equation}
\forall \sigma>0, \|v_{2}\|_{L^{\infty}}\lesssim_{\sigma} \frac{1}{\lambda(t)^{\sigma}},
\end{equation}
\end{lem}
For a proof 
see Subsubsection \ref{proofimprovelinfty}.

\begin{rem}
This lemma should be understood as an improvement of  the $L^{\infty}$ estimate. Indeed, from the viewpoint of Sobolev embedding, $H^{1+}(\RRR^{2})\hookrightarrow L^{\infty}$, thus the trivial $L^{\infty}$ estimate is $\|v_{2}\|_{L^\infty}\lesssim \|u(t)\|_{L^{\infty}}\lesssim \frac{1}{\lambda(t)^{1+}}$. 
\end{rem}
 Lemma \ref{lemosimprovelinfty} improves Lemma \ref{lemosfirstgain} immediately:
\begin{lem}\label{lemosfirstgainimpr}
For all $0<\nu<1$ and for $t\in [0,T]$, we have 
\begin{equation}
\|\nabla^{\nu}v_{3}(t)\|_{L^{2}}\lesssim 1.
\end{equation}
\end{lem}
For  a proof see 
Subsubsection \ref{prooffirstgainimpr}.

Finally, we have the following lemma (which is analogue to Lemma 4 in \cite{raphael2009standing}) to iterate the gain of regularity for large $K_{2}$.
\begin{lem}\label{lemositeration}
 If the following estimate holds for some $r>0$ and some $i$, $1\leq i\leq L-1$,
\begin{equation}
\|v_{i}(t)\|_{H^{r}}\lesssim 1, t\in [0,T].
\end{equation} 
then there is a gain of regularity on $v_{i+1}$, that is 
\begin{equation}\label{gain}
\forall \tilde{r}<\frac{2(K_{2}-r)}{K_{2}}-1+r, \|v_{i+1}\|_{H^{\tilde{r}}}\lesssim_{\tilde{r}} 1, t\in [0,T].
\end{equation}
\end{lem}
For a proof see 
Subsubsection \ref{proofositeration}.

Now we are ready to end the proof of Lemma \ref{osboot}, i.e. to prove \eqref{K1smooth}.
\begin{proof}[Proof of \eqref{K1smooth} in Lemma \ref{osboot}]
Lemma \ref{lemositeration} is enough for us to obtain the regularity estimate \eqref{K1smooth} for $K_{1}<\frac{K_{2}}{2}$. To see this consider the end point case in \eqref{gain}, i.e. if one consider $r=\tilde{r}=\frac{2(K_{2}-r)}{K_{2}}-1+r$, then one obtains $r=\frac{K_{2}}{2}$. Thus, when $K_{1}<\frac{K_{2}}{2}$, by choosing the iteration time $L$ large, the desired estimate \eqref{K1smooth} follows.
\end{proof}

\subsection{Rough control}\label{roughcontrol}
This subsection is devoted to the proof of Lemma \ref{lemosroughcontrol}, similar type of estimates are derived in \cite{raphael2009standing} by considering some pseudo energy. Here, we rely on I-method. Without loss of generality, we only show:
\begin{equation}\label{goalroughonesoliton}
\|u(T)\|_{H^{K_{2}}}\lesssim_{\sigma} (\frac{1}{\lambda(T)})^{K_{2}+\sigma}.
\end{equation}

\subsubsection{LWP interval}
Let's introduce the so-called LWP interval as in \cite{colliander2009rough}.
To make the argument easier,  we observe  that, under the bootstrap hypothesis \eqref{bhfirst},  we can  show that $\lambda(t)$ is actually strictly decreasing. Indeed,   by \eqref{bhfirst}, $b>0$, $ \lambda\ll b$ in the sense of \eqref{feelingofscale}. Thus by \eqref{simple1} and \eqref{feelingofscale}, we obtain $-\frac{\lambda_{s}}{\lambda}\geq \frac{1}{2}b>0$, thus $\lambda(t)$ is strictly decreasing.\\
We  define $k_{0}, k_{T}$ as:
\begin{equation}
\frac{1}{2^{k_{0}}}\leq \lambda(0) < \frac{1}{2^{k_{0}-1}}, \frac{1}{2^{k_{T}}}\leq \lambda(T)<\frac{1}{2^{k_{T}-1}},
\end{equation}
and for $k_{0}\leq k\leq k_{T}$, let $t_{k}$ be the (unique) time such that
\begin{equation}
\lambda(t_{k})=\frac{1}{2^{k}}.
\end{equation}
Then as in \cite{colliander2009rough}, we can perform a  bootstrap argument:
\begin{lem}\label{intevallength}
In $[0,T]$,  assuming the bootstrap hypothesis
\begin{equation}\label{bhinteval}
t_{k+1}-t_{k}\leq k\lambda^{2}(t_{k}),
\end{equation}
we obtain the  bootstrap estimate
\begin{equation}\label{beinteval}
t_{k+1}-t_{k}\leq \sqrt{k}\lambda^{2}(t_{k}).
\end{equation}
\end{lem}
This estimate is  not sharp, indeed, morally one should have $t_{k+1}-t_{k}\sim (\ln \ln k ) \lambda(t_{k})^{2}$ according to the log-log law. We refer to (2.39) in Lemma 2.4 of \cite{colliander2009rough}.\par

As in \cite{colliander2009rough}, we divide all intervals $[t_{k},t_{k+1}]$ into disjoint intervals $\cup_{j=0}^{J_{k}-1}[\tau_{k}^{j},\tau_{k}^{j+1}]$, where
\begin{equation}\label{intevaldecompose}
\begin{aligned}
t_{k}=\tau_{k}^{0}<\tau_{k}^{1}\cdots <\tau_{k}^{J_{k}-1}<\tau_{k}^{J_{k}}=t_{k+1},\\
\tau_{k}^{j+1}-\tau_{k}^{j}=\frac{\delta_{1}}{4}\lambda(t_{k+1})^{2}, \forall j\leq J_{k}-2\\
\tau_{k}^{J_{k}}-\tau_{k}^{J_{k}-1}\leq \frac{\delta_{1}}{4}\lambda(t_{k+1})^{2}, \forall j\leq J_{k}-2.
\end{aligned}
\end{equation}
Here $\delta_{1}$ is a fixed constant which will be chosen later.
Now, by Lemma \ref{intevallength}, using the bootstrap estimate \eqref{beinteval} (indeed, we   only use the bootstrap hypothesis \eqref{bhinteval}, which is weaker),
we have 
\begin{equation}\label{numberofiteration}
J_{k}\leq 10k^{2}, \sum_{k} J_{k}\leq 10k_{T}^{3}\lesssim |\ln \lambda(T)|^{3}.
\end{equation}

\subsubsection{A quick introduction of upside-down I-method}
We point out without proof the following classical fact on any LWP interval:
\begin{equation}\label{trivial}
\forall  \tkj, \sup_{t\in \tkj}\|u(t)-u(\tau_{k}^{j})\|_{H^{K_{2}}}\lesssim \|u(\tau_{k}^{j})\|_{H^{K_{2}}},
\end{equation}
and directly iterate this over all the LWP intervals (recall we have about $|\ln \lambda(T)|^{3}$ such intervals), to obtain the estimate $\|u(T)\|_{H^{K_{2}}}\lesssim C^{|\ln \lambda(T)|^{3}}$ for some $C>1$, which is clearly weaker than \eqref{goalroughonesoliton} but actually quite close.\par
The idea of the original I-method \cite{colliander2002almost} or up-side down I-method \cite{sohinger2010bounds} are both to  improve the estimate \eqref{trivial} on a LWP interval by working on  certain a slowly varying/almost conserved quantity.\par 
We  introduce the upside-down I-operator $\DDD_{N}$:
\begin{equation}
\widehat{\DDD_{N}f}(\xi):={\mathcal{M}}(\frac{\xi}{N})\hat{f}(\xi),
\end{equation}
where $\mathcal{M}(\xi)$ is a smooth function such that
\begin{equation}
\mathcal{M}(\xi):=
\begin{cases}
(|\xi|)^{K_{2}},  |\xi|\geq 2,\\
1, |\xi|\leq 1.
\end{cases}
\end{equation}
It is easy  to see that  for any $f\in H^{K_{2}}$ we have
\begin{equation}\label{basicrelation}
\|\DDD_{N}f\|_{L^{2}}\lesssim \|f\|_{H^{K_{2}}}\lesssim  N^{K_{2}}\|\DDD_{N}f\|_{l^{2}}.
\end{equation}
Now for some $v$ that solves NLS,
the idea of  the upside-down I-method is to use $\|\DDD_{N}v\|_{L^{2}}^{2}$ to model $\|v\|_{H^{K}}^{2}$, and show that $E^{1}(v):=\|\DDD_{N}v\|_{2}^{2}$ is slow varying, see  \cite{sohinger2010bounds}.

\begin{lem}\label{upsidedown}$[$Proposition 3.4,Lemma 4.5 in \cite{sohinger2010bounds}$]$
There exists a  higher modified energy $E^{2}$ (or to be more precise, $E^{2}_{N}$), such that for any $f$ in $H^{K_{2}}$,
\begin{equation}\label{i1}
\lvert E^{2}(f)- \|\DDD_{N}f\|_{2}^{2}\rvert\lesssim  \frac{1}{N^{1-}}\|\DDD_{N}f\|_{2}^{2},  
\end{equation} 
$\forall M>0$, there exists $\delta_{0}=\delta_{0}(M)>0$ such that if $v$ solves \eqref{nls} with initial data $v_{0}$ and  $\|v_{0}\|_{H^{1}}\leq M,$, then $[0,\delta_{0}]$ is in the lifespan of $v$ and $E^{2}(v)$ is slow varying in the following sense:
\begin{equation}\label{i2}
\sup_{0\leq \tau\leq \delta_{0}} \lvert E^{2}(v(\tau))-E^{2}(v_{0})\rvert\lesssim \frac{1}{N^{\frac{7}{4}-}}E^{2}(v_{0}).
\end{equation}
and 
\begin{equation}\label{iii}
\sup_{0\leq \tau\leq \delta_{0}}\|\DDD_{N}v(t)\|_{L^{2}}^{2}\leq (1+C\frac{1}{N^{\frac{7}{4}-}})(1+C\frac{1}{N^{1-}})\|\DDD_{N}v(t_{0})\|_{L^{2}(\RRR^{2})}^{2},
\end{equation}
where $C$ is some universal constant
\end{lem}
 Actually in  \cite{sohinger2010bounds} one finds only  \eqref{i1} and \eqref{i2}, but   \eqref{iii} is directly implied. Formula \eqref{i1} can be found in the last formula of the proof of the Proposition 3.4 in \cite{sohinger2010bounds}.

\begin{rem}
In \cite{sohinger2010bounds}, the equation is  defocusing, while here we are working with a  focusing equation. However, when one restricts analysis locally, the  two problems actually have no real difference. In \cite{sohinger2010bounds}, the defocusing condition is only used to ensure that any $H^{s}, s>1$ solution  considered is global. Note \cite{sohinger2010bounds} deal with both Euclidean case and Torus case.
\end{rem}

\subsubsection{Proof of Lemma \ref{lemosroughcontrol}}
Now we are ready  to finish the proof of Lemma \ref{lemosroughcontrol}.
\begin{proof}
 Recall we need only to show \eqref{goalroughonesoliton}. Let
\begin{equation}
N=(\frac{1}{\lambda(T)})^{1+\sigma_{1}}.
\end{equation} 
Here $\sigma_{1}=\sigma_{1}(\sigma)$ is a small positive constant chosen later.\\
We have the following estimate on LWP interval:
\begin{equation}\label{keyiteration}
\sup_{t\in\tkj}\|\DDD_{N}u(t)\|_{L^{2}}^{2}\leq (1+C\frac{1}{(N\lambda(T))^{\frac{1}{2}}})\|\DDD_{N}u(\tau_{k}^{j})\|_{L^{2}}^{2}.
\end{equation}
The constant $C$ is independent of $j,k$.
We now assume \eqref{keyiteration} temporarily and we finish the proof of \eqref{goalroughonesoliton}.
By our choice of N, we immediately obtain  from \eqref{keyiteration} that 
\begin{equation}
\sup_{t\in\tkj}\|\DDD_{N}u(t)\|_{L^{2}}^{2}\leq (1+C\frac{1}{(\lambda(T)^{-\sigma_{1}})^{\frac{1}{2}}})\|\DDD_{N}u(\tau_{k}^{j})\|_{L^{2}}^{2}.
\end{equation}
Then we iterate this estimate and we recall that  the total number of LWP intervals is controlled by
\eqref{numberofiteration}, thus we obtain 
\begin{equation}
\|\DDD_{N}u(T)\|_{2}^{2}\lesssim (1+C{\lambda(T)^{\frac{\sigma_{1}}{2}}})^{C(|\ln \lambda(T)|)^{3}}\|u_{0}\|_{H^{K_{2}}}\lesssim 1.
\end{equation}
Thus, we arrive to the  estimate
\begin{equation}
\|u\|_{H^{K_{2}}}\lesssim N^{K_{2}}\|\DDD_{N}u\|_{2}\lesssim \frac{1}{\lambda(T)^{K_{2}(1+\sigma_{1})}}.
\end{equation}
The desired  estimate \eqref{ineqroughcontrol} clearly follows if we choose $\sigma_{1}\leq\sigma/K_{2}$.\\
What is now left  is to prove \eqref{keyiteration}.
We indeed show that 
\begin{equation}\label{strkeyiter}
\sup_{t\in\tkj}\|\DDD_{N}u(t)\|_{L^{2}}^{2}\leq (1+C\frac{1}{(N\lambda(t_{k}^{j}))^{\frac{1}{2}}}\|\DDD_{N}u(\tau_{k}^{j})\|_{L^{2}}^{2}.
\end{equation}
Since, $\lambda(T)\leq \lambda(t_{k+1})$, clearly \eqref{keyiteration} follows from \eqref{strkeyiter}.\\
We now prove \eqref{strkeyiter}.
Indeed, by scaling or direct computation, let $$u_{j,k}(t,x):=\lambda(t_{k+1}) u(\lambda(t_{k+1})^{2}(\tau_{k}^{j}+t),\lambda(t_{k+1}x)),$$ then
$u_{j,k}$ solves \eqref{nls} with initial data $\lambda(t_{k+1})u(\tau_{k}^{j},\lambda(t_{k+1})x)$.
A direct computation leads to $$\|\lambda(t_{k+1})u(\tau_{k}^{j},\lambda(t_{k+1})x)\|_{H^{1}}=\|Q+\epsilon\|_{H^{1}}\leq 2\|Q\|_{H^{1}}.$$
To apply the upside-down I-method in  Lemma \ref{upsidedown}, we first choose $\delta_{0}$ in Lemma \ref{upsidedown} as $\delta_{0}(2\|	Q\|_{H^{1}}^{2})$ and  then we choose $\delta_{1}$ in \eqref{intevaldecompose} as $\delta_{1}=\frac{\delta_{0}}{2}$, and as a consequence for all $j,k$, $$\frac{\tau_{k}^{j}-\tau_{k}^{j+1}}{\lambda(t_{k+1})^{2}}<\delta_{0}.$$
By the upside-down I-method \eqref{iii}, we have:
\begin{equation}
\begin{aligned}
&sup_{t\in [0,\frac{\tau_{k}^{j}-\tau_{k}^{j+1}}{\lambda(t_{k+1})^{2}}]}\|\DDD_{N\lambda(t_{k})}u_{j,k}(t)\|^{2}\\\leq& (1+C\frac{1}{(N\lambda(t_{k}^{j}))^{\frac{7}{4}-}}(1+C\frac{1}{(N\lambda(t_{k}^{j}))^{1-}} \|\DDD_{N\lambda(t_{k}^{j})}u_{j,k}(t)\|_{L^{2}}^{2}\\
 \leq&  (1+C\frac{1}{(N\lambda(t_{k}^{j}))^{\frac{1}{2}}}))\|\DDD_{N\lambda(t_{k}^{j})}u_{j,k}(t)\|_{L^{2}}^{2}.
\end{aligned}
\end{equation}
We now  observe that  for any $\lambda>0$ and for any function $f$ we have 
\begin{equation}
\begin{aligned}
&\lambda (\DDD_{N}f)(\lambda x)=\DDD_{N\lambda}\left(\lambda f(\lambda x)\right),\\
&\|f\|_{2}=\|\lambda f(\lambda x)\|_{2}.
\end{aligned}
\end{equation}
\eqref{strkeyiter} clearly follows.
\end{proof}

\begin{rem}\label{remrough}
It is not hard to see that the proof only relies  on the fact that one  can divide $[0,T]$ into disjoint intervals $\cup_{j,k}I_{k,j}$ such that
\begin{enumerate}
\item $\|u(t)\|_{H^{1}}\sim 2^{-k}$ for $t\in I_{k,j}$(this is equivalent to $\lambda(t)\sim 2^{-k}$ for $t\in I_{k,j}$.)
\item $| I_{k,j}|\sim \frac{1}{2^{2k}}$, $\forall k,j$.
\item $\sharp \{I_{k,j}\}\lesssim |\ln \lambda(T)|^{3}$
\item $\lambda(T)\lesssim \lambda(t), t\in [0,T]$.

\end{enumerate}
and $u(0)\in H^{K_{2}}$.\par
Note that  condition 3 follows from the bootstrap Lemma \ref{intevallength}.

We finally remark that  by further dividing  $I_{k,j}$, it is very easy to improve condition 2 to $|I_{k,j}|\leq \frac{\delta}{{2^{2k}}}$ for any fixed $\delta>0$.\\
\end{rem}

\subsection{ Propagation of regularity}
In this subsection, we give the proof of Lemma \ref{lemosfirstgain}, Lemma \ref{lemosimprovelinfty}, Lemma \ref{lemosfirstgainimpr} and Lemma \ref{lemositeration}. Since our proof relies  on the Strichartz estimates, for completeness we recall them below.
\subsubsection{Strichartz estimates}
 Consider the equation:
\begin{equation}
\begin{cases}
iU_{t}+\Delta U=F,\\
U(0,x)=U_{0}.
\end{cases}
\end{equation}
We write it in the integral form using the Duhamel Formula:
\begin{equation}\label{formuladuhamel}
U(t)= e^{it\Delta}U_{0}-i\int_{0}^{t}e^{i(t-s)\Delta}F(s)ds,
\end{equation}
where $e^{it\Delta}$ is the propagator of linear Schr\"odinger equation.
For notation convenience, let:
$$\Gamma F:=-i\int_{0}^{t}e^{i(t-s)\Delta}F(s)ds.$$
Then one has the following Strichartz estimates:
\begin{equation}
\begin{aligned}
&\forall \frac{2}{q}+\frac{2}{r}=\frac{2}{2},2<q\leq \infty,\\
&\|e^{it\Delta}U_{0}\|_{L^{q}([0,\infty],L^{r}(\RRR^{2}))}\lesssim_{q,r} \|U_{0}\|_{2},\\
&\|\Gamma F\|_{L^{q}([0,t], L^{r}(\RRR^{2}))}\lesssim_{q,r} \|F\|_{L^{q'}([0,t], L^{r'}(\RRR^{2}))},
\end{aligned}
\end{equation}
where $(q', r')$ is the dual of $(q, r)$, i.e. $\frac{1}{q}+\frac{1}{q'}=1, \frac{1}{r}+\frac{1}{r'}=1$.
We  call $(q, r)$ an admissible pair  if and only if $\frac{2}{q}+\frac{2}{r}=\frac{2}{2}, \, 2<q\leq \infty$.
We refer to \cite{tao2006nonlinear}, \cite{keel1998endpoint} and the reference within for a proof.
\subsubsection{Proof of Lemma \ref{lemosfirstgain}}\label{proofosfirstgain}
We fix $\nu<1$, and we estimate $\|\nabla^{\nu}v_{2}\|_{2}$.
Note that $v_{2}$   satisfies:
\begin{equation}\label{modinls}
\begin{cases}
i\partial_{t}v_{2}+\Delta v_{2}=\Delta \chi_{2}v_{1}+2\nabla\chi_{2}\nabla v_{1}-v_{2}|v_{1}|^{2}.
\end{cases}
\end{equation}

As in previous section, we  only need to show:
\begin{equation}\label{goal}
\|\nabla^{\nu}v_{2}(T)\|_{2}\lesssim (\frac{1}{\lambda(T)})^{\sigma}.
\end{equation}
We explain some heuristics for this estimate.
We view the system as a perturbation of the linear Schr\"odinger equation, and the dominating term in the perturbation is the last term on the right side of \eqref{modinls}.  We can view \eqref{modinls} as:
\begin{equation}
i\partial_{t}v_{2}+\Delta v_{2}\approx O(|v_{1}|^{2}v_{2}).
\end{equation}
From the viewpoint of persistence of regularity one gets:
\begin{equation}\label{moralprop}
\|\nabla_{\nu}v_{2}(T)\|_{2}\lesssim \|\nabla^{\nu}v_{2}(0)\|_{2}e^{\|v_{1}(t)\|^{4}_{L^{4}([0,t],L^{4})}}.
\end{equation}
By estimating
$$\|v_{1}(t)\|_{L^{4}(\RRR^{2})}\lesssim \|u(t)\|_{H^{\frac{1}{2}}}\lesssim \frac{1}{\lambda^{\frac{1}{2}}(t)},$$ one obtains by \eqref{scalefromblowupspeed}
\begin{equation}
\|\nabla^{\nu} v_{2}(T)\|_{2}\lesssim e^{\int_{0}^{T}\frac{1}{\lambda^{2}(t)}}\lesssim e^{|\ln \lambda(T)|^{101}}.
\end{equation}
This estimate though is too week, of course. On the other hand we  didn't even  use  the key estimate \eqref{origingain} in the log-log regime. To obtain the   stronger estimate \eqref{goal} instead of the  $L^{4}_{t,x}$ norm in \eqref{moralprop} we use a more flexible 
$L^{q}_{t}L^r_{x}$. More precisely 
we replace $\|v_{1}(t)\|^{4}_{L^{4}([0,t],L^{4})}$ by some $\|v_{1}(t)\|^{q}_{L^{q}([0,T], L^{r}(\RRR^{2}))}$ such that $(q, r)$ is an admissible pair. 
Now, we make the key observation that $L^{2}\dot{H}^{1}$ and $L_{t}^{4}L_{x}^{4}$ have the same scaling as $L^{q}_{t}L_{x}^{r}$ whenever $(q, r)$ is an admissible pair. Thus, by interpolating  the two estimates
\begin{equation}\label{preinter}
\begin{aligned}
&\int_{0}^{T}\|v_{1}(t)\|_{4}^{4}\lesssim |\ln \lambda(t)|^{101},\\
&\int_{0}^{T}\|\nabla v_{1}(t)\|_{2}^{2} \lesssim 1,
\end{aligned}
\end{equation}
we  derive 
\begin{equation}
\int_{0}^{T}\|u(t)\|_{L^{r}}^{q}\lesssim (|\ln \lambda(t)|)^{1/100},
\end{equation}
for an admissible $(q, r)$ carefully picked. Then we obtain 
\begin{equation}
\|\nabla^{\nu}v_{2}(T)\|_{2}\lesssim \|\nabla^{\nu}v_{2}(0)\|_{2}e^{\|v_{1}(t)\|^{q}_{L^{q}([0,T],L^{r})}}\lesssim e^{|\ln (\lambda(T)|)^{\frac{1}{100}}},
\end{equation}
which implies the desired estimate \eqref{goal}.\\
We now go to the details.
We  need the following technical lemmata.
\begin{lem}\label{tech1}
Let $\frac{2}{q_{0}}+\frac{2}{r_{0}}=1$, let $q_{0}'$, $r_{0}'$ be their dual,
let $p_{0}$ be defined by $\frac{1}{r_{0}'}=\frac{1}{2}+\frac{1}{p_{0}}$,
let $h_{0}$ be defined by $h_{0}=1-\frac{1}{p_{0}}$,
then we have the following estimate uniform with respect to  any time interval $I$ 
\begin{eqnarray}\label{11}
&&\||v_{1}|^{2}\nabla^{\nu} v_{2}\|_{L^{r_{0}'}}\lesssim \|\nabla^{\nu}v_{2}\|_{L^{2}}\|v_{1}\|^{2}_{L^{2p_{0}}},\\\label{12}
&&\||v_{1}|^{2}\nabla^{\nu} v_{2}\|_{L^{q_{0}'}(I;L^{r_{0}'})}\lesssim \|\nabla^{\nu} v_{2}\|_{L^{\infty}(I;L^{2})}\|v_{1}\|^{2}_{L^{2q_{0}'}(I;L^{2p_{0}})},\\\label{13}
&&\|v_{1}\|_{L^{2p_{0}}}\lesssim \|v_{1}\|_{H^{h_{0}}},\\\label{14}
&&\|v_{1}\|_{L^{2q_{0}'}(I;H^{h_{0}})}\lesssim \|v_{1}\|_{L^{2}(I;H^{1})}^{\frac{h_{0}-1/2}{1/2}}\|v_{1}\|_{L^{4}(I;H^{\frac{1}{2}})}^{\frac{1-h_{0}}{1/2}}.
\end{eqnarray}
Moreover we can choose $q_{0}$ large  enough such that $1-\frac{1}{10000}<h_{0}<1$, $q_{0}'\leq 2$.
\end{lem}
\begin{lem}\label{tech2}
Let $\frac{2}{q_{1}}+\frac{2}{r_{1}}=1$, let $q'_{1}, \, r'_{1}$ be their dual, let $g_{1}, \tilde{g}_{1}$ be defined by $\nu-1=-\frac{2}{g_{1}}, -\nu\equiv(1-\nu)-1=-\frac{2}{\tilde{g}_{1}}$, let $p_{1}$ be defined as $\frac{1}{r_{1}'}=\frac{1}{2}+\frac{1}{p_{1}}$, let $h_{1}$ be defined as 
$h_{1}-1=-\frac{2}{p_{1}}$ and  let $w_{1}$ be defined as $\frac{1}{q_{1}'}=\frac{1}{2}+\frac{1}{w_{1}}$,
then we have the following estimates uniform with respect to  any time interval $I$ 
\begin{eqnarray}\label{21}
&&\||v_{1}||\du v_{1}|v_{2}\|_{L^{r'_{1}}}\lesssim \|\du v_{1}\|_{L^{\tilde{g}_{1}}}\|v_{1}\|_{L^{p_{1}}}\|v_{2}\|_{L^{g_{1}}},\\\label{22}
&& \|\du v_{1}\|_{L^{\tilde{g}_{1}}}\lesssim \|v_{1}\|_{H^{1}}, \|v_{2}\|_{L^{g_{1}}}\lesssim \|\nabla^{\nu}v_{2}\|_{L^{2}},\\\label{23}
&&\||v_{1}||\du v_{1}|v_{2}\|_{L^{q'_{1}}(I; L^{r'_{1}})}\lesssim \|v_{2}\|_{L^{\infty}(I;H^{\nu})}\|v_{1}\|_{L^{2}(I;H^{1})}\|v_{1}\|_{L^{w_{1}}(I;L^{p_{1}})},\\\label{24}
&&\|v_{1}\|_{L^{p_{1}}}\lesssim \|v_{1}\|_{H^{h_{1}}},\\\label{25}
& &\|v_{1}\|_{L^{w_{1}}(I;H^{h_{1}})}\lesssim \|v_{1}\|_{L^{2}(I;H^{1})}^{\frac{h_{1}-1/2}{1/2}}\|v_{1}\|_{L^{4}(I;H^{\frac{1}{2}})}^{\frac{1-h_{1}}{1/2}}.
\end{eqnarray}
Moreover we can choose $q_{1}$ large such that $1-\frac{1}{10000}<h_{1}<1$, $w_{1}\leq 4$.
\end{lem}
The proofs are a direct consequence of Sobolev and H\"older inequalities as well as  standard interpolation  
$$\|u\|_{H^{s}}\lesssim \|u\|_{H^{s_{1}}}^{\frac{s_{2}-s}{s-s_{1}}}\|u\|_{H^{s_{2}}}^{\frac{s-s_{1}}{s_{2}-s_{1}}}, s_{1}<s<s_{2}.$$ 
We finally remark that all the indices happen to coincide  because of   scaling reasons. Indeed, we can check  that $(2q_{0}',2p_{0})$, $(w_{1},p_{1})$ are both admissible pairs.

Now, we pick $(q_{0}, r_{0}), (q_{1},r_{1})$ as in Lemma \ref{tech1} and  Lemma \ref{tech2}, and we let all the other associated  indices  be determined as in these two lemmas.
By estimating  
$\|u(t)\|_{H^{\frac{1}{2}}}\lesssim \frac{1}{\lambda(t)^{\frac{1}{2}}},$
 we have 
 $\|v_{1}(t)\|_{H^{\frac{1}{2}}}\lesssim \frac{1}{\lambda(t)^{\frac{1}{2}}}$. Thus, by \eqref{scalefromblowupspeed} we have 
Now we have:
\begin{equation}\label{pre}
\int_{0}^{T}\|v_{1}(t)\|_{4}^{4}\lesssim |\ln \lambda(t)|^{101}.
\end{equation}
Also from \eqref{origingain}, we have
\begin{equation}\label{pre1}\int_{0}^{T}\|\nabla v_{1}(t)\|_{2}^{2} \lesssim 1.
\end{equation}
Using  \eqref{pre}, \eqref{pre1} and  \eqref{13},  \eqref{14} in Lemma \ref{tech1} we obtain 
$$\|v_{1}\|_{L^{2q_{0}'}([0,T];L^{2p_{0}})}^{2q_{0}'}\leq |\ln \lambda(T)|^{1/2},$$
and using  \eqref{pre}, \eqref{pre1} and   \eqref{24}, \eqref{25} in Lemma \ref{tech2} we obtain  $$\|v_{1}\|_{L^{w_{1}}([0,T];L^{p_{1}})}^{w_{1}}\leq |\ln\lambda(T)|^{1/2}.$$
Thus, we  are able to divide $[0,T]$ into $J_{k}$ disjoint intervals $[\tau_{k},\tau_{k+1}], k=1,...,J_{k}$, such that
\begin{equation}\label{iterationcontrol2}
\begin{aligned}
&J_{k}\sim_{\epsilon}| \ln (\lambda(T))|^{\frac{1}{2}},\\
&\|v_{1}\|_{L^{2q_{0}'}([0,T];L^{2p_{0}})}\leq \epsilon, \\
&\|v_{1}\|_{L^{w_{1}}([0,T]);L^{p_{1}}}\leq \epsilon,
\end{aligned}
\end{equation}
where $\epsilon>0$ is a fixed small  constant, to be chosen later.

Now we use the Duhamel formula \eqref{formuladuhamel} for \eqref{modinls} in $[\tau_{k},\tau_{k+1}]$,  and we obtain for any $t\in [\tau_{k},\tau_{k+1}]$,
\begin{equation}\label{propinlwpinterval}
\begin{aligned}
&\quad \|v_{2}(t)-v_{2}(\tau_{k})\|_{\dot{H}^{\nu}}\\&\leq 
C\int_{0}^{T}\| v_{1}\|_{H^{\nu}}+C\int_{0}^{T}\|v(t)\|_{H^{1+\nu}}
+\|\int_{\tau_{k}}^{t}e^{i(t-\tau)\Delta}\nabla^{\nu}(|v_{1}|^{2}v_{2})(\tau)d\tau\|_{2}\\
&\leq C\int_{0}^{T}\|u(t)\|_{H^{1+\nu}}+C\||v_{1}|^{2}\nabla^\nu v_{2}||\|_{L^{q_{0}'}([t_{k},\tau];L^{r_{0}'})}+C\|\nabla^{\nu}(|v_{1}|^{2})v_{2}\|_{L^{q'_{1}}([t_{k},\tau]);L^{r_{1}'})},
\end{aligned}
\end{equation}
where in the last step we have used Strichartz estimates and Fractional Leibniz rule (See Theorem A.8, \cite{kenig1993well}).  We remark It is not hard to see that we can choose $(q_{0},r_{0})=(q_{1},r_{1}).$
Now we  plug in the estimates in \eqref{12}, \eqref{22} into \eqref{propinlwpinterval}, to obtain
 
\begin{equation}
\begin{aligned}
&\quad \|v_{2}(t)-v_{2}(\tau_{k})\|_{\dot{H}^{\nu}}\\
&\leq C\int_{0}^{T}\|u(t)\|_{H^{1+\nu}}\\
&+C(\|v_{1}\|_{L^{2q_{0}'}([\tau_{k},\tau_{k+1}];L^{2p_{0}})}^{2}+\|v_{1}\|_{L^{2q_{0}'}([\tau_{k},\tau_{k+1}];L^{2p_{0}})}\|v_{1}\|_{L^{2}([0,T];H^{1})})\sup_{[\kappa\in [\tau_{k},t] ]}\|v_{2}\|_{H^{\kappa}},
\end{aligned}
\end{equation}
i.e.
\begin{equation}
\quad \|v_{2}(t)-v_{2}(\tau_{k})\|_{{H}^{\nu}}\leq C+C\int_{0}^{T}\|u(t)\|_{H^{1+\nu}}+C\epsilon sup_{t\in [t_{k},\tau]}\|u(t)\|_{H^{\nu}}.
\end{equation}
The term $\int_{0}^{T}\|u(t)\|_{H^{1+\nu}}$ is not hard to control. Interpolating between $\|u\|_{H^{1}}$ and $\|u\|_{H^{K^{2}}}$, we obtain
\begin{equation}
\|u(t)\|_{H^{1+\nu}}\lesssim \|u(t)\|_{H^{1}}^{\frac{K_{2}-1-\nu}{K_{2}-1}}\|u(t)\|_{H^{K_{2}}}^{\frac{\nu}{K_{2}-1}},
\end{equation}
Now we plug in the estimate $\|u(t)\|_{H^{1}}\lesssim \frac{1}{\lambda(t)}$ and  $\|u(t)\|_{H^{K_{2}}}\lesssim_{\sigma}\frac{1}{\lambda(t)^{K_{2}+\sigma}}$ from Lemma \ref{lemosroughcontrol}, and we obtain
\begin{equation}
\|u(t)\|_{H^{1+\nu}}\lesssim_{\sigma}\frac{1}{\lambda^{1+\nu+C\sigma}(t)}.
\end{equation}
By choosing $\sigma$ small enough such that $1+\nu+C\sigma<2$, and  using \eqref{scalefromblowupspeed}, we have
\begin{equation}\label{controloftroubleterm}
\int_{0}^{T}\|u(t)\|_{H^{1+\nu}}\lesssim 1.
\end{equation}
We plug in \eqref{controloftroubleterm} into \eqref{propinlwpinterval},  and we choose $\epsilon$ small enough to  obtain
\begin{equation}
sup_{t\in [t_{k},t_{k+1}]}\|u(t)\|_{H^{\nu}}\leq 2\|u(t_{k})\|_{H^{\nu}}+C.
\end{equation} 
 Iterating this over the $J_{k}\sim |\ln \lambda(T)|^{\frac{1}{2}}$ intervals,
we obtain the estimate
\begin{equation}
\|u(T)\|_{H^{\nu}}\lesssim e^{C|\ln \lambda(T)|^{\frac{3}{4}}},
\end{equation}
which  implies the desired estimate \eqref{goal}.

\subsubsection{Proof of Lemma \ref{lemosimprovelinfty}}\label{proofimprovelinfty}
We now  prove Lemma \ref{lemosimprovelinfty}. Indeed, by Lemma \ref{lemosroughcontrol} and Lemma \ref{lemosfirstgain}, we obtain for any  $\nu<1$, $\tilde{\sigma}>0$ the estimates
\begin{eqnarray}
\|v_{2}(t)\|_{H^{K_{2}}}\lesssim_{\tilde{\sigma}}\frac{1}{\lambda^{K^{2}+\tilde{\sigma}}(t)},\\
\|v_{2}(t)\|_{H^{\nu}}\lesssim_{\tilde{\sigma}}\frac{1}{\lambda^{\tts}(t)}.
\end{eqnarray}
Now, by Sobolev embedding and  interpolation, we obtain
\begin{equation}
\|v_{2}(t)\|_{L^{\infty}(\RRR^{2})}\lesssim_{\tilde{\sigma}}\|v_{2}(t)\|_{H^{1+\tilde{\sigma}}}\lesssim_{\tts} \frac{1}{\lambda^{C\tilde{\sigma}}(t)}\frac{1}{\lambda^{\frac{K_{2}-K_{2}\nu}{K_{2}-\nu}}(t)}.
\end{equation}
Since $\tts>0$ is arbitrary and  we can choose $\nu$ as close to 1 as we want,  this clearly gives Lemma \ref{lemosimprovelinfty}.

\subsubsection{Proof of Lemma \ref{lemosfirstgainimpr}}\label{prooffirstgainimpr}
Now we prove Lemma \ref{lemosfirstgainimpr}. 
We fix $\nu$ and we point out that all the constants in the proof will  depend the choice of $\nu$.
By choosing $\sigma$ small enough, Lemma \ref{lemosimprovelinfty} gives for some small $c=c_{\nu}>0$
\begin{equation}\label{techlinfty}
\|v_{2}(t)\|_{L^{\infty}}\lesssim \frac{1}{\lambda^{1-c\nu}(t)}.
\end{equation}
By choosing $\sigma$ even smaller, by Lemma \ref{lemosfirstgain}, we have
\begin{equation}\label{techhnu}
\|v_{2}(t)\|_{H^{\nu}}\lesssim \frac{1}{\lambda^{c\nu/10}(t)}.
\end{equation}
Clearly $v_{3}$ also satisfies  the estimate \eqref{techlinfty}, \eqref{techhnu}, and also it  satisfies the  equation
\begin{equation}
i\partial_{t}v_{3}+\Delta v_{3}=\Delta \chi_{3} v_{2}+2\nabla \chi_{3}\nabla v_{2}-v_{3}|v_{2}|^{2}.
\end{equation}
By the Duhamel's formula, we obtain
\begin{equation}
v_{3}(t)=e^{it\Delta}(v_{3}(0))+i\int_{0}^{t}e^{i(t-s)\Delta}(\Delta \chi_{3} v_{2}+2\nabla \chi_{3}\nabla v_{2}-v_{3}|v_{2}|^{2})(s)ds.
\end{equation}
Thus, 
\begin{equation}
\begin{aligned}
\|v_{3}(t)\|_{H^{\nu}}\lesssim \|v_{3}(0)\|_{H^{\nu}}+\int_{0}^{t}\|v_{2}\|_{H^{\nu}}+\int_{0}^{t}\|v_{2}\|_{H^{1+\nu}}+\int_{0}^{t}\|v_{3}|v_{2}|^{2}\|_{H^{\nu}}.
\end{aligned}
\end{equation}
As argued previously in \eqref{controloftroubleterm}, $\int_{0}^{t}\|v_{2}\|_{H^{1+\nu}}\lesssim 1$. 
Thus, to finish the proof of Lemma \ref{lemosfirstgainimpr}, we only need to show
\begin{equation}\label{finalgoalproofoflemfirstgain}
\int_{0}^{T}\|v_{3}|v_{2}|^{2}(t)\|_{H^{\nu}}\lesssim 1.
\end{equation}
Indeed,
\begin{equation}\label{whatever}
\|v_{3}|v_{2}|^{2}\|_{H^{\nu}}\lesssim \|v_{3}\|_{H^{\nu}}\|v_{2}\|_{L^{\infty}}^{2}+\|v_{3}\|_{L^{\infty}}\|v_{2}\|_{L^{\infty}}\|v_{2}\|_{H^{\nu}}.
\end{equation}
Now we plug  the estimates \eqref{techlinfty} and \eqref{techhnu} into \eqref{whatever} and we recall that $v_{3}$ also satisfy estimate \eqref{techlinfty} and \eqref{techhnu}, we have
\begin{equation}
\|v_{3}|v_{2}|^{2}(t)\|_{H^{\nu}}\lesssim \frac{1}{\lambda^{2-c\nu}(t)}.
\end{equation}
Thus, by \eqref{scalefromblowupspeed}, estimate \eqref{finalgoalproofoflemfirstgain} follows. 

\subsubsection{Proof of Lemma \ref{lemositeration}}\label{proofositeration}
Finally we prove Lemma \ref{lemositeration}. This is quite straightforward. First, one can directly check (again) that $v_{j+1}$ satisfies:
\begin{equation}
i\partial_{t}v_{j+1}+\Delta v_{j}=\Delta \chi_{j+1}v_{j}+2\nabla \chi_{j+1}\nabla v_{j}-v_{j+1}|v_{j}|^{2}.
\end{equation}
Since we assume that $K_{2}>2$, then  by Lemma \ref{lemosimprovelinfty}, we have (for $j\geq 2$):
\begin{equation}
\|v_{j}(t)\|_{\infty}\lesssim_{\sigma} \frac{1}{\lambda^{\sigma}(t)}.
\end{equation}
Clearly the same estimates hold for $v_{j+1}$.\\
Again by the Duhamel's formula:
\begin{equation}\label{lastduhamel}
v_{j+1}(t)=e^{it\Delta}(\chi_{j+1}u_{0})+i\int_{0}^{t}e^{i(t-\tau)\Delta}(\Delta \chi_{j+1}v_{j}+2\nabla \chi_{j+1}\nabla v_{j}-v_{j+1}|v_{j}|^{2})(\tau)d\tau,
\end{equation}
We remark here that the main term to control is  $2\nabla \chi_{j+1}\nabla v_{j}$.
Directly from \eqref{lastduhamel}, we obtain:
\begin{equation}
\|v_{j+1}(t)\|_{H^{\tilde{r}}}\lesssim \|\chi_{j}u(0)\|_{H^{\tilde{r}}}+\int_{0}^{t}\|v_{j}(\tau)\|_{H^{\tilde{r}}}+\int_{0}^{t}\|v_{j}(\tau)\|_{H^{\tilde{r}+1}}+\int_{0}^{t}\|v_{j+1}|v_{j}|^{2}(\tau)\|_{H^{\tilde{r}}}.
\end{equation}
Clearly, to finish the proof we only need to show 
\begin{equation}\label{oskey1}
\int_{0}^{t}\|v_{j}(\tau)\|_{H^{\tilde{r}+1}}\lesssim 1,
\end{equation}
and
\begin{equation}\label{oskey2}
\int_{0}^{t}\|v_{j+1}|v_{j}|^{2}\|_{H^{\tilde{r}}} \lesssim 1.
\end{equation}
We first prove \eqref{oskey2}. By interpolating  between the estimate 
$$\|v_{j}\|_{H^{r}}\lesssim 1\quad \mbox{  and } \quad \|u(\tau)\|_{H^{K_{2}}}\lesssim_{\sigma}\frac{1}{\lambda^{K_{2}+\sigma}(\tau)},$$ we obtain
\begin{equation}
\|v_{j}(\tau)\|_{H^{\tilde{r}}}\lesssim_{\sigma} \frac{1}{\lambda^{1+C\sigma}(\tau)}.
\end{equation}
Note that the same estimates hold for $v_{j+1}$.
Thus,
\begin{equation}
\|v_{j+1}|v_{j}|^{2}(\tau)\|_{H^{\tilde{r}}}\lesssim \|v_{j}\|_{H^{\tilde{r}}}\|v_{j}\|_{L^{\infty}}^{2}\lesssim_{\sigma} \frac{1}{\lambda(\tau)^{1+C\sigma}}.
\end{equation}
By choosing $\sigma$ small enough, such that $1+C\sigma<2$, \eqref{oskey2} clearly follows from \eqref{scalefromblowupspeed}.
Finally, we turn to the proof of \eqref{oskey1}.  Again by  interpolation, we obtain
\begin{equation}
\|v_{j}\|_{H^{\tilde{r}+1}}\lesssim \|v_{j}(\tau)\|_{H^{r}}^{\frac{K_{2}-\tilde{r}-1}{K_{2}-r}}\|v_{j}(\tau)\|_{H^{K^{2}}}^{\frac{\tilde{r}+1-r}{K_{2}}-r}\lesssim_{\sigma} \left(\frac{1}{\lambda(\tau)}\right)^{\frac{K_{2}(\tilde{r}+1-r)}{K_{2}-r}+C\sigma}.
\end{equation}
The key point is our choice of $\tilde{r}$ that ensures $\frac{K_{2}(\tilde{r}+1-r)}{K_{2}-r}<2$, thus by choosing $\sigma$ small enough, 
 we have $\frac{K_{2}(\tilde{r}+1-r)}{K_{2}-r}+C\sigma<2$. Finally by \eqref{scalefromblowupspeed} estimate \eqref{oskey1} follows . 

\section{Proof of Lemma \ref{bootstrap}: Multi Solitons Model}\label{proofofts}
This section is devoted to the proof of Lemma \ref{bootstrap}.

The proof is involved, hence we first  discuss some heuristics .
 Lemma \ref{bootstrap} is the consequence of the following facts.
\begin{enumerate}
\item  From previous work, \cite{merle2006sharp},\cite{planchon2007existence}, the  solution $u_{j}$ to \eqref{nls} with initial data $\chi_{1,loc}(x-x_{j,0})u_{0}$ , $j=1,\dots, m$,    blows up according to log-log law.
\item  Assume some solution $v$ to \eqref{nls}  blows up according to log-log law, and assume $t_{0}$ is close enough to the blow up time $T^{+}(v)$. Let $F=F(t,x)$ be some smooth perturbation. Then, the solution $\tilde{v}$ to the following Cauchy problem
\begin{equation}
\begin{cases}
i\tilde{v}_{t}=-\Delta  \tilde{v}-|\tilde{v}|^{2}\tilde{v}+F,\\
\tilde{v}(0)=v(t_{0}),
\end{cases}
\end{equation}
still blows up according to the log-log law.
\item By our smoothness condition \eqref{initialoutsmooth1}, we can show  that solution  to \eqref{nls} with initial data $\chi_{1,loc}(x-x_{j,0})u_{0}$  is  smooth in the region $|x-x_{j}|\geq \frac{1}{2}$ for j=1,\dots, m.\\
\item  For $ j=1,\dots, m$ the function $u_{j}:=\chi_{1}(x-x_{j}(t))$ satisfies the following equation
\begin{equation}
i\partial_{t}u_{j}=-\Delta u_{j}-|u_{j}|^{2}u_{j}+F_{j},
\end{equation}
where
\begin{equation} 
\begin{aligned}
F_{j}=&-\nabla \chi_{1,loc}(x-x_{j})\frac{dx_{j}}{dt}+\Delta \chi_{1,loc}(x-x_{j}) u(t,x)+2\nabla \chi_{1,loc}(x-x_{j})\nabla u\\
&+(\chi^{3}_{1,loc}-\chi_{1,loc})(x-x_{j})|u|^{2}u .
\end{aligned}
\end{equation}
 \end{enumerate}
 
 The idea is the following loop argument, which is false of course, but it can be made rigorous by bootstrap argument.
If we assume  $u_{j}$ evolves according to the log-log law for each $j$, we basically know that $u$ is smooth in the region $\min_{j}\{|x-x_{j}|\}\geq 1/2$. And since $\nabla \chi_{1,loc}(x-x_{j})$ and $\chi_{1,loc}(x-x_{j})$ are  supported in $|x-x_{j}|\geq \frac{2}{3}$, this implies that  $F_{j}$ above is  smooth. Thus fact 4 and  fact 2 imply that $u_{j}$ actually evolves according to the log-log law for each $j$. Thus, the assumption, that $u_{j}$ evolves according to log-log law, is right.
 
 Let us turn to the details and a rigorous mathematical proof.
\subsection{Outline of the Proof}
Recall that the re-scaled time $s_{j}$ satisfies $\frac{dt}{ds_{j}}=\lambda_{j}^{2}$ for $j=1,\dots, m$. The system \eqref{tsnls2} implies the following:
\begin{equation}
\begin{aligned}
&&\frac{d}{ds_{j}}\{(\eps_{1}^{j}, |y|^{2}\Sigma_{b_{j}})+(\eps^{j}_{2},|y|^{2}\Theta_{b_{j}})\}=0, j=1,\dots, m,\\
&&\frac{d}{ds_{j}}\{(\eps^{j}_{1},y\Sigma_{b_{1}})+(\eps^{j}_{2},y\Theta_{b_{j}})\}=0, j=1,\dots, m,\\
&&\frac{d}{ds_{j}}\{-(\eps_{1}^{j},\Lambda^{2}\Theta_{b_{j}})+(\epsilon^{j}_{2},\Lambda^{2}\Sigma_{b_{j}})\}=0, j=1,\dots, m,\\
&&\frac{d}{ds_{j}}\{-(\eps_{1}^{j}, \Lambda \Theta_{b_{j}})+(\eps_{2}^{j}, \Lambda \Sigma_{b_{j}})\}=0, j=1,\dots, m.
\end{aligned}
\end{equation}
Now, by pure algebraic computation  as in \cite{raphael2005stability}, one is  able to write down almost the same modulation ODE for $\{b_{j},\lambda_{j}, x_{j}, \tilde{\gamma}_{j}\}$ for  $j=1,\dots, m$ as in formula (71), (72), (73), (74) in \cite{raphael2005stability}, where $\tilde{\gamma}_{j}(s_{j}):=\gamma_{j}-s_{j}$.  Basically, the only difference is  that the $\lambda^{2}E$ term in \cite{raphael2005stability}  is  replaced\footnote{Because in \cite{raphael2005stability}, the solution $u$ has the form $u:=\frac{1}{\lambda(t)}(\tilde{Q}_{b}+\epsilon)(\frac{x-x(t)}{\lambda(t)})e^{-i\gamma}$, thus all the term $E(\tilde{Q}_{b}+\epsilon)$ is equal to $\lambda^{2}E$. In our model , we cannot make this substitution and have to keep the term of form $E(\tilde{Q}_{b_{j}}+\eps^{j})$.} by $E(\tilde{Q}_{b_{j}}+\eps^{j})$.

\subsubsection{Modulation estimates} 
Now, by exactly the same argument as the proof of Lemma  5 in \cite{merle2006sharp} 
we have the analogue of Lemma  \ref{lemmoduestimateinitial}
\begin{lem}\label{lemmoduationestimate}
In the setting of Lemma \ref{bootstrap}, the following modulation estimates hold for $t\in [0,T]$, and $j=1,\dots, m$:
\begin{equation}\label{bootmodu1}
  |\frac{1}{\lambda_{j}}\frac{d\lambda_{j}}{ds_{j}}+b_{j}|+|\frac{db_{j}}{ds_{j}}|\leq C\left(\int |\nabla \eps^{j}|^{2}+\int |\eps^{j}|^{2} e^{-|y|}\right)+\Gamma_{b_{j}}^{1-C\eta}+C|E(\qbj+\epj)|,
  \end{equation}
  \begin{equation}\label{bootmodu2}
  \begin{aligned}
&  |\frac{d\tilde{\gamma}_{j}}{ds_{j}}-\frac{1}{|Q_{1}|^{2}}(\eps^{j}_{1},L_{+}Q_{2}|+|\frac{1}{\lambda_{j}}\xs|\\
  \leq &\delta(\alpha)(\int |\nabla\eps^{j}|^{2}e^{-2(1-\eta)\frac{\theta(b_{j}|y|)}{b_{j}}}+\int |\eps^{j}|^{2}e^{-|y|}|)^{\frac{1}{2}}+C\int |\nabla \epsilon^{j}|^{2}+\Gamma_{b_{j}}^{1-C\eta}\\
  &+C|E(\qbj+\epj)|.
  \end{aligned}
  \end{equation}
\end{lem}
The proof follows exactly as  the proof of (2.36), (2.37) in Lemma 5 of \cite{merle2006sharp}.

\subsubsection{Estimates by the conservation law}
Similarly, following the proof of Lemma 5 in \cite{merle2006sharp},
we have the analogue of Lemma \ref{lemenergyinitial},
\begin{lem}\label{estiamtebyconservationlaw}
In the setting of Lemma \ref{bootstrap},  the following estimates  hold for  $t\in [0,T]$, for $j=1,\dots, m$
 \begin{equation}\label{energyestimate1}
 |2(\eps^{j}_{1},\Sigma_{b_{j}})+2(\eps^{j}_{2},\Theta_{b_{j}})|\leq C\left(\int |\nabla\eps^{j}|^{2}+|\eps^{j}|^{2}e^{-|y|}\right)+\Gamma_{b_{j}}^{1-C\eta}+C|E(\qbj+\epj)|,
 \end{equation}
 \begin{equation}\label{momentumestimate1}
 |(\epsilon^{j}_{2},\nabla\Sigma_{b_{j}})|\leq C\delta(\alpha)(\int |\nabla \eps^{j}|^{2}+\int |\eps^{j}|^{2}e^{-|y|})^{\frac{1}{2}}+CP(\qbj+\epj).
 \end{equation}
\end{lem}
\begin{proof}
The proof of \eqref{energyestimate1} is exactly the same as the proof of (2.35) in Lemma 5 in \cite{merle2006sharp}.
The proof of \eqref{momentumestimate1} is a little different, since in \cite{merle2006sharp} the authors  use the zero-momentum condition which is not used here. A direct simple algebraic computation  shows that 
$$P(\qbj+\epj)=-(\nabla \Sigma_{b_{j}},\eps^{j}_{2})+(\nabla \Theta_{b_{j}},\eps_{1}^{j})-2(\nabla \eps_{1}^{j},\eps_{2}^{j}).$$
By a point-wise control  $|\nabla\Theta_{b_{j}}|(y)\lesssim e^{-|K||y|}$ and Cauchy Schwartz, one has:
$$|(\nabla \Theta_{b_{j}},\epsilon)|\leq C(\int |\eps^{j}|^{2}e^{-K|y|})^{\frac{1}{2}}.$$
From here one has that $|(\nabla \eps_{1}^{j},\eps_{2}^{j})|\leq \|\eps\|_{2}\|\nabla\eps\|_{2}$.

Using the general functional analysis fact\footnote{This estimate holds for all $H^{1}$ functions, see (2.38) in Lemma 5 in \cite{merle2006sharp}.}$$\int |\eps^{j}|^{2}e^{-K|y|}\leq C_{K}(\int |\nabla \eps^{j}|^{2}+\int |\eps^{j}|^{2}e^{-|y|}),$$
the  bootstrap hypothesis \eqref{assumuepertbation} and \eqref{assumesmallnessoflocalerror}:
$$\|\epj\|_{L^{2}}\leq \delta(\alpha), \int |\nabla \eps^{j}|^{2}+\int |\eps^{j}|^{2}e^{-|y|}\leq \delta(\alpha).$$
we have that \eqref{momentumestimate1} follows.\\
\end{proof}

\subsubsection{Control of local quantity}
 We use the  bootstrap hypothesis (in particular, the control of local conserved quantity) to  show the following
\begin{lem}\label{control of local quantity}
In the setting of Lemma \ref{bootstrap},  for $t\in [0,T]$ and $j=1,\dots, m$ the following estimates hold
\begin{eqnarray}\label{energycontrol}
&&|E(\eps^{j}+\qbj)|\lesssim \lambda_{j}^{2-},\\\label{momentumcontrol}
&&|P(\eps^{j})+\qbj|\lesssim \lambda_{j}^{1-}.
\end{eqnarray}
\end{lem}
\begin{proof}
We only prove \eqref{energycontrol}, and the second inequality follows by a similar argument.\\
Direct computation shows that 
\begin{equation}\label{techstep}
\begin{aligned}
  &\lvert E(\epj+\qbj)(t)-\lambda_{j}^{2}(t)E_{loc}(x_{j}(t),u(t))\rvert\\
=&\lambda_{j}^{2}\lvert E(u\chi_{1,loc}(x-x_{j}))- E_{loc}(x_{j}(t),u(t))\rvert\\
\leq 
&\lambda^{2}_{j}\int_{\frac{2}{3}\leq |x-x_{j}|\leq 1}|\nabla u|^{2}+|u|^{4}.
\end{aligned}
\end{equation}
Note that by a standard Sobolev imbedding and by the bootstrap hypothesis \eqref{assumeoutsidesmoothness}, we have:
$$\int_{\frac{2}{3}\leq |x-x_{j}|\leq 1}|\nabla u|^{2}+|u|^{4}\leq \|u\|^{4}_{H^{N_{1}}(\min_{j}\{|x-x_{j,0}|\}\geq \frac{1}{2}\})}\leq (\frac{1}{\max_{j}\{\lambda_{j,0}\}})^{4}.$$
(We will choose $N_{2}$ large enough, and thus $N_{1}<\frac{N_{2}}{2}$ can be chosen large enough so  that all the desired Sobolev embedding holds.)

On the other hand, by the bootstrap hypothesis \eqref{assumelocalcontrolofconserveredquantity}, and the assumption on initial data \eqref{initialtamenessofenergy}, we obtain
$$\lambda_{j}^{2}|E_{loc}(x_{j}(t),u(t))|\lesssim \lambda_{j}^{2}.$$
We plug these two estimates into \eqref{techstep},  and we obtain
\begin{equation}\label{techstep2}
|E(\eps^{j}+\qbj)|\lesssim \lambda_{j}^{2}\left(1+(\frac{1}{\max_{j}\{\lambda_{j,0}\}})^{4}\right).
\end{equation}
Finally, by the bootstrap hypothesis \eqref{assumealmostmono}, we have  $\lambda_{j}(t)\lesssim \max_{j}\{\lambda_{j,0}\}$ for $j=1,\dots, m.$
Thus,  \eqref{energycontrol} clearly follows from \eqref{techstep2}.\\
\end{proof}

\subsubsection{Modulation estimates and estimates by conservation law, restated}
In this section we summarize what we have found above. 
\begin{lem}\label{mees}
\setting, and for $j=1,\dots, m$
\begin{equation}\label{bootmodu1re}
  |\frac{1}{\lambda_{j}}\frac{d\lambda_{j}}{ds_{j}}+b_{j}|+|\frac{db_{j}}{ds_{j}}|\leq C\left(\int |\nabla \eps^{j}|^{2}+\int |\eps^{j}|^{2} e^{-|y|}\right)+\Gamma_{b_{j}}^{1-C\eta},
  \end{equation}
  \begin{equation}\label{bootmodu2re}
  \begin{aligned}
&  |\frac{d\tilde{\gamma}_{j}}{ds_{j}}-\frac{1}{|Q_{1}|^{2}}(\eps^{j}_{1},L_{+}Q_{2}|+|\frac{1}{\lambda_{j}}\xs|\\
  \leq &\delta(\alpha)(\int |\nabla\eps^{j}|^{2}e^{-2(1-\eta)\frac{\theta(b_{j}|y|)}{b_{j}}}+\int |\eps^{j}|^{2}e^{-|y|}|)^{\frac{1}{2}}+C\int |\nabla \epsilon^{j}|^{2}+\Gamma_{b_{j}}^{1-C\eta},
  \end{aligned}
  \end{equation}
   \begin{equation}\label{energyestimate1re}
 |2(\eps^{j}_{1},\Sigma_{b_{j}})+2(\eps^{j}_{2},\Theta_{b_{j}})|\leq C\left(\int |\nabla\eps^{j}|^{2}+|\eps^{j}|^{2}e^{-|y|}\right)+\Gamma_{b_{j}}^{1-C\eta},
 \end{equation}
 \begin{equation}\label{momentumestimate1re}
 |(\epsilon^{j}_{2},\nabla\Sigma_{b_{j}})|\leq C\delta(\alpha)(\int |\nabla \eps^{j}|^{2}+\int |\eps^{j}|^{2}e^{-|y|}).
 \end{equation}
\end{lem}
\begin{proof}
We just need to  combine Lemma \ref{lemmoduationestimate}, Lemma \ref{estiamtebyconservationlaw}, Lemma \ref{control of local quantity}.
\end{proof}

\subsubsection{Local virial estimate and Lypounov functional control}
Below we combine the orthogonality conditions \eqref{oth1}, \eqref{oth2}, \eqref{oth3}, \eqref{oth4}, modulation estimates \eqref{bootmodu1re}, \eqref{bootmodu2re} and estimates induced by the (local) conservation laws \eqref{energyestimate1re}, \eqref{momentumestimate1re},  following the work of  Merle and Rapha\"el to obtain the analogue of Lemma 7 in  \cite{merle2006sharp}.
\begin{lem}[local virial]
In the setting of Lemma \ref{bootstrap}, the following estimates hold  for $t\in [0,T]$, and $j=1,\dots, m$
\begin{equation}\label{tslocalvirial}
\frac{db_{j}}{ds_{j}}\geq \delta_{1}(\enj+\lmj)-\Gamma_{b_{j}}^{1-C\eta},
\end{equation}
where  $\delta_{1}$ is a universal constant.
\end{lem}
\begin{proof}
This  lemma is highly nontrivial, and it  is actually one of the key elements in the work of  Merle and Rapha\"el. However, by applying  exactly the same argument, which they used to derive Lemma 7 in \cite{merle2006sharp}, one can derive
\begin{equation}\label{templocalvirial}
\frac{db_{j}}{ds_{j}}\geq \delta_{1}(\enj+\lmj)-\Gamma_{b_{j}}^{1-C\eta}-C|E(\qbj+\epj)|, j=1,\dots, m.
\end{equation}
Now using estimate \eqref{energycontrol} in Lemma \ref{control of local quantity}, one obtains $|E(\qbj+\epj)|\lesssim \lambda_{j}^{2-}$, which  is negligible  compared to $\Gamma_{b_{j}}^{1-C\eta}$, since we have \eqref{assumeloglog1}. Thus ,\eqref{tslocalvirial} clearly follows from \eqref{templocalvirial}. 
\end{proof}
Next, we recover the Lyaponouv functional, which is essential to establish the sharp log-log regime. This is the analogue of Proposition 8 in 
\cite{merle2006sharp}. 
\begin{lem}\label{lyplem}
\setting, and for j=1,\dots, m
\begin{equation}\label{bootlyp}
\frac{d\JJJ_{j}}{ds_{j}}\leq -Cb_{j}\{\Gamma_{b_{j}}+\tailj +\int_{A_{j}}^{2A_{j}}|\epj|^{2}\}, 
\end{equation}
with
\begin{equation}
\begin{aligned}
\JJJ_{j}:=\left(\int |\qbj|^{2}-\int Q^{2}\right)+2(\epjo,\Sigma_{b_{j}})+2(\epjt,\Theta_{b_{j}})+\int(1-\phi_{A_{j}})|\eps^{j}|^{2}\\
-\frac{\delta_{1}}{800}\left(b\tilde{f}_{1}(b_{j})-\int_{0}^{b_{j}}\tilde{f}_{1}(v)dv+b\{(\epjt,\Lambda\Re\tzj)-(\epjo,\Lambda\Im\tzj)\}\right),
\end{aligned}
\end{equation}
where
\begin{equation}
\tilde{f}_{1}(b):=\frac{b}{4}|y\qbb|_{2}^{2}+\frac{1}{2}\Im (\int y\nabla \tilde{\zeta}_{b}\bar{\tilde{\zeta}}_{b}),
\end{equation}
\begin{equation}
\tilde{\eps}^{j}=\eps^{j}-\tilde{\zeta}_{b_{j}},
\end{equation}
and $\phi_{A_{j}}$ is a non-negative  smooth cut-off function, $ j=1,\dots, m$
\begin{equation}
\begin{cases}
\phi_{A_{j}}(x)=0,  |x|\leq \frac{A_{j}}{2},\\
\frac{1}{4_{A_{j}}}\leq |\nabla \phi_{A_{j}}|\leq |\frac{1}{2A_{j}}|, A_{j}\leq |x|\leq 2A_{j},\\
\phi_{A_{j}}(x)=1, |x|\geq  3A_{j},
\end{cases}
\end{equation}
where $A_{j}=A(b_{j})=\Gamma_{b_{j}}^{-a}.$
\end{lem}
Most parts of the  proof  follows directly the proof of Proposition 8 in \cite{merle2006sharp}, however, some extra technical elements  need to be treated hence for completeness  we will explain the proof in Subsection \ref{lyapounouv}.

We have the following control for the scale of the Lyapounouv functional:
\begin{lem}\label{b2}
\setting, and for $j=1,\dots, m$,
\begin{equation}\label{bb2}
\frac{\JJJ_{j}}{b_{j}^{2}}=C^{*}(1+O(\delta(\alpha)),
\end{equation}
where $C^{*}$ is some fixed constant.
\end{lem}
\begin{proof}
Follow the proof of (5.15) in \cite {merle2006sharp}, one can derive $\JJJ_{j}\sim b_{j}^{2}$. Further refined analysis, will give \eqref{bb2}, see  the formula between (5.24) and (5.25) in \cite{merle2006sharp}.
\end{proof}
\subsubsection{Bootstrap estimates except \eqref{estimateoutsidesmoothness}}
So far, we already have all the ingredients to prove most of the the bootstrap estimates. In fact for  \eqref{estimateperturbation}, \eqref{estimateloglog1}, \eqref{estimateloglog2},\eqref{estimateofsmalllocalerror}, \eqref{estimatealmostmono}, \eqref{estimatesmalltranslation} one can follow the  arguments  of Planchon and Rapha\"el  in \cite{planchon2007existence}, which we will  review  for completeness  in Subsection \ref{boot}. Here we  prove instead \eqref{estimatelocalcontrolofconservedquantity}, \eqref{estimatelocalcontrolofmomentum}. Actually we  only show the details for  \eqref{estimatelocalcontrolofconservedquantity}, since  \eqref{estimatelocalcontrolofmomentum} is similar.\\
\begin{proof}[Proof of \eqref{estimatelocalcontrolofconservedquantity}]
A direct computation (using $u$ that solves \eqref{nls}) shows that for $j=1,\dots, m$
\begin{equation}
\begin{aligned}
&\lvert E_{loc}(x_{j}(t),u(t))-E(x_{j,0})\rvert\\
\leq &\int_{0}^{t}\lvert\frac{d}{d\tau}E_{loc}(x_{j}(\tau), u(t))\rvert d\tau\\
\leq & \int_{0}^{T}\lvert\frac{d}{dt}E_{loc}(x_{j}(t),u(t))\rvert dt\\
\leq &E_{1}+E_{2},
\end{aligned}
\end{equation}
where 
\begin{eqnarray}
&&E_{1}=\int_{0}^{T}\int_{\RRR^{2}}\rvert\left(\frac{1}{2}|\nabla u|^{2}-\frac{1}{4}|u|^{4}\right)\nabla \chi_{0,loc}(x-x_{j}(t))\frac{dx_{j}}{dt}\lvert,\\
&&E_{2}=\int_{0}^{T}\int \frac{1}{2}\lvert\nabla \chi_{0,loc}(x-x_{j}(t)) \left(2i\Im\Delta u\nabla \bar{u}+2i\Im |u|^{2}u\nabla \bar{u}\right).\rvert.
\end{eqnarray}
Recall  that $\chi_{0, loc}(x-x_{j})$ vanishes for $\frac{3}{4}\leq |x-x_{j}|\leq 1$. Using the  bootstrap hypothesis \eqref{assumesmalltranslation}, and  Sobolev embedding, we obtain
\begin{equation}
\begin{aligned}
&\int\rvert\left(\frac{1}{2}|\nabla u|^{2}-\frac{1}{4}|u|^{4}\right)\nabla \chi_{0}(x-x_{j}(t))\frac{dx_{j}}{dt}\lvert\lesssim \|u\|^{4}_{H^{N}(\min_{j}\{|x-x_{j,0}|\}\geq \frac{1}{2})},\\
&\int \frac{1}{2}\lvert\nabla \chi_{0}(x-x_{j}) \left(\Delta u\nabla \bar{u}-\nabla u\Delta\bar u+|u|^{2}u\nabla u-|u|^{2}\bar u\nabla u\right)\rvert \lesssim \|u\|^{4}_{H^{N}(\min_{j}\{|x-x_{j,0}\}\geq \frac{1}{2})}.
\end{aligned}
\end{equation}
Thus,
\begin{equation}
E_{1}+E_{2}\lesssim \int_{0}^{T} (1+\frac{d}{dt}x_{j})\|u\|^{4}_{H^{N}(\min_{j}\{|x-x_{j,0}|\}\geq \frac{1}{2})}.
\end{equation}
By the bootstrap hypothesis \eqref{assumeoutsidesmoothness}, we have that 
$$\|u\|^{4}_{H^{N}(\min_{j}{|x-x_{j,0}|}\geq \frac{1}{2})}\leq \frac{1}{\max_{j}\{\lambda_{j,0}\}}.$$ Note also that by the bootstrap hypothesis \eqref{assumealmostmono}, also $ \frac{1}{\max_{j}\{\lambda_{j,0}\}}\lesssim \frac{1}{\lambda_{j}(t)}^{\frac{1}{2}}$ for all $j$.\\
By the modulation estimate \eqref{bootmodu2re}, $|\frac{1}{\lambda_{j}}\frac{dx_{j}}{ds_{j}}|\lesssim 1$. Thus, $|\frac{dx_{j}}{dt}|=|\frac{1}{\lambda_{j}^{2}}\frac{dx_{j}}{ds_{j}}|\lesssim \frac{1}{\lambda_{j}}$,
hence
\begin{equation}\label{localenergytech1}
E_{1}+E_{2}\lesssim \int_{0}^{T}\frac{1}{\lambda_{j}(t)^{1.5}}dt\lesssim \lambda_{j,0}^{0.1}\int_{0}^{T}\lambda_{j}(t)^{1.6}.
\end{equation}
Here  we use $\lambda_{j}(t)\lesssim \lambda_{j,0}$, i.e. \eqref{assumealmostmono}.

Finally we have for $j=1,\dots, m$ that 
\begin{equation}\label{localenergytech2}
\int_{0}^{T}\frac{1}{\lambda_{j}^{\mu}}\lesssim C_{\mu}, \forall \mu<2.
\end{equation}
This is the analogue of \eqref{scalefromblowupspeed}, and it  follows exactly the same proof that only relies on the bootstrap hypothesis \eqref{assumeloglog1}, \eqref{assumeloglog2}.
Clearly \eqref{localenergytech1} and  \eqref{localenergytech2} end the proof.
\end{proof}

\subsubsection{Propagation of regularity and end of bootstrap estimate}
To end the bootstrap estimate, we still have to show \eqref{estimateoutsidesmoothness}, and it  will be done in Subsection \ref{propagation}.

\subsection{Recovering  the Lyapounouv functional under bootstrap hypothesis}\label{lyapounouv}
This subsection is devoted to the proof of Lemma \ref{lyplem}. 
We emphasize again that  this proof follows the computation in  \cite{merle2006sharp}  except for two points:
\begin{itemize}
\item We do not use the global energy, and all those $\lambda^{2}E$ terms in \cite{merle2006sharp} are  replaced by $E(\epj+\qbj)$.
\item  In the definition of the original Lyapounov functional $\JJJ$ in \eqref{lyp}, there is a term $\|\qbb\|_{2}^{2}-\|Q\|_{2}^{2}+(\epsilon_{1},\Sigma)+(\epsilon_{2},\Theta)+\|\epsilon\|_{2}^{2}$, that is actually a constant  in \cite{merle2006sharp}, thanks to the conservation of $L^{2}$ mass. In our case, since we  analyze the dynamics locally, the natural substitution is the  local mass, which is no longer a constant. So we need to show that the  local mass is slowly varying.
\end{itemize}

We need the following two lemmas.
\begin{lem}[analogue of Lemma 6 in \cite{merle2006sharp}]
\setting \, for some universal constant $\delta_{1}$ and $j=1,\dots, m$
\begin{equation}\label{bootvirial}
\frac{d}{ds_{j}}f_{1}^{j}\geq \delta_{1}\left(\tailj\right)+c\Gamma_{b_{j}}-\frac{1}{\delta_{1}}\int_{A_{j}}^{2A_{j}}|\epj|^{2},
\end{equation}
with
\begin{equation}
f^{j}_{1}(s)=\frac{b_{j}}{4}\|y\qbj\|_{2}^{2}+\frac{1}{2}\Im \left(y\nabla \tzj\bar{\tilde{\zeta}}_{b_{j}}\right)+(\epj_{2},\Lambda\Re\tzj)-(\epj_{1},\Lambda\Im\tzj).
\end{equation}
\end{lem}
\begin{proof}
This lemma is  one of the most fundamental  points in \cite{merle2006sharp}.
We quickly recall its proof. If $\frac{1}{\lambda(t)}(Q_{b}+\epsilon)(\frac{x-x(t)}{\lambda(t)})e^{-i\gamma}$ solves \eqref{nls}, then one is able  to derive two equations for  $\epsilon:=\epsilon_{1}+i\epsilon_{2}$, i.e. \eqref{sys1}, \eqref{sys2}.

Then one takes  the inner product of \eqref{sys2} with $\Lambda(\Sigma_{b}+\Lambda\Re \tilde{\zeta})$ and of \eqref{sys2} with $-\Lambda(-\Theta_{b}+\Lambda\Im\tilde{\zeta})$ and sums. The detailed computation is displayed in Appendix B of \cite{merle2006sharp}.\\
Here we follow exactly  the same procedure. We pick $j=1,\dots, m$ and we recall that by definition, 
$$\frac{1}{\lambda_{j}}(\qbj+\epj)(\frac{x-x_{j}}{\lambda_{j}})e^{-i\gamma_{j}}=\chi_{1,loc}(x-x_{j})u,$$
 and $\chi_{1,loc}(x-x_{j})u$ almost solves \eqref{nls}. One may derive  similar equations for $\epj(y)=\eps^{j}_{1}(y)+i\eps_{j}(y)$ as in \eqref{sys1}, \eqref{sys2} with some extra terms in right hand side.  Since $\chi_{1,loc}(x-x_{j})\equiv 1$ in $|x-x_{j}|\leq \frac{2}{3}$, these extra terms are  supported in $|y|\geq \frac{1}{2\lambda_{j}}$ (i.e. $|x-x_{j}|\geq \frac{1}{2}$). Since  $\qbj,\tzj$ is supported in $|y|\leq \Gamma_{b_{j}}^{-a}$ and by the  bootstrap hypothesis \eqref{assumeloglog1}, $\Gamma_{b_{j}}^{-a}\ll \frac{1}{\lambda_{j}}$, then when one pairs these equations with $\Lambda(\Sigma_{b_{j}}+\Lambda\Re \tzj)$ or $-\Lambda(-\Theta_{b_{j}}+\Lambda\Im{\tzj})$, these extra  terms  automatically cancel.   \\
 As a consequence all the algebraic computation in the appendix B of \cite{merle2006sharp}  follow.
There is one more difference compared to  Appendix B of \cite{merle2006sharp}. There the authors use the energy  conservation, (formula(4.17) in \cite{merle2006sharp}), here instead we need to replace  term  $\lambda^{2}E_{0}$  in that formula by the term $E(\epj+\qbj)$. Once this is done we follow the argument in \cite{merle2006sharp}  to recover the virial type estimate
\begin{equation}
\frac{d}{ds_{j}}f_{1}^{j}\geq \delta_{1}\left(\tailj\right)+c\Gamma_{b_{j}}-\frac{1}{\delta_{1}}\int_{A_{j}}^{2A_{j}}|\epj|^{2}-CE(\eps^{j}+\qbj), j=1,\dots, m.
\end{equation}  
By Lemma \ref{control of local quantity} and the bootstrap hypothesis  \eqref{assumeloglog1}, we obtain
\begin{equation}
|E(\epj+\qbj)|\lesssim \lambda_{j}^{2-}\ll \Gamma_{b_{j}}^{10},
\end{equation}
and as a consequence  formula \eqref{bootvirial}.
\end{proof}
We have the following Lemma.
\begin{lem}[analogue of Lemma 7 in \cite{merle2006sharp}]\label{bootradiation}
\setting, and for $j=1,\dots, m$
\begin{equation}
\frac{d}{ds_{j}}\{\int |\phi_{A_{j}}|\epj|^{2}\}\geq \frac{b_{j}}{400}\int_{A_{j}}^{2A_{j}}|\epsilon|^{2}-\Gamma_{b_{j}}^{1+z_{0}}-\Gamma_{b_{j}}^{\frac{a}{2}}\int |\nabla \epj|^{2}.
\end{equation}
\end{lem}
\begin{proof}
The proof follows as the proof of Lemma 7 in \cite{merle2006sharp}  except, as above, for replacing  $\lambda^{2}E$ with  $E(\epj+\qbj)$, which is much smaller than $\Gamma_{b_{j}}^{100}$ by the bootstrap assumption, and thus negligible.
\end{proof}
The following lemma illustrates the fact that the local mass is slowly varying with respect to the rescaled time variable $s_{j}$.

\begin{lem}[Slow varying of local mass]\label{localmassslowvarying}
With the same assumptions as in  Lemma \ref{bootstrap}, the following estimate hold
\begin{equation}
|\frac{d}{ds_{j}}\{\int |\qbj|^{2}+2(\epj_{1},\Sigma_{b_{j}})+2(\epjt,\Theta_{b_{j}})+\int |\epj|^{2} \}|\leq \Gamma_{b_{j}}^{10}.
\end{equation}
\end{lem}
\begin{proof}
First we observe that 
\begin{equation}
\{\int |\qbb|^{2}+2(\epsilon_{1},\Sigma)+2(\epsilon_{2},\Theta)+\int |\epsilon|^{2}\}\equiv \|\qbj+\epj\|_{2}^{2}.
\end{equation}
 By \eqref{geodef}, one has the geometric decomposition
\begin{equation}
u(t,x)=\sum_{j=1}^{m}\frac{1}{\lambda_{j}}\qbj(\frac{x-x_{j}}{\lambda_{j}})e^{-i\gamma_{j}}+\Xi.
\end{equation}
Also recall that $\epj(y)=\chi_{1,loc}(\lambda_{j} y)\epsilon^{j}(y)$ and that 
 $\qbj$ is supported in $|x|\lesssim \frac{2}{b_{j}}$. Note that by the bootstrap assumption \eqref{assumeloglog1}, $\frac{1}{\lambda_{j}}\gg\frac{1}{b_{j}}$. 
We then  obtain
\begin{equation}
\begin{aligned}
 &\|\qbj+\epj\|_{2}^{2}\\
=&\|\qbj\chi_{1,loc}(\lambda y)+\epsilon^{j}\chi_{1,loc}(\lambda y)\|_{2}^{2}
=&\|u(\chi_{1, oc}(x-x_{j}))\|_{2}^{2}.
\end{aligned}
\end{equation}
Thus,
\begin{equation}
\begin{aligned}
&\lvert\frac{d}{ds_{j}}\{\int |\qbj|^{2}+2(\epsilon^{j}_{1},\Sigma_{b_{j}})+2(\epsilon^{j}_{2},\Theta_{b_{j}})+\int |\epsilon|^{2}\}\rvert\\
\lesssim &|\frac{dx_{j}}{ds_{j}}|\|u\|_{2}^{2}
+2\lambda_{j}^{2}\int_{\frac{2}{3}\leq |x-x_{j}|\leq \frac{3}{4}}|\nabla \chi_{1,loc}||u||\nabla u|.
\end{aligned}
\end{equation}
 Recall  now that $\frac{d}{ds_{j}}=\lambda_{j}^{2}\frac{d}{dt}$. The first term on left hand side is controlled by the modulation estimate \eqref{bootmodu2re}, and the conservation of mass $\|u(t)\|_{2}=\|u_{0}\|_{2}$,
\begin{equation}
|\frac{dx_{j}}{ds_{j}}|\|u\|_{2}^{2}\lesssim |\frac{1}{\lambda_{j}}\xs|\lambda_{j}.
\end{equation}
The second term on the left hand side is controlled by the bootstrap assumption \eqref{assumeoutsidesmoothness}, 
\begin{equation}
2\lambda_{j}^{2}\int_{\frac{2}{3}\leq |x-x_{j}|\leq \frac{3}{4}}|\nabla \chi_{1,loc}||u||\nabla u|\lesssim \lambda_{j}^{2}.
\end{equation}
We finally use the bootstrap assumption \eqref{assumeloglog1} to control $\lambda_{j}$ and the desired estimate easily follows.
\end{proof}
Now we are in good position to finish the proof of Lemma \ref{bootlyp}.
\begin{proof}[proof of Lemma \ref{bootlyp}]
With lemma \ref{bootvirial} and lemma \ref{bootradiation}, we follow the proof of Proposition 4 in \cite{merle2006sharp} and we  obtain
\begin{equation}
\begin{aligned}
\frac{d}{ds_{j}}\left(\int-\phi_{A}|\eps^{j}|^{2}
-\frac{\delta_{1}}{800}\left(b\tilde{f}_{1}(b_{j})-\int_{0}^{b_{j}}\tilde{f}_{1}(v)dv+b\{(\epjt,\Lambda\Re\tzj)-(\epjo,\Lambda\Im\tzj)\}\right)\right)\\
\leq -C\left(\int |\nabla\eps^{j}|^{2}+|\eps^{j}|^{2}e^{-|y|}\right)+\Gamma_{b_{j}}^{1-C\eta}.
\end{aligned}
\end{equation}
Now plug in  the estimate in Lemma \ref{localmassslowvarying} above and the  desired estimate follows.\\
\end{proof}

\subsection{Propagation of regularity under bootstrap hypothesis}\label{propagation}
This subsection mostly follows from the arguments in Section \ref{proofofos}, indeed, the reason why we write Section \ref{proofofos} is to make this subsection more accessible.\\
First, the analogue of \eqref{scalefromblowupspeed} holds,
\begin{lem}
\setting 
\begin{equation}\label{tsscalefromblowuprate}
\int_{0}^{t}\left(\sum_{j=1}^{m}\frac{1}{{\lambda_{j}^{\mu}(\tau)}}\right)d\tau
\lesssim
\begin{cases}
C_{\mu}, \mu<2,\\
\sum_{j=1}^{m}\frac{|\ln \lambda_{j}(t)|^{101}}{\lambda_{j}(t)^{\mu-2}}, \mu \geq 2.
\end{cases}
\end{equation}
\end{lem}
\begin{proof}
The proof of \eqref{tsscalefromblowuprate} only relies on the bootstrap hypothesis \eqref{assumeloglog1}, \eqref{assumeloglog2}, and it is similar to the proof of \eqref{scalefromblowupspeed}.
\end{proof}
We again introduce a sequence of cut-off functions\footnote{We still use the notation of $\chi$, but the definition of $\chi$ is different from Section \ref{proofofos}.}   $\{\chi_{l}\}_{l=0}^{L}$
\begin{equation}
\chi_{l}=
\begin{cases}
0, |x-x_{j,0}|\leq a_{l}, j=1 \text{ or } 2,\\
1, |x-x_{j,0}|\geq b_{l}, j=1 \text{ and } 2,
\end{cases}
\end{equation}
such that  $a_{l}<b_{l}<a_{l-1}, a_{0}=\frac{1}{250}\ll 1, b_{L}=\frac{2}{5}<\frac{1}{2}.$\\
The idea is ( again ) that we want  to retreat from $\chi_{l-1}u$ to $\chi_{l}u$ for each $l$, showing that $\chi_{l}u$ has higher regularity than $\chi_{l-1}u$. We still  use the notation $v_{i}=\chi_{i}u$. \\
The key to gain regularity in Section \ref{proofofos} is formula \eqref{origingain}. Similar estimates also holds here:
\begin{lem}
\setting,
\begin{equation}\label{tsinitialgain}
\int_{0}^{T} \|\nabla (\chi_{0}u(t))\|^{2}dt\lesssim 1.\\
\end{equation}
\end{lem}
\begin{proof}
First recall that by the bootstrap hypothesis \eqref{assumesmalltranslation}, if $x\in supp\chi_{0}$, then $\min_{j}\{|x-x_{j}|\}\geq \frac{1}{500}$. Thus we need only to control 
$$\int_{0}^{T}\int_{\min_{j}\{|x-x_{j}|\}\geq \frac{1}{500}} \|\nabla (\chi_{0}u)\|_{2}^{2}.$$
Secondly, since $u$ is bounded in $L^{2}$ and $T$ is bounded (indeed $T<\delta(\alpha)$, see Remark \ref{finite}, we need only  to control 
$$\int_{0}^{T}\int_{\min_{j}\{|x-x_{j}|\}\geq \frac{1}{500}} \|\nabla u\|_{2}^{2}.$$
Indeed
\begin{equation}\label{initialganstep1}
\begin{aligned}
 &\int_{0}^{T}\int_{\min_{j}\{|x-x_{j}|\}\geq \frac{1}{500}} |\nabla u|^{2}\\
\leq&\int_{0}^{T}\int_{\min_{j}\{|x-x_{j}|\}\geq \frac{3}{5}}|\nabla u|^{2}+\int_{0}^{T}\sum_{j=1}^{m}\int_{ \frac{1}{500}\leq |x-x_{j}|\leq \frac{3}{5}} |\nabla u|^{2}\\
:=&E_{1}+E_{2}.
\end{aligned}
\end{equation}
The first term is controlled  by the bootstrap hypothesis \eqref{assumeoutsidesmoothness} and \eqref{assumealmostmono}. In fact  we easily have:
\begin{equation}\label{firsterror}
\begin{aligned}
E_{1}\lesssim &\int_{0}^{T}\min_{j}\{|\ln \lambda_{j,0}|\}\\
        \lesssim &\int_{0}^{T} \sum_{j} |\ln \lambda_{j}(t)|\\
        \lesssim &1
\end{aligned}
\end{equation}
In last step we have applies \eqref{tsscalefromblowuprate}.
The second term $E_{2}$ is  estimated by using the Lyapounov functional $\JJJ_{j}$ in  \eqref{bootlyp}. This is actually one of the key estimates in \cite{merle2005profiles}. We recall it here.
Note that 
$$\chi_{1,loc}(x)\equiv 1 \text{ for  } |x|\leq \frac{2}{3} \quad \mbox{ and  } \quad  \chi_{1,loc}(x-x_{j})=\frac{1}{\lambda_{j}}(\qbj+\epj)(\frac{x-x_{j}}{\lambda_{j}})e^{-i\gamma_{j}}.$$
 We make the observation that since $\lambda_{j}\lesssim  e^{-e^{\frac{\pi}{10b_{j}}}}$ by the bootstrap hypothesis \eqref{assumeloglog1},  when $|x-x_{j}|\geq \frac{1}{500}$ we have  $\qbj(\frac{|x-x_{j}|}{\lambda_{j}})\equiv 0,$ since $supp \qbj \subset\{|y|\lesssim \frac{1}{b_{j}}\}$. Similarly, in this region the radiation term $\tilde{\zeta}_{b_{j}}(\frac{x-x_{j}}{\lambda_{j}}) $ also vanishes. Thus, for $j=1,\dots, m$
\begin{equation}
\begin{aligned}
 E_{2}= &\int_{\frac{1}{500}\leq |x-x_{j}|\leq \frac{3}{5}}|\nabla u|^{2}
=\int_{\frac{1}{500}\leq |x-x_{j}|\leq \frac{3}{5}}|\nabla  \frac{1}{\lambda_{j}}(\qbj+\epj)(\frac{x-x_{j}}{\lambda_{j}})e^{-i\gamma_{j}} |^{2}\\
=&\lambda_{j}^{-2}\int_{\frac{1}{500\lambda_{j}}\leq |x|\leq \frac{3}{5\lambda_{j}}}|\nabla \epj|^{2}
=\lambda_{j}^{-2}\int_{\frac{1}{500\lambda_{j}}\leq |x|\leq \frac{3}{5\lambda_{j}}}|\nabla \tilde{\eps}^{j}|^{2}.
\end{aligned}
\end{equation}
As a consequence
\begin{equation}
\begin{aligned}
  E_{2}
\leq &\int_{s_{j}(t_{1})}^{s_{j}(t_{2})}|\nabla \tilde{\eps}^{j}|^{2}
\lesssim sup_{t\in [0,T)} \sqrt{\JJJ_{j}}\sim b_{j}\lesssim 1.
\end{aligned}
\end{equation}
In the last step, we use \eqref{bootlyp} and the fact that the Lyapounov functional $\JJJ_{j}\sim b_{j}^{2}$, see Lemma \ref{b2}.
This is enough to end the proof.
\end{proof}
We now again  use I-method to recover a rough control:
\begin{lem}
\setting,
\begin{equation}\label{roughcontrolts}
\|u(t)\|_{H^{N_{2}}}\lesssim_{\sigma} \left(\sum_{j}\frac{1}{\lambda_{j}(t)}\right)^{N_{2}+\sigma}.
\end{equation}
\end{lem}
\begin{proof}
Without loss of generality, we only show \eqref{roughcontrolts} for $t=T $.
We remark that if one defines $\tilde{\lambda}=\min_{j} \{\lambda_{j}\}$, then $\|u(t)\|_{H^{1}}\sim \frac{1}{\tilde\lambda(T)}$, and \eqref{roughcontrolts} is equivalent to $\|u(t)\sim \left(\frac{1}{\tilde{\lambda}(t)}\right)^{N_{2}+\sigma}$.\\
The proof is almost the same as the proof  of \eqref{goal} in Lemma \ref{lemosroughcontrol}. Following Remark \ref{remrough} we only need  to show that we can divide $[0,T]$ into disjoint intervals $\cup {I_{k,h}}$, such that
\begin{itemize}
\item $\|u(t)\|_{H^{1}}\sim 2^{k}$ for $t\in I_{k,j}$
\item $|I_{k,h}|\sim \frac{1}{2^{2k}}$, $\forall k,h$.
\item $\sharp \{I_{k,h}\}\lesssim (|\ln \tl(T)|)^{3}.$
\end{itemize}
These  clearly follow from the  facts below.  Fix $j$, we can divide $[0,T]$ into disjoint intervals $\cup {I^{j}_{k,h}}$, such that 
\begin{itemize}
\item $\lambda_{j}(t) \sim 2^{-k}$ for $t\in I_{k,j}$
\item $|I^{j}_{k,h}|\sim \frac{1}{2^{2k}}$
\item $\sharp \{I^{j}_{k,h}\}\lesssim (|\ln \lambda_{j}(T)|)^{3}$
\end{itemize}
Now we prove these facts. Then by the bootstrap hypothesis \eqref{assumealmostmono}, we can divide $[0,T]$ into $0=t_{k_{0}}<...t_{k}<...t_{K(T)}=T$ such that $\lambda_{1}(t)\sim 2^{-k} $for $t\in[t_{k},t_{k+1}]$, then relying on the  bootstrap hypothesis \eqref{assumeloglog1}, \eqref{assumeloglog2}, similarly as in  the proof of Lemma \ref{intevallength}, one can  show that  $t_{k+1}-t_{k}\leq \sqrt{k}\lambda_{1}^{2(t_{k})}\sim \sqrt{k}2^{-2k}$. This estimate is enough for us to further divide $[t_{k},t_{k+1}]$ into disjoint intervals $\cup_{h=1}^{J_{h}}[\tau_{k}^{h},\tau_{k}^{h+1}]$ such that our desired $I_{k,h}$ can be chosen as $[\tau_{k}^{h},\tau_{k}^{h+1}]$, see \eqref{intevaldecompose},\eqref{numberofiteration}.
\end{proof}
We now remark that the bootstrap estimate \eqref{estimateoutsidesmoothness} follows from the lemma below.
\begin{lem}
\setting, 
\begin{enumerate}
\item $\forall \nu<1$, $\| v_{1}\|_{H^{\nu}}\lesssim 1$.
\item If there holds estimate for some $r>0$, and some $1\leq i\leq L-1$
\begin{equation}
\|v_{i}\|_{H^{r}}\lesssim 1,
\end{equation}
then we will have a gain of regularity on $v_{i+1}$,
\begin{equation}
\forall \tilde{r}<\frac{2(K_{2}-r)}{K_{2}}-1+r, \|v_{i}\|_{H^{\tilde{r}}}\lesssim 1.\\
\end{equation}
\end{enumerate}
\end{lem}
The proof  follows as the proof of Lemma \ref{lemosfirstgain}, Lemma \ref{lemosimprovelinfty}, Lemma \ref{lemosfirstgainimpr}, Lemma \ref{lemositeration}, see also Proof of \eqref{K1smooth} in Lemma \ref{osboot}.

\subsection{Proof of bootstrap estimate except \eqref{estimateoutsidesmoothness}, \eqref{estimatelocalcontrolofconservedquantity},\eqref{estimatelocalcontrolofmomentum}}\label{boot}
The proof of these estimates can be found (up to a small modification) in \cite{planchon2007existence}, \cite{raphael2009standing} and the references therein. See in particular Proposition 1 in \cite{raphael2009standing}. We quickly review those estimates for the convenience of the readers.\\
\subsubsection{Proof of \eqref{estimateperturbation}}
We bound the left hand  side of \eqref{estimateperturbation} by $C\alpha$.
The control of $b_{j}$ directly follows from the bootstrap hypothesis \eqref{assumeloglog2}, which implies $0\leq b_{j}\leq \frac{10\pi}{\ln s_{j,0}}\lesssim b_{j,0}$ and we note that by the initial condition \eqref{initailsigh}, we have $b_{j,0}\leq \alpha$. Thus $|b_{j}|\lesssim \alpha\ll \alpha$.
The control of $\|\nabla \epj\|_{2}$ follows from the bootstrap hypothesis \eqref{assumesmallnessoflocalerror}, therefore we have $\|\nabla \epj\|_{2}\lesssim \gamma_{b_{j}}^{\frac{1}{100}}\ll \alpha$.
The control of $\|\epj\|_{2}$ comes from the $L^{2}$ conservation law. Indeed $$\sum_{j=1}^{m}\|\eps^{j}\|_{2}^{2}+\sum_{j=1}^{m}\|\qbj\|_{2}^{2}\leq \|u\|_{2}^{2}+O\left(\sum_{j=1}^{m}(|\eps^{j}|,|\qbj|)\right).$$
Note  that 
\begin{eqnarray}
(|\eps^{j}|,|\qbj|)\leq  \enj+ \lmj,\\
\|\qbj\|_{2}^{2}=\|Q\|_{2}^{2}+O(|b_{j}|^{2})\\
\|u\|_{2}^{2}=\|u_{0}\|_{2}^{2}
\end{eqnarray}
Thus the control of $\|\nabla \eps^{j}\|_{2}$ follows easily from the bootstrap hypothesis \eqref{assumesmallnessoflocalerror} and our choice of initial data.\\
\subsubsection{Proof of \eqref{estimateloglog1},\eqref{estimateloglog2}}
We first show \eqref{estimateloglog2}. It follows from the local virial estimate \eqref{tslocalvirial} and  the control  of the Lyapounouv functional  \eqref{bootlyp}.
We first show the  lower bound of \eqref{estimateloglog2}.
Note that by \eqref{tslocalvirial}, one obtains
$
\frac{db_{j}}{ds_{j}}\geq -\Gamma_{b_{j}}^{1-C\eta},
$
which implies
$$
\left(\frac{d e^{\frac{\pi(1-C\eta)}{b_{j}}}}{ds_{j}}\right)\leq 1.
$$
Recall that $s_{j}\geq s_{j,0}\equiv  e^{\frac{3\pi}{4b_{j,0}}}$, thus we have:
\begin{equation}\label{finaltech1}
\begin{aligned}
e^{\frac{\pi(1-C\eta)}{b_{j}}}
&\leq s_{j}-s_{j,0}+e^{\frac{\pi(1-C\eta)}{b_{j,0}}}\leq s_{j}+e^{\frac{\pi(1-C\eta)}{b_{j,0}}}\leq s_{j}^{\frac{4}{3}}.
\end{aligned}
\end{equation}
Therefore one obtains
$b_{j,0}\geq \frac{\pi}{5\ln s_{j}}$, which is the lower bound of \eqref{estimateloglog2}.

Now we turn to the upper bound of \eqref{estimateloglog2}. The control \eqref{bootlyp} for the 
Lypounov functional  directly implies 
\begin{equation}\label{quicklyp}
\frac{d\JJJ_{j}}{ds_{j}}\leq -Cb_{j}\Gamma_{b_{j}}.
\end{equation}
From Lemma \ref{b2}, we know that there exists  a $C_{*}$ such that
\begin{equation}\label{techlower}
\frac{\JJJ_{j}}{b_{j}^{2}}=C^{*}(1+O(\delta(\alpha)).
\end{equation}
Let $g_{j}=\sqrt{\JJJ_{j}}$, then by  combining \eqref{quicklyp} and \eqref{techlower}, one obtains:
\begin{equation}
\begin{aligned}
\frac{dg_{j}}{ds_{j}}=\frac{1}{2}\frac{d\JJJ_{j}}{s_{j}}\frac{1}{g_{j}}\leq -Cg_{j}\times \frac{1}{2}C^{*}e^{\frac{\pi(1+C\eta)\sqrt{C^{*}(1-O(\delta(\alpha))}}{g_{j}}}.
\end{aligned}
\end{equation}
 One obtains
\begin{equation}
\frac{d}{d_{s_{j}}}\left( e^{\frac{\pi\sqrt{C^{*}}(1+O(\delta(\alpha)))}{g_{j}}}  \right)\geq 1.
\end{equation}
Thus 
\begin{equation}\label{finaltech2}
e^{\frac{\pi(1+O(\delta(\alpha)))}{b_{j}}}\geq s_{j}-s_{j,0}+e^{\frac{\pi(1+O(\delta(\alpha)))}{b_{j,0}}}\geq s_{j},
\end{equation}
where we recall that  $s_{j,0}=e^{\frac{3\pi}{4b_{j,0}}}$.
This implies the lower bound of \eqref{estimateloglog2}.

Now we turn to the proof of \eqref{estimateloglog1}. By \eqref{finaltech1}, \eqref{finaltech2}, we have
\begin{equation}
\frac{3\pi}{4\ln s_{j}}\leq b_{j}\leq \frac{4\pi}{3 \ln s_{j}}.
\end{equation}
We only show the upper bound of \eqref{estimateloglog1}, the lower bound will be similar.\\
The computation follows from the proof of Lemma 6 in \cite{raphael2006existence},  in fact  formula (2.68) mentioned in that lemma is exactly the upper bound we want.\\
 The modulation estimate \eqref{bootmodu1} plus the bootstrap hypothesis \eqref{assumesmallnessoflocalerror} imply that 
\begin{equation}
|\frac{1}{\lambda_{j}}\frac{d\lambda_{j}}{ds_{j}}+b_{j}|\leq C\Gamma_{b_{j}}^{\frac{1}{2}}.
\end{equation}
Thus
$$\frac{d\lambda_{j}}{ds_{j}}=-(1+O(\delta(\alpha)))b_{j}.$$
and from here 
$$\lambda_{j}((t))\leq \lambda_{j}(s_{j}(t))\equiv \lambda_{j}(s_{j})\leq \lambda_{j,0}+\int_{s_{j,0}}^{s}\frac{1}{2\ln s},$$
(note that according to our notation $\lambda_{j,0}=\lambda_{j}(s_{j,0})$.)\\
A direct calculation  implies that 
$$-\ln \lambda_{j}(s_{j})\geq -\frac{1}{2}\ln \lambda_{j}(s_{j,0})+\frac{\pi}{3}\frac{s}{ln s},$$ which further implies that $s_{j}\lambda(s_{j})\leq e^{-e^{\frac{3\pi}{8b_{j}}}}$. (See (3.106), (3.107) in \cite{raphael2006existence} for more details).\\
This already implies our desired upper bound.\\
\subsubsection{Proof of \eqref{estimateofsmalllocalerror}}
This is exactly the step 3 of the proof of Proposition 5 in \cite{merle2006sharp}, called pointwise control of  $\epsilon$ by $b$.\\

\subsubsection{Proof of \eqref{estimatealmostmono}} The modulation estimate \eqref{bootmodu1} plus the bootstrap hypothesis \eqref{assumesmallnessoflocalerror} imply
\begin{equation}
|\frac{1}{\lambda_{j}}\frac{d\lambda_{j}}{ds_{j}}+b_{j}|\leq C\Gamma_{b_{j}}^{\frac{1}{2}}.
\end{equation}
Thus
$$\frac{d\lambda_{j}}{ds_{j}}=-(1+O(\delta(\alpha)))b_{j}\lambda_{j}<0,$$
(since $b_{j}>0$ due to hypothesis \eqref{assumeloglog2}). 
This clearly implies \eqref{estimatealmostmono}.\\
\subsubsection{Proof of \eqref{estimatesmalltranslation}}\label{subsubsectiontanslation}
This easily follows from the modulation estimate \eqref{bootmodu2} and the bootstrap hypothesis \eqref{assumeloglog1} and \eqref{assumeloglog2}, as in \cite{planchon2007existence}. We quickly review it for completeness.\par
For all $t\in [0,T]$,
\begin{equation}
\begin{aligned}
|x_{j}-x_{j,0}|
&\leq \int_{0}^{t}|\frac{dx_{j}}{dt}|dt\\
&= \int_{0}^{t}\frac{dx_{j}}{ds_{j}}\frac{1}{\lambda_{j}^{2}}dt
&=\int_{s_{j,0}}^{s_{j}(t)}|\frac{1}{\lambda_{j}} \xs|\lambda_{j}ds.
\end{aligned}
\end{equation}
Note that  by \eqref{bootmodu2}, $|\frac{1}{\lambda_{j}}\xs|\leq \delta(\alpha)$, and by \eqref{assumeloglog1} and  \eqref{assumeloglog2}, one has  $\lambda_{j}(s)\leq s_{j}^{-100}$.
Thus one clearly has $|x_{j}-x_{j,0}|\lesssim \delta(\alpha)$, which easily implies  \eqref{estimatesmalltranslation}.
\begin{rem}\label{positionlimit}
The above computation indeed shows
\begin{equation}
\int_{0}^{T^{+}(u)}|\frac{dx_{j}}{dt}|dt<\infty,
\end{equation}
which of course implies $\lim_{t\rightarrow T^{+}}x_{j}(t)$ exists.
\end{rem}

\section{Proof of Main Theorem}\label{proofofmain}
We will need several parameters throughout the whole section.
$$1\gg a_{0}\gg a_{1}\gg a_{2}>0.$$

Recall our goal is to construct log-log blow up solution $u$, which blows up at $m$ prescribed points $x_{1,\infty},...x_{m,\infty}$, and has asymptotic near blow up time as \eqref{loglogblowup}. We will still focus on dimension $d=2$. And since we have scaling symmetry,  we assume  without loss of generality
\begin{equation}\label{noloss}
|x_{j,\infty}-x_{j',\infty}|\geq 20, j\neq j'.
\end{equation}
It is clear that by Lemma \ref{bootstrap}, all initial data $u_{0}$ describe in Subsection \ref{datats}, the associated solution $u$ have the following geometric decomposition for $t< T^{+}(u)$. \begin{equation}
u(t,x)=\sum_{j}\frac{1}{\lambda_{j}(t)}\qbj(\frac{x-x_{j}(t)}{\lambda_{j}})e^{-i\gamma_{j}(t)}+\Xi(t,x),
\end{equation}
 such that for each $j$, $$\chi_{1,loc}(x-x_{j})u(t,x)\equiv \frac{1}{\lambda_{j}}(\qbj+\epj)(\frac{x-x_{j}}{\lambda_{j}})e^{-i\gamma_{j}},$$
and $\Xi \equiv (1-\sum_{j}\chi_{1,loc})u$ in the region $\{x\vert |x-x_{j,{0}}|\geq \frac{1}{2}, j=1,\dots, m\}$ and  bounded in $H^{N_{1}}$ in this region.

To finish the construction in Theorem \ref{thmmain}, we need to construct initial data $u_{0}$ such that
\begin{itemize}
\item (a) The associated solution $u$ blows up in finite time according to the log-log law.
\item (b) The $m$ points blows up simultaneously, i.e. 
\begin{equation}
\lambda_{j}(t)\xrightarrow{t\rightarrow T^{+}} 0, j=1,...,m.
\end{equation}
\item (c) The blow up points are as prescribed, i.e.
\begin{equation}
x_{j}(t)\xrightarrow{t\rightarrow T^{+}} x_{j,\infty}.
\end{equation}
\end{itemize}
It is easy to check once we get (a), (b), (c), then other requirements in the Theorem \ref{thmmain} will be automatically satisfied. To see this , just observe 
\begin{enumerate}
\item $\lambda_{j}\rightarrow 0 \Rightarrow b_{j}\rightarrow 0$ since we have \eqref{estimateloglog1},
\item  Thus, $\qbj$ converges strongly to $Q$ in the sense of \eqref{convergence},
\item $\epj$ is bounded in $L^{2}$ by \eqref{estimateperturbation} and converges to $0$ in $\dot{H}^{1}$ as $b_{j} \rightarrow 0$ by \eqref{estimateofsmalllocalerror}.
\end{enumerate} 
Now let us turn to condition (a), (b), (c). Indeed, Lemma \ref{bootstrap} already implies for all data $u_{0}$, the associated solution $u$ will blow up in finite time with log-log blow up rate
\begin{equation}
\|\nabla u(t)\|_{2}\sim \sqrt{\frac{\ln|\ln T-t|}{T-t}}.
\end{equation}
Not all data described as in Subsection \ref{datats} will give a solution satisfy condition (b), (c). Morally speaking, the initial data in Subsection \ref{datats} has two types of  parameters, $\lambda_{j,0}$ and $x_{j,0}$. $\lambda_{j,0}$ describes how concentrated  the j-th bubble is and the $x_{j,0}$ describes the initial position of $j$th bubble. Lemma \ref{bootstrap} says the m bubbles evolves with very weak interaction with each other. Thus, we need to choose $\lambda_{j}$ carefully to make all the bubbles to blow up at the same time, and choose initial position $x_{j,0}$ to make the bubbles blows up at the prescribed points $x_{j,\infty}$. This is achieved by certain topological argument.  If one does not want to prescribed the blow up points and only wants to get condition (b), an  argument similar to the topological argument in \cite{cote2011construction} will suffice. However, if one wants to further  prescribe the blow up points, then the problem is actually more tricky, since the choice of $\lambda_{j,0}$ and $x_{j,0}$ will be coupled. 

Before we continue, we point out it is not hard to show $\lim_{t\rightarrow T^{+}} x_{j}(t)$ exists. Indeed, as remarked, we actually have
\begin{equation}
\int_{0}^{T^{+}}|\frac{dx_{j}}{dt}|dt<\infty.
\end{equation}
The hard part here is to prescribe the limit, i.e. for given $x_{1,\infty},...x_{m,\infty}$, we want to construct solution such that
\begin{equation}
\lim_{t\rightarrow T^{+}}x_{j}(t)=x_{j,\infty}.
\end{equation}
\subsection{Preparation of data}
Let us first fix a smooth-cut off function
\begin{equation}
\chi(x)=
\begin{cases}
1, |x|\leq 1,\\
0, |x|\geq 2.
\end{cases}
\end{equation}
Let $u_{0}=\frac{1}{\lambda_{0}}\tilde{Q_{b_{0}}}(\frac{x}{\lambda_{0}})+\epsilon_{0}$ as described  in Subsection \ref{osdata}. Note Lemma \ref{originalboot} and Lemma \ref{osboot} hold for those data.
We need to choose $K_{2}$ large enough  for later use. We also require $u_{0}$ be radial.  The associated solution will have the geometric decomposition

\begin{equation}\label{datafamily}
u(t,x)=\frac{1}{\lambda(t)}(\tilde{Q}_{b}(t)+\epsilon(t))(\frac{x}{\lambda(t)})e^{-i\gamma(t)}
\end{equation}
such that the conclusion of Lemma \ref{originalboot} and Lemma \ref{osboot}.
 And in the spirit of  Remark \ref{sharpen}, we may further sharpen the initial condition such that $u$ further satisfies 
 $$\int |\nabla \epsilon|^{2}+|\epsilon|^{2}e^{-|y|}\leq \Gamma_{b}^{1-a_{2}}, \quad e^{-e^{\frac{(1+a_{2})\pi}{b}}}\leq \lambda(t)\leq e^{-e^{\frac{(1-a_{2})\pi}{b}}}.$$
 Recall $a_{2}$ is the parameter we have fixed at the beginning of this Section.
 
Note $\lambda, b, \gamma, \epsilon(t)$ depends on $t$ in an continuous way. $u(t,x)$ can be understood as a family of data continuously depending on $t$.   By  \eqref{simple1} and \eqref{feelingofscale}, we obtain $-\frac{\lambda_{s}}{\lambda}\geq \frac{1}{2}b>0$, since $\lambda_{t}=\lambda^{2}\lambda_{s}<0$. Thus, the map from $t\rightarrow \lambda(t)$ is a homeomorphism. Thus $u(t,x)$ in \eqref{datafamily} can also be understood as a family of data indexed by $\lambda$.
To summarize, we have a family of data index by $\lambda$ small enough,
\begin{equation}\label{oneoneone}
u_{\lambda}(x)=\frac{1}{\lambda}(\tilde{Q}_{b(\lambda)}+\epsilon(\lambda))(\frac{x}{\lambda})e^{-i\gamma(\lambda)},
\end{equation}
with
\begin{equation}\label{oneloglog}
\begin{aligned}
&\lambda+b<\alpha/10,  \quad b>0, \quad e^{e^{-\frac{(1+a_{2})\pi}{b}}}<\lambda<e^{e^{-\frac{(1-a_{2})\pi}{b}}},\\
&\int_{|x|\leq \frac{10}{\lambda(t)}}\|\nabla\epsilon(t)|^{2}+|\epsilon(t)|^{2}\leq \Gamma_{b(t)}^{1-a_{2}},\\
&\lambda^{2}\lvert\int \chi_{0,loc}(x-x(t))\left(\frac{1}{2}|\nabla u(t)|^{2}-\frac{1}{4}|u|^{4}\right)\rvert\leq \Gamma_{b(t)}^{10000}, \\
&\lambda(t)\lvert \Im\int \chi_{0,loc}(x-x(t))\left(\nabla u(t)\bar{u}(t)\right)\rvert\leq \Gamma_{b(t)}^{10000},
\|u(t)\|_{H^{N_{2}}(|x-x_{1,0}|\geq \frac{1}{10000})}\leq \alpha/10,
\end{aligned}
\end{equation}
and $b,\gamma, \epsilon$ depending on $\lambda$ continuously.  Recall $\chi_{0,loc}$ is defined in \eqref{chi0}.

We will consider a family of data
\begin{equation}\label{initialdatafamily}
\begin{aligned}
u_{0,\boldsymbol{\lambda},\boldsymbol{x}}
&\equiv u_{\lambda_{1,0},\lambda_{2,0},..\lambda_{m,0}, x_{1,0},x_{2,0},..x_{m,0}}\\
&\equiv \sum_{j=1}^{m}\chi(x-x_{j,0})u_{\lambda_{j,0}}(x-x_{j,0}).
\end{aligned}
\end{equation}

Here $u_{\lambda}$ is defined as in \eqref{oneoneone}. And we also require $|x_{j,0}-x_{j,\infty}|\leq 1$.  Note this implies $|x_{j,0}-x_{j',0}|\geq 10, j \neq j'$, since we have \eqref{noloss}.

Now let us consider \eqref{nls} with initial data $\fd$.
 Lemma \ref{bootstrap} will work for  $\fd$.  Thus the associated solution $u_{\boldsymbol{\lambda}.\boldsymbol{x}}$ will satisfy the geometric decomposition in its lifespan
\begin{equation}\label{evolvefamiliy}
u_{\boldsymbol{\lambda},\boldsymbol{x}}(t,x)=\frac{1}{\lambda_{j,\llll,\xxxx}}\tilde{Q}_{b_{j,\llll,\xxxx}}(\frac{x-x_{j,\llll,\xxxx}}{\lambda_{j,\llll,\xxxx}})e^{-i\gamma_{j,\llll,\xxxx}}+\Xi_{\llll,\xxxx}
\end{equation}
and all the bootstrap estimates in Lemma \ref{bootstrap} holds.
For notation convenience, we will write $u_{\llll,\xxxx}$ as $u$, write $\lambda_{j,\llll,\xxxx}$ as $\lambda_{j}$, write $x_{j,\llll,\xxxx}$ as $x_{j}$, write $\gamma_{j,\llll,\xxxx}$ as $\gamma_{j}$.

Again, in the spirit  of Remark \ref{sharpen}, we can further sharpen the condition on \eqref{oneloglog}, i.e. make $a_{2}$ small enough, such that
\begin{equation}\label{shapen222}
e^{-e^{\frac{(1+a_{1})\pi}{b_{j}}}}\leq \lambda_{j} \leq e^{-e^{\frac{(1-a_{1})\pi}{b_{j}}}}.
\end{equation}
\subsection{log-log blow up and almost sharp blow up dynamic}\label{sectiones}
We first show the solution $u$ with initial data $u_{\llll,\xxxx}$ blows up according to the log-log law, indeed, we show every bubble itself evolves according to the log-log law, and for later use, we need the almost sharp dynamic, and keep in mind condition our data has already been sharpened so that \eqref{shapen222} holds for the associated solution.
\begin{lem}\label{sharpdynamic}
Let $u$ be the solution with initial data $u_{\llll,\xxxx}$ as in \eqref{initialdatafamily}, then for each $j=1,...m$, there is a $T_{j}$ such that
\begin{eqnarray}\label{blowupestimate}
&&\lambda_{j,0}^{2}\ln |\ln \lambda_{j,0}|= 2\pi \left(1+O(a_{1})\right)T_{j},\\ \label{almostloglog}
&&\lambda_{j}(t)^{2}\ln |\ln \lambda_{j}(t)|= 2\pi\left(1+O(a_{1})\right)(T_{j}-t).
\end{eqnarray} 
In particular, since the blow up rate is modeled by $\min_{j}\lambda_{j}$, we have that the solution blow up in finite time $T^{+}$ and 
\begin{equation}
\|\nabla u(t)\|_{2}\sim \sqrt{\frac{\ln|\ln (T^{+}-t)|}{T^{+}-t}}.
\end{equation}
\end{lem}
\begin{rem}
We implicitly require $\alpha$ to be small enough, as the whole paper.
\end{rem}
\begin{rem}
Note Lemma \ref{bootstrap} says the $m$ bubbles evolves according to the log-log law without really seeing each other. $T_{j}$ is the time that the $j$th bubble which is supposed to blow up. The solution will blow up at $T^{+}=\min_{j}T_{j}$, and the dynamic will be stopped at $T^{+}$. In particular, if $T_{j}>T_{j'}$, then it means $j'$th bubble 'blows' up faster than $j$th bubble, though they may both not blow up.
 \end{rem}
\begin{rem}\label{finite222}
The proof of Lemma \ref{sharpdynamic} needs to use the bootstrap estimate rather than bootstrap hypothesis in Lemma \ref{bootstrap} since we need to get control in term of $O(a_{1})$. It is not hard to see one can argue as the proof below, and with bootstrap hypothesis rather than bootstrap estimate to show 
\begin{eqnarray}\label{blowupestimate1}
&&\lambda_{j,0}^{2}\ln |\ln \lambda_{j,0}|\sim T_{j},\\ \label{almostloglog1}
&&\lambda_{j}(t)^{2}\ln |\ln \lambda_{j}(t)|\sim (T_{j}-t)
\end{eqnarray} 
which in particular shows  the associated solution $u$ blows up in finite time $T^{+}<\delta (\alpha)$.
\end{rem}
\begin{proof}[Proof of Lemma \ref{sharpdynamic}]
We will follow the computation in \cite{merle2006sharp}, which is used to show the exact log-log law. 
\begin{equation}
|\frac{1}{\lambda_{j}}\ls+b_{j}|\leq \Gamma_{b_{j}}^{\frac{1}{2}-C\eta},\\
\end{equation}
which immediately implies
\begin{equation}\label{bcontrolvariation}
(1-\delta(\alpha))b_{j}\leq -\frac{1}{\lambda_{j}}\ls \leq (1+\delta(\alpha))b_{j}.
\end{equation}
Note also \eqref{shapen222} implies
\begin{equation}\label{bcontrol2}
b_{j}=\frac{(1+O(a_{1}))\pi}{\ln |\ln \lambda_{j}|} .
\end{equation}
Thus 
\begin{equation}
\frac{d}{dt}\lambda_{j}^{2}\ln |\ln\lambda_{j}|=2\frac{d\lambda_{j}}{\lambda_{j}^{2}ds}\lambda_{j}\ln| \ln \lambda_{j}| (1+\delta(\alpha))=2\left(1+\delta(\alpha)\right)\frac{d\lambda_{j}}{\lambda_{j}ds_{j}}\ln |\ln \lambda_{j}|.
\end{equation}
Now, plug in \eqref{bcontrolvariation} and \eqref{bcontrol2}, and choosing $\alpha$ small enough, we get
\begin{equation}
\frac{d}{dt}\lambda_{j}^{2}\ln |\ln\lambda_{j}|=-2\pi\left(1+O(a_{1})\right).
\end{equation}
which immediately implies \eqref{blowupestimate}, \eqref{almostloglog}.
\end{proof}
\subsection{A quick discussion of blow up at same time}\label{sametime}
Now, we have a family of data $u_{\llll,\xxxx}$, and the strategy is to adjust parameters to make the $m$ bubbles to blow up at the same time at the prescribed position. If one only wants to make the $m$ bubble to blow up at same time and does not track the final blow up points, then a topological argument similar to the one in \cite{cote2011construction} will be enough. We quickly  illustrate this. We will fix $x_{1,0},...x_{m,0}$ and $\lambda_{1,0}$ . And we will  adjust $\lambda_{2,0},...\lambda_{m,0}$ to make the $m$ bubbles blow up simultaneously.
\begin{lem}\label{lemsim}
Fixed $x_{1,0},...x_{m,0}$ and $\lambda_{1,0}$, ($|x_{j,0}-x_{j',0}|\geq 10, j\neq j'$), there exist $(\beta_{2,0},...\beta_{m,0})\in [(1-a_{0})\lambda_{1,0}, (1+a_{0})\lambda_{1,0}]^{m-1}$ such that,
 the associated solution $u$ to \eqref{nls} with initial data $u_{\lambda_{1,0},\beta_{2,0},..\beta_{m,0}, x_{1,0},...x_{m,0}}$ as in\eqref{datafamily}, will blow up simultaneously at $m$ points, i.e.
 \begin{equation}\label{simul}
 \lambda_{j}(t)\sim \lambda_{j'}(t), t\in [0,T^{+}).
 \end{equation}
\end{lem}
For notation convenience, we write  $u_{\lambda_{1,0},\beta_{2,0},..\beta_{m,0}, x_{1,0},...x_{m,0}}$ as $u_{\beta_{2},...\beta_{m}}$, and we further write  $u_{\beta_{2,0},..\beta_{m,0}}$ as $u_{\boldsymbol{\beta}}.$ We rewrite \eqref{evolvefamiliy} as
\begin{equation}
u_{\bbb}(t,x)=\sum_{j=1}^{m}\frac{1}{\lambda_{j,\bbb}(t)}\tilde{Q}_{b_{j,\boldsymbol{\beta}}(t)}(\frac{x-x_{j,\bbb}(t)}{\lambda_{j,\bbb}})e^{-i\gamma_{j,\bbb}(t)}+\Xi_{\bbb}(t,x).
\end{equation}
Now we prove Lemma \ref{lemsim} by contradiction. Assume \eqref{simul} is not true for any $\bbb\in [(1-a_{0})\lambda_{1,0},(1+a_{0})\lambda_{1,0}]^{m-1}.$
We consider several maps as following:
\begin{itemize}
\item Let $F(t,\bbb)\equiv F(t,\beta_{2},..., \beta_{m})\equiv (\frac{\lambda_{2,\bbb}(t)}{\lambda_{1,\bbb}(t)},\frac{\lambda_{3,\bbb}(t)}{\lambda_{1,\bbb}(t)},...,\frac{\lambda_{m,\bbb}(t)}{\lambda_{1,\bbb}}).$
\item Let $T_{\bbb}$ be the first time $F(t,\bbb)$ hits $\partial [(1-a_{0}), (1+a_{0})]^{m-1}$.
\item Let $G(\bbb)$ be $F(T_{\bbb},\bbb)$.
\end{itemize}
Here $\bbb\in [(1-a_{0})\lambda_{1,0},(1+a_{0})\lambda_{1,0}]^{m-1}$.

Since we assume \eqref{simul} is not true for any $\bbb \in [(1-a_{0})\lambda_{1,0},(1+a_{0})\lambda_{1,0}]^{m-1}$, then $T_{\bbb}$ is always well defined,  i.e. $T_{\beta}<\infty.$  The key point here is that $T_{\bbb}$ depends  on $\bbb$ continuously. Assume this for the  moment and let us  finish the proof by   deriving  a contradiction.
Note that $F$ is clearly continuous (by standard well posedness theory of NLS) . Since we assume that $T_{\bbb}$ depends  on $\bbb$ continuously, $G$ is also continuous. Make the observation $T_{\bbb}=0$ for $\bbb \in \partial [(1-a_{0})\lambda_{1,0},(1+a_{0})\lambda_{1,0}]^{m-1} $, and then it is easy to see $G|_{\partial [(1-a_{0})\lambda_{1,0},(1+a_{0})\lambda_{1,0}]^{m-1}}$ is an heomorphsim from $ \partial [(1-a_{0})\lambda_{1,0},(1+a_{0})\lambda_{1,0}]^{m-1}$ to $\partial [(1-a_{0}), (1+a_{0})]^{m-1}$. 
Then we have constructed a continuous  map from $[(1-a_{0})\lambda_{1,0},(1+a_{0})\lambda_{1,0}]^{m-1}$ to $\partial [(1-a_{0}), (1+a_{0})]^{m-1}$ such that its restriction on $[(1-a_{0})\lambda_{1,0},(1+a_{0})\lambda_{1,0}]^{m-1}$ is a homeomophism, which is clearly false by classical Homology theory. A contradiction!

We need to check $T_{\bbb}$ does depend on $\bbb$ continously. Note  that by LWP of NLS,  the map $F(t,\bbb)$ is continuous and differentiable. Thus, to show that $T_{\bbb}$ depends  on $\bbb$ continuously, we need only to  show that $\partial_{t}F(T_{\bbb},\bbb)$ points outward $\partial [1-a_{0},1+a_{0}]^{m-1}$. In order to show this, without loss of generality, we assume $\frac{\lambda_{2,\bbb}}{\lambda_{1,\bbb}}(T_{\bbb})=1+a_{0}$ and check $$\frac{d}{dt}(\frac{\lambda_{2,\bbb}}{\lambda_{1,\bbb}})(T_{\bbb})>0.$$
This clearly follows from the fact that 
$$\frac{1}{\lambda_{2,\bbb}}\frac{d}{dt}\lambda_{2,\bbb}-\frac{1}{\lambda_{1,\bbb}}\frac{d}{dt}\lambda_{1,\bbb}>0,$$ which is equivalent to 
\begin{equation}\label{revise1131}
-\frac{1}{(\lambda_{1,\bbb})^{3}}\frac{d}{ds_{1}}\lambda_{1,\bbb}>-\frac{1}{(\lambda_{2,\bbb})^{3}}\frac{d}{ds_{2}}\lambda_{2,\bbb},
\end{equation}
 when $t=T_{\bbb}$.
Note that by \eqref{bcontrolvariation} and \eqref{bcontrol2} , (again, we will choose $\alpha $ small enough),
\begin{equation}\label{revise1132}
-\frac{1}{(\lambda_{j,\bbb})^{3}}\frac{d}{ds_{j}}\lambda_{j,\bbb}=(1+O(a_{1})) \frac{1}{(\lambda_{j,\bbb})^{2}}\frac{1}{\ln |\ln \lambda_{j,\bbb}|}.
\end{equation}
 Note $\frac{\lambda_{2,\bbb}}{\lambda_{1,\bbb}}(T_{\bbb})=1+a_{0}$, and $a_{0}\gg a_{1}$,  \eqref{revise1131} follows once one plugs in \eqref{revise1132}.

 This concludes the proof of Lemma \ref{lemsim}.

\subsection{Prescription of blow up points}\label{positionprescription}
To prescribe the blow up points, i.e., to make the solution blows up exactly at the given points $x_{1,\infty},... x_{m,\infty}$ is more tricky.
We will need a topological argument inspired by \cite{merle1992solution}, see also \cite{planchon2007existence}. Morally speaking, we are dealing with a family of initial data with parameters $\lambda_{1,0}, \lambda_{2,0}, ...\lambda_{m,0},x_{1,0},...x_{m,0}$, and the goal is to adjust those parameters to make the solution blow up at $m$ points and the $m$ points  should be the given $x_{1,\infty}$,... $x_{m,\infty}$. The analysis in Subsection \ref{sametime} basically says for any given initial parameter $x_{1,0},...x_{m,0}$, one will be able to find $\lambda_{1,0},...\lambda_{m,0}$ to make the solution blow up at $m$ points. If one can choose $\lambda_{1,0},...\lambda_{m,0}$ according to $x_{1,0},...x_{m,0}$ in a continuous way, then the argument in \cite{merle1992solution} will be able to help us adjust $x_{1,0}, ...x_{m,0}$ to make the $m$ blow up points be exactly the prescribed $x_{1,\infty}$,... $x_{m,\infty}$. However, with our arguments in Subsection \ref{sametime}, the choice of  $\lambda_{1,0},...\lambda_{m,0}$ is not even uniquely determined by the $x_{1,0}, ...x_{m,0}$. Though this can be somehow fixed, it is very unclear whether one can choose $\lambda_{1,0},...\lambda_{m,0}$ in a continuous way 
\footnote{If one wants to direct borrow the arguments in \cite{merle1992solution}, one will need an maximal principle type argument, which says the following: Fixed $x_{1,0},...x_{m,0}$, let $\lambda_{1,0}$,...$\lambda_{m,0}$ be chosen such that the m bubbles blows up at the same time, then if one further adjusts $\lambda_{1,0}$ to be smaller and keep other parameters unchanged, then the first bubble will blow up first. This is not clear  in our setting, and we even think this argument may not hold for log-log blow up solutions.}
according to $x_{1,0},...x_{m,0}.$
Before we continue, let us first make several important observations

\begin{itemize}
\item The sharp dynamic of log-log blow up is known, we should make full use of it.
\item The impact of parameters of $x_{1,0},...x_{m,0}$ is of lower order than $\lambda_{1,0},...\lambda_{m,0}$.
\end{itemize}
Our strategy is  to choose all the parameters $x_{1,0},...x_{m,0},\lambda_{1,0},..\lambda_{m,0}$ simultaneously to make the solutions blow up at exactly m points $x_{1,\infty}$,... $x_{m,\infty}$.
Finally, at the technique level, in \cite{merle1992solution} and \cite{planchon2007existence}, they rely on the following topological lemma (they call it index Theorem).
\begin{lem}\label{ttt}
Let $f$ be a continuous  map from $R^{n}$ to $R^{n}$, let $r>0$ and suppose
\begin{equation}
 |f(y)-y|<|y|, \forall y \in \partial B_{r}
\end{equation}
 Then there is $y_{0}\in B_{r}$ such  $f(y_{0})=0$.
\end{lem}
We will need a modified version
\begin{lem}\label{topology}
Let $f$ be a map from $\Omega\subset \RRR^{n}$ to $\RRR^{n}$. Let $\Omega$ be a convex domain and let $\partial \Omega$ be a closed surface which is homeomorhpic  to the sphere.   We assume the original point is in the $\Omega$. Let us further assume for each 
$y$ in $\partial \Omega$, we have
\begin{equation}\label{home}
\boldsymbol{0}\notin \{(1-t)y+ tf(y)| t\in [0,1]\}.
\end{equation}
Then $ \boldsymbol{0}\in f(\bar \Omega).$
\end{lem}
We will prove Lemma \ref{topology} in Appendix \ref{toto}, we point out Lemma \ref{toto} actually implies Lemma \ref{ttt}.

Now, we turn to the prescription of blow up points.  As previous mentioned in \eqref{noloss}, we assume
\begin{equation}
|x_{j,\infty}-x_{j',\infty}|\geq 20, j\neq j'.
\end{equation}

We will still consider the data as in \eqref{datafamily} and we will fix $\lambda_{1,0}$ and adjust parameters  $\lambda_{2,0},...\lambda_{m,0}, x_{1,0}, ...x_{m,0}$ to make the $m$ bubbles blows up at the same time in $x_{1,\infty},...,x_{m,\infty}.$
\begin{lem}\label{lemutli}
Fix $\lambda_{1,0}$, there exists $(\beta_{2},...\beta_{m}, d_{1},...d_{m})\in [-a_{0}\lambda_{1,0},a_{0}\lambda_{1,0}]^{m-1}\times (B_{1})^{m}$, (here $B_{1}\subset \RRR^{2}$ is the unit ball) such that the associated solution $u$ to \eqref{nls} with initial data 
$$u_{\llll,\xxxx}:=u_{\lambda_{1,0},\lambda_{1,0}+\beta_{2},..\lambda_{1,0}+\beta_{m,0}, x_{1,0}+d_{1},...x_{m,0}+d_{m}}$$
will blow up at $m$ given prescribed points $x_{1,\infty}, ..x_{m,\infty}$, i.e.
\begin{equation}\label{ultigoal}
\begin{aligned}
&\lim_{t\rightarrow T^{+}(u)}\lambda_{j}(0)=0, &j=1,...,m.\\
&\lim_{t\rightarrow T^{+}(u)}x_{j}(t)=x_{j,0}, & j=1,...,m.
\end{aligned}
\end{equation}
Here, we use $\lambda_{j},x_{j}$ to denote $\lambda_{j,\llll,\xxxx}$, $x_{j,\llll,\xxxx}$ for notation convenience.
\end{lem}
\begin{proof}[Proof of Lemma \ref{lemutli}]
As we previously did in the proof of Lemma \ref{lemsim}, we write 
$u_{\lambda_{1,0},\lambda_{1,0}+\beta_{2},..\lambda_{1,0}+\beta_{m,0}, x_{1,0}+d_{1},...x_{m,0}+d_{0}}$ as $u_{\beta_{2},...\beta_{m},d_{1},...d_{m}}$. And we further write $u_{\beta_{2},...\beta_{m},d_{1},...d_{m}}$ as $u_{\AAA}$, where $\AAA=(\beta_{2},...\beta_{m},d_{1},...d_{m})\in \RRR^{m-1}\times \RRR^{2m}$.We rewrite \eqref{evolvefamiliy} as
\begin{equation}
u_{\AAA}(t,x)=\sum_{j=1}^{m}\frac{1}{\lambda_{j,\AAA}(t)}\tilde{Q}_{b_{j,\AAA}(t)}(\frac{x-x_{j,\AAA}(t)}{\lambda_{j,\AAA}})e^{i\gamma_{j,\AAA}(t)}+\Xi_{\AAA}(t,x).
\end{equation}
Let $T_{\AAA}$ be the blow up time of $u_{\AAA}$.
 We now consider the following map:
 \begin{equation}\label{finalmap}
 \begin{aligned}
 &\FFF: [-a_{0}\lambda_{1,0},a_{0}\lambda_{1,0}]^{m-1}\times (B_{1})^{m} \rightarrow \RRR^{m-1}\times \RRR^{2m} \\
 &\FFF(\AAA):= (y_{2,\AAA},...y_{m,\AAA},z_{1,\AAA},...z_{m,\AAA})\\
 & y_{i,\AAA}=\lambda_{i,\AAA}(T_{\AAA})-\lambda_{1,\AAA}(T_{\AAA}), z_{j,\AAA}=x_{j,\AAA}(T_{\AAA})-x_{j,\infty}, i=2,...m, j=1,...,m.
 \end{aligned}
 \end{equation}
 Here $B_{1}\subset \RRR^{2}$ is the unit ball. And $\lambda_{j,\AAA}(T_{\AAA})$ and $x_{j,\AAA}(T_{\AAA})$ are defined as
 \begin{equation}\label{finaldefn}
 \begin{aligned}
 \lambda_{j,\AAA}(T_{\AAA})=\lim_{t\rightarrow T_{\AAA}} \lambda_{j,\AAA}(t), j=1,...,m,\\
 x_{j,\AAA}(T_{\AAA})=\lim_{t\rightarrow T_{\AAA}}x_{j,\AAA}(t), j=1,...,m.
 \end{aligned}
 \end{equation}
 $\lambda_{j,\AAA}(T_{\AAA})$ is well defined as we have \eqref{bcontrolvariation}, which implies $\lambda_{j,\AAA}$ is strictly decreasing. $x_{j,\AAA}(T_{\AAA})$ is well defined as mentioned in Subsection \ref{sametime}, see also Remark \ref{positionlimit}. The point is 
 \begin{lem}\label{finaltechproof}
 The map $\AAA \rightarrow \lambda_{j,\AAA}(T_{\AAA})$ and the map $\AAA \rightarrow x_{j,\AAA}(T_{\AAA})$ is continuous.
 \end{lem}
 We will prove Lemma \ref{finaltechproof} in Appendix \ref{finalapp}. 
 
 Note if $\FFF(\AAA)=0$ for some $\AAA$, then $u_{\AAA}$ is the desired solution which blows up according to log-log law at exactly $m$ prescribed points.
  
 Lemma \ref{finaltechproof} implies the map $\FFF$ is continuous and we will use Lemma \ref{topology} to show 
 \begin{equation}\label{criticalindex}
 \boldsymbol{0}\in \FFF([-a_{0}\lambda_{1,0},a_{0}\lambda_{1,0}]^{m-1}\times (B_{1})^{m}).
 \end{equation}

 To achieve this ,we need to show if $\AAA$ in $\partial\{ [-a_{0}\lambda_{1,0},a_{0}\lambda_{1,0}]^{m-1}\times (B_{1})^{m}\}$, then 
 \begin{equation}\label{keytest}
 \boldsymbol{0}\notin \{t\AAA+(1-t)\FFF(\AAA), t\in [0,1]\}.
 \end{equation}
Note if $\AAA\in \partial\{ [-a_{0}\lambda_{1,0},a_{0}\lambda_{1,0}]^{m-1}\times (B_{1})^{m}\}$, then at least one of the following holds (recall the notation $\AAA=(\beta_{2},...\beta_{m},)$)
\begin{itemize}
\item Case 1:$|\beta_{j}|=a_{0}\lambda_{1,0}$, for some $j=2,...,m$.
\item Case 2:$|d_{j}|=1$, for some $j=1,...m$. 
\end{itemize}
 We first show \eqref{keytest} holds in Case 2. Indeed, by bootstrap estimate \eqref{estimatesmalltranslation}  in Lemma \ref{bootstrap},
 we have $\sup_{t<T_{\AAA}}|x_{j,\AAA}(t)-x_{j,0}-d_{j}|\leq \frac{1}{2000}$, which implies  $|d_{j}-z_{j,\AAA}\leq \frac{1}{2000}|$. (Recall our notation $\FFF(\AAA):= (y_{2,\AAA},...y_{m,\AAA},z_{1,\AAA},...z_{m,\AAA})$). Since $|d_{j}|=1$, this implies 
 $$t d_{j}+(1-t)z_{j,\AAA}\neq \boldsymbol{0},\forall \in [0,1]$$ 
 In particular \eqref{keytest} holds in Case 2.
 
 Next we show \eqref{keytest} also holds in Case 1. Without loss of generality, we assume $\beta_{2}=(a_{0})\lambda_{1,0}$.  Using Lemma \ref{sharpdynamic}, we can find $T_{1,\AAA}$, $T_{2,\AAA}$ such that
 \begin{equation}
 \begin{aligned}
& \lambda_{j,\AAA}^{2}(t)\ln |\ln\lambda_{j,\AAA}(t)|=2\pi(1+O(a_{1}))(T_{j,\AAA}-t), t<T_{\AAA},\\
& T_{1,\AAA}=2\pi(1+O(a_{1}))(\lambda_{1,0}^{2}\ln|\ln \lambda_{1,0}|),\\
&T_{2, \AAA}=2\pi(1+O(a_{1}))\left((1+a_{0})^{2}\lambda_{1,0}^{2}\ln|\ln (1+a_{0})\lambda_{1,0})|\right)
 \end{aligned}
 \end{equation}
 (Note this also implies $T_{\AAA}\leq \min (T_{1,\AAA},T_{2,\AAA})$, and since $a_{1}\ll a_{0}$, $T_{1,\AAA}<T_{2,\AAA}$).
 
 Since $a_{0}\gg a_{1}$, it is easy to see
 \begin{equation}
 \inf_{t\leq T_{1,\AAA}} \left( \lambda_{1,\AAA}^{2}(t)\ln |\ln\lambda_{1,\AAA}(t)|-\lambda_{2,\AAA}^{2}(t)\ln |\ln\lambda_{2,\AAA}(t)|\right)<0
 \end{equation}
 (Note it is important here we have  $<$ rather than $\leq$).
 
 In particular, we have $\lambda_{2,\AAA}(T_{\AAA})>\lambda_{1,\AAA}(T_{\AAA})$  , ( since $T_{\AAA}\leq T_{1,\AAA}$).
 
 Thus, since $\beta_{2}>0$, and $y_{2,\AAA}=\lambda_{2,\AAA}(T_{\AAA})-\lambda_{1,\AAA}(T_{\AAA})>0$,  
$$ t\beta_{2}+(1-t)y_{2,\AAA}\neq 0, t\in [0,1].$$
which recovers \eqref{keytest}.
This concludes the proof.
\end{proof}

\section{Acknowledgment}
The author is very grateful to his adviser, Gigliola Staffilani for consistent support and encouragement and very careful reading of the material. The author thanks Pierre Rapha\"el for very helpful comments, and in particular for telling  the topological argument in \cite{cote2011construction} and the connection between $m$ points blow up solutions and multiple standing ring blow up solutions. The author thanks Yvan Martel and Pierre Rapha\"el for very helpful discussion about the prescription of blow up points. The author thanks  Vedran Sohinger for helpful discussion  about upside-down I-method. The author also thanks Svetlana Roudenko for helpful comments. And, the author wants to thank Boyu Zhang, for discussion about the classical topology.  Finally, this material is based upon work supported by the National Science Foundation under Grant No. 0932078000 while the author was in residence at the Mathematical Sciences Research Institute in Berkeley, California, during Fall 2015 semester.
\section*{Appendix}
\begin{appendices}
\section{The local wellposedness of the modified system}\label{odelwp}
We explain briefly why one can always locally solve \eqref{modifiedsystem}, \eqref{tsnls}.
We only explain the case about  \eqref{modifiedsystem} here, \eqref{tsnls} is similar.
Indeed, the system is NLS couples with 4 ODEs. Since NLS is locally well posed and ODE is always locally well posed, it is no surprise \eqref{modifiedsystem} is locally well posed. To construct a solution, one first solve NLS $iu_{t}=-\Delta u-|u|^{2}u$ in a time interval $[0,T_{1}]$, and plug this $u(t,x)$ (which is not unknown in $[0,T_{1}]$) into the last 4 equations,  we will obtain 4 ODEs about $\{\lambda(t),b(t),x(t),\gamma(t)\}$. We just do some computation to illustrate this, say, the equation $$\frac{d}{dt}\{(\epsilon_{1}(t),|y|^{2}\Sigma_{b(t)})+(\epsilon_{2}(t),|y|^{2}\Theta_{b(t)})\}=0$$ is now equivalent to 
$$\frac{d}{dt}\Re(\frac{1}{\lambda(t)}\epsilon(\frac{x-x(t)}{\lambda(t)})e^{-i\gamma(t)},\frac{1}{\lambda(t)}(|y|^{2}\qbb)(\frac{x-x(t)}{\lambda(t)}) e^{-i\gamma(t)})=0,$$ which is equivalent to 
$$\frac{d}{dt}\Re(u-\frac{1}{\lambda(t)}\qbb(\frac{x-x(t)}{\lambda(t)}), \frac{1}{\lambda(t)}(|y|^{2}\qbb)(\frac{x-x(t)}{\lambda(t)}) )e^{-i\gamma(t)})=0,$$which is equivalent to
\begin{equation}\label{odesimp}
\begin{aligned}
&\Re (i\Delta u+i|u|^{2}, \frac{1}{\lambda(t)}(|y|^{2}\qbb)(\frac{x-x(t)}{\lambda(t)}) e^{-i\gamma(t)})\\
+&\Re (u, \frac{d}{dt} \frac{1}{\lambda(t)}(|y|^{2}\qbb)(\frac{x-x(t)}{\lambda(t)})e^{-i\gamma(t)})\\
-&\frac{d}{dt}\Re(\frac{1}{\lambda(t)}\qbb(\frac{x-x(t)}{\lambda(t)}), \frac{1}{\lambda(t)}(|y|^{2}\qbb)(\frac{x-x(t)}{\lambda(t)}) e^{-i\gamma(t)})\\
=&0.
\end{aligned}
\end{equation}
Though \eqref{odesimp} is complicate, it is an ODE about $\{\lambda(t), b(t),x(t), \gamma(t)\}$.\\
Similarly, the last 3 equations in \eqref{modifiedsystem} can also be transformed into ODEs about  $\{\lambda(t), b(t),x(t), \gamma(t)\}$.
\\
Thus the local well posedness theory about \eqref{modifiedsystem} is equivalent to the local well posedness theory about NLS.\\
\section{Proof of Lemma \ref{topology}}\label{toto}
Let us turn to the proof of Lemma \ref{topology} now.   This is very standard in algebraic topology. Note $\bar{\Omega}-\{0\}$ is the retract of $\partial{\Omega}$, i.e.
there is a map  $$r: \RRR^{n}-\{\boldsymbol{0}\}\rightarrow \partial \Omega,$$ such that
\begin{equation}
r\circ\iota= id_{\partial\Omega},
\end{equation}
here $\iota:\partial \Omega \rightarrow \RRR^{n}-\{\boldsymbol{0}\}$ is the natural inclusion map.

Now we prove Lemma \ref{topology} by contradiction. Assume $\boldsymbol{0}\notin f(\bar{\Omega})$.
Then $g:=r\circ f$ is well defined and continuous. Note $g$ is map from $\bar{\Omega}$ to $\partial \Omega$.

On the other hand $g|_{\partial_{\Omega}}$ is homotopic to the $id|_{\partial \Omega}$. Indeed, we may write down the homotopy explicitly 
$$r\circ (tf+(1-t)id).$$

We emphasize here this homotopy is well defined since $tf(y)+(1-t)y\neq \boldsymbol{0}$ for $t\in [0,1]$ and $y\in\partial \Omega$.

Now we have constructed a map $g$ from $\bar{\Omega}$ to $\partial \Omega$, and $g|_{\partial \Omega}$ is homotopic to $id|_{\partial|_{\Omega}}$. Note $\Omega$ is convex domain and $\partial \Omega$ is homeomorphic to the sphere. This is a clear contradiction from standard homology theory.
\section{Proof of Lemma \ref{finaltechproof}}\label{finalapp}
Before we go to the proof, let us point out Lemma \ref{finaltechproof} basically say the blow up point (model by $x_{j,A}(t)$) and the blow up time (modeled by $\lambda_{j,\AAA}(t)$) depending on the initial data (modeled by $\AAA$) in a continuous way. We remark here, in general, the problem whether  blow up point and blow up time  depend on the initial data in a continuous way is not an easy problem.  Indeed, if one have a sequence initial data $u_{0,n}$, whose associated solution to \eqref{nls} blows up according to the log-log law, and one assumes $u_{0,n}$ converges to $u_{0}$ in $H^{1}$, it is not always right the associated solution $u$ to \eqref{nls} will blow up according to the log-log law. And for NLS, if we don't have some information about the dynamic near the blow up time, we cannot even define the blow up point. However, see \cite{merle2004universality} for this direction.

The proof of Lemma \ref{finaltechproof} is much easier, because we are only working on data with finite parameters $\AAA$. And if $\AAA_{n}\rightarrow \AAA$, clearly $u_{\AAA}$ still blows up according the log-log law with dynamic described by Lemma \ref{bootstrap}. Let us turn to the proof. We recall $$\AAA\in [-a_{0}\lambda_{1,0},a_{0}\lambda_{1,0}]^{m-1}\times (B_{1})^{m}.$$

\begin{proof}
Recall to understand the evolution of $x_{j,\AAA}(t), \lambda_{j,\AAA}(t)$, one need to consider system \eqref{tsnls} or equivalently \eqref{tsnls2}. As explained in Appendix \ref{odelwp}, the system is NLS coupled with 4 ODEs.  Use the standard stability arguments for NLS and stability argument of ODEs, we have that for any $T<T_{\AAA}$, (recall $T_{\AAA}$ is the blow up time of $u_{\AAA}$. ) the map $\AAA \rightarrow x_{j,\AAA}(T),\lambda_{j,\AAA}(T) $ is continuous. Now, Lemma \ref{finaltechproof} easily follows from the the following lemma
\begin{lem}\label{appfinallemma}
Given $A$, for any $\epsilon>0$,there is $T<T_{\AAA}$ and $\delta=\delta(\AAA)$ such that for any $A'$ with
\begin{equation}
|\lambda_{j,\AAA'}(T)-\lambda_{j,\AAA}(T)|<\delta,
\end{equation}
then 
\begin{eqnarray}\label{appfinal1}
&&\sup_{t\in[T,T_{\AAA}']}|\lambda_{j,\AAA'}(t)-\lambda_{j,\AAA'}(T)|<\epsilon, j=1,..,m.\\ \label{appfinal2}
&&\sup_{t\in [T,T_{\AAA'}]}|x_{j,\AAA'}(t)-x_{j,\AAA'}(T)|<\epsilon, j=1,...m.
\end{eqnarray}
\end{lem}
We now prove Lemma \ref{appfinallemma}. 

Note $\AAA$ is given, and we only need to prove \eqref{appfinal1} and \eqref{appfinal2} for every given $j$. We discuss the two cases.
\begin{itemize}
\item Case 1: the $j$th bubble of  $u_{\AAA}$ blows up, i.e. $\lambda_{j,\AAA}(T_{\AAA})=0$.
\item Case 2: the $j$th bubble of  $u_{\AAA}$ blows up, i.e. $\lambda_{j,\AAA}(T_{\AAA})\neq 0$.
\end{itemize}
We first discuss Case 1. In this case ,we can choose $T$ close to $T_{\AAA}$  enough, and $\delta$ small enough,  such that $\lambda_{j,\AAA'}(T)$ is small enough, 
\begin{equation}
\lambda_{j,\AAA'}(T)\leq \epsilon_{0}\ll \epsilon^{100}.
\end{equation}
Then \eqref{appfinal1} follows since $\frac{d}{dt}\lambda_{j,\AAA'}<0$ and $\lambda_{j,\AAA'}(t)>0, \forall t<T_{\AAA'}$.
To prove \eqref{appfinal2}, one needs modify a little bit the analysis in Subsubsection \ref{subsubsectiontanslation}.
 as argued in  in Subsubsection \ref{subsubsectiontanslation}.
 \begin{equation}
 \begin{aligned}
 |x_{j,\AAA'}(t)-x_{j,\AAA'}(T)|
&\leq \int_{T}^{t}\frac{dx_{j}}{ds_{j}}\frac{1}{\lambda_{j}^{2}}d\tau
&\leq  \int_{T}^{T_{\AAA'}}\frac{1}{\lambda_{j,\AAA}}(\tau)d\tau
\end{aligned}
 \end{equation}
 (We have use $|\frac{1}{\lambda_{j,\AAA'}}\frac{dx_{j,\AAA'}}{ds_{j}}|\lesssim 1$ by \eqref{bootmodu2}.)
 And recall \eqref{tsscalefromblowuprate}, we further have
 \begin{equation}
  |x_{j,\AAA'}(t)-x_{j,\AAA'}(T)|\lesssim \lambda_{j,\AAA'}^{1/2}(T)\int_{T}^{T_{\AAA'}}\frac{1}{\lambda_{j,\AAA}^{3/2}}\lesssim \lambda_{j,\AAA'}^{1/2}(T)\ll \epsilon^{50}.
 \end{equation}
 This gives \eqref{appfinal2}.
 
 Now we discuss Case 2. Without loss of generality, we assume $\lambda_{j,\AAA'}(T)\geq \epsilon_{0}\geq \epsilon^{100}$, otherwise we just argue as Case 1.  Note since $u_{\AAA}$ blows up at $T_{\AAA}$, at least one bubble blows up at $T_{\AAA}$, let us assume 
 $\lim_{t\rightarrow T_{\AAA}}\lambda_{j_{0},\AAA}=0.$
 Thus, we are able to choose $T$ close enough to $T_{\AAA}$ and $\delta$ small enough such that
 \begin{equation}
 \lambda_{j, \AAA'}(T)\ggg \lambda_{j_{0}}(T),
 \end{equation}
 and of course $\lambda_{j_{0},\AAA'}(T)\ll \epsilon^{100}$.
 
 We remark here we actually may need 
 \begin{equation}\label{reallytech}
 \lambda_{j_{0},\AAA'}\ll \exp(-\exp\exp\exp\exp \{\frac{1}{\lambda_{j,\AAA'}(T)}\}),
\end{equation} 
  and don't worry about the special form because our arguments are kinds of soft.
  
  The idea is $u_{\AAA}$ is going to blow up so fast that the $j$th bubble do not have much time to change.
 
 Note it is not clear whether the $j_{0}$th bubble will blow up when $u_{\AAA'}$ blows up, however,  we can still estimate blow up time by Lemma \ref{sharpdynamic}, i.e. the blow up time $T_{\AAA'}$ of $u_{\AAA'}$ is controlled by $T_{j_{0},\AAA}$ predicted by Lemma \ref{sharpdynamic}, we have \footnote{Lemma \ref{sharpdynamic} is stated for solutions started at $t=0$, clearly we can also set the start time as $t=T$.}
 \begin{equation}\label{daoshu3}
 T_{\AAA'}-T\lesssim \lambda_{j_{0},\AAA'}^{2}\ln|\ln \lambda_{j_{0},\AAA'}|.
 \end{equation}
 And, using Lemma \ref{sharpdynamic} again, we can ensure 
 \begin{equation}\label{daoshu2}
 \lambda_{j,\AAA'}(t)\sim \lambda_{j,\AAA'}(T), t\in [T,T_{\AAA'}].
 \end{equation}
 Now we use estimate, (we need \eqref{bcontrolvariation}, \eqref{bcontrol2})
 \begin{equation}\label{daoshu1}
0< -\frac{d}{dt}\lambda_{j,\AAA'}\sim -\frac{1}{\lambda_{j,\AAA'}^{3}}\frac{d\lambda_{j,\AAA'}}{ds_{j}}\sim \frac{1}{\lambda_{j,\AAA'}^{3}}b_{j,\AAA'}\sim  \frac{1}{\lambda_{j,\AAA'}^{3}}\ln|\ln \lambda_{j,\AAA}|
 \end{equation}
 Combine \eqref{daoshu1} and \eqref{daoshu2}, plug in \eqref{daoshu3}, we have
 \begin{equation}
 \sup_{t\in [T,T_{\AAA'})}|\lambda_{j,\AAA'}(T)-\lambda_{j,\AAA'}(t)|\leq \lambda_{j,\AAA'}^{-3}\ln|\ln \lambda_{j,\AAA}|\lambda_{j_{0},\AAA'}^{2}\ln|\ln \lambda_{j_{0},\AAA'}|
 \end{equation}
 The desired estimate \eqref{appfinal1} follows since we have \eqref{reallytech}.
 
 Similar arguments works for \eqref{appfinal2}.
\end{proof}
\end{appendices}
\bibliographystyle{abbrv}
\bibliography{BG}

\begin{thebibliography}{10}

\bibitem{colliander2002almost}
J.~Colliander, M.~Keel, G.~Staffilani, H.~Takaoka, and T.~Tao.
\newblock Almost conservation laws and global rough solutions to a nonlinear
  schrodinger equation.
\newblock {\em arXiv preprint math/0203218}, 2002.

\bibitem{colliander2009rough}
J.~Colliander and P.~Rapha{\"e}l.
\newblock Rough blowup solutions to the {$L^2$} critical nls.
\newblock {\em Mathematische Annalen}, 345(2):307--366, 2009.

\bibitem{cote2011construction}
R.~C{\^o}te, Y.~Martel, F.~Merle, et~al.
\newblock Construction of multi-soliton solutions for the $l ^ 2$-supercritical
  gkdv and nls equations.
\newblock {\em Revista Matematica Iberoamericana}, 27(1):273--302, 2011.

\bibitem{ginibre1979class}
J.~Ginibre and G.~Velo.
\newblock On a class of nonlinear schr{\"o}dinger equations. ii. scattering
  theory, general case.
\newblock {\em Journal of Functional Analysis}, 32(1):33--71, 1979.

\bibitem{glassey1977blowing}
R.~T. Glassey.
\newblock On the blowing up of solutions to the cauchy problem for nonlinear
  schr{\"o}dinger equations.
\newblock {\em Journal of Mathematical Physics}, 18:1794--1797, 1977.

\bibitem{godet2012blow}
N.~Godet.
\newblock Blow up on a curve for a nonlinear schr$\backslash$" odinger equation
  on riemannian surfaces.
\newblock {\em arXiv preprint arXiv:1204.3301}, 2012.

\bibitem{holmer2012blow}
J.~Holmer and S.~Roudenko.
\newblock Blow-up solutions on a sphere for the 3d quintic nls in the energy
  space.
\newblock {\em Analysis \& PDE}, 5(3):475--512, 2012.

\bibitem{keel1998endpoint}
M.~Keel and T.~Tao.
\newblock Endpoint strichartz estimates.
\newblock {\em American Journal of Mathematics}, pages 955--980, 1998.

\bibitem{kenig1993well}
C.~E. Kenig, G.~Ponce, and L.~Vega.
\newblock Well-posedness and scattering results for the generalized korteweg-de
  vries equation via the contraction principle.
\newblock {\em Communications on Pure and Applied Mathematics}, 46(4):527--620,
  1993.

\bibitem{landman1988rate}
M.~Landman, G.~Papanicolaou, C.~Sulem, and P.~Sulem.
\newblock Rate of blowup for solutions of the nonlinear schr{\"o}dinger
  equation at critical dimension.
\newblock {\em Physical Review A}, 38(8):3837, 1988.

\bibitem{martel2015strongly}
Y.~Martel and P.~Raphael.
\newblock Strongly interacting blow up bubbles for the mass critical nls.
\newblock {\em arXiv preprint arXiv:1512.00900}, 2015.

\bibitem{merle1990construction}
F.~Merle.
\newblock Construction of solutions with exactly k blow-up points for the
  schr{\"o}dinger equation with critical nonlinearity.
\newblock {\em Communications in mathematical physics}, 129(2):223--240, 1990.

\bibitem{merle1992solution}
F.~Merle.
\newblock Solution of a nonlinear heat equation with arbitrarily given blow-up
  points.
\newblock {\em Communications on pure and applied mathematics}, 45(3):263--300,
  1992.

\bibitem{merle1993determination}
F.~Merle et~al.
\newblock Determination of blow-up solutions with minimal mass for nonlinear
  schr{\"o}dinger equations with critical power.
\newblock {\em Duke Mathematical Journal}, 69(2):427--454, 1993.

\bibitem{merle2003sharp}
F.~Merle and P.~Raphael.
\newblock Sharp upper bound on the blow-up rate for the critical nonlinear
  schr{\"o}dinger equation.
\newblock {\em Geometric \& Functional Analysis GAFA}, 13(3):591--642, 2003.

\bibitem{merle2004universality}
F.~Merle and P.~Raphael.
\newblock On universality of blow-up profile for {$L^{2}$} critical nonlinear
  schr{\"o}dinger equation.
\newblock {\em Inventiones mathematicae}, 156(3):565--672, 2004.

\bibitem{merle2005one}
F.~Merle and P.~Raphael.
\newblock On one blow up point solutions to the critical nonlinear
  schr{\"o}dinger equation.
\newblock {\em Journal of Hyperbolic Differential Equations}, 2(04):919--962,
  2005.

\bibitem{merle2005profiles}
F.~Merle and P.~Raphael.
\newblock Profiles and quantization of the blow up mass for critical nonlinear
  schr{\"o}dinger equation.
\newblock {\em Communications in mathematical physics}, 253(3):675--704, 2005.

\bibitem{merle2006sharp}
F.~Merle and P.~Raphael.
\newblock On a sharp lower bound on the blow-up rate for the {$L^2$} critical
  nonlinear schr{\"o}dinger equation.
\newblock {\em Journal of the American Mathematical Society}, 19(1):37--90,
  2006.

\bibitem{merle2005blow}
F.~Merle, P.~Raphael, et~al.
\newblock The blow-up dynamic and upper bound on the blow-up rate for critical
  nonlinear schrodinger equation.
\newblock {\em Annals of mathematics}, 161(1):157, 2005.

\bibitem{perelman2001blow}
G.~Perelman.
\newblock On the blow up phenomenon for the critical nonlinear schr{\"o}dinger
  equation in 1d.
\newblock {\em Nonlinear dynamics and renormalization group (Montreal, QC,
  1999)}, 27:147--164, 2001.

\bibitem{planchon2007existence}
F.~Planchon and P.~Rapha{\"e}l.
\newblock Existence and stability of the log--log blow-up dynamics for the
  {$L^2$}-critical nonlinear schr{\"o}dinger equation in a domain.
\newblock In {\em Annales Henri Poincar{\'e}}, volume~8, pages 1177--1219.
  Springer, 2007.

\bibitem{raphael2005stability}
P.~Raphael.
\newblock Stability of the log-log bound for blow up solutions to the critical
  non linear schr{\"o}dinger equation.
\newblock {\em Mathematische Annalen}, 331(3):577--609, 2005.

\bibitem{raphael2006existence}
P.~Rapha{\"e}l et~al.
\newblock Existence and stability of a solution blowing up on a sphere for an
  {$L^2$}-supercritical nonlinear schrodinger equation.
\newblock {\em Duke Mathematical Journal}, 134(2):199--258, 2006.

\bibitem{raphael2009standing}
P.~Rapha{\"e}l and J.~Szeftel.
\newblock Standing ring blow up solutions to the n-dimensional quintic
  nonlinear schr{\"o}dinger equation.
\newblock {\em Communications in Mathematical Physics}, 290(3):973--996, 2009.

\bibitem{sohinger2010bounds}
V.~Sohinger.
\newblock Bounds on the growth of high sobolev norms of solutions to 2d hartree
  equations.
\newblock {\em arXiv preprint arXiv:1003.5709}, 2010.

\bibitem{tao2006nonlinear}
T.~Tao.
\newblock {\em Nonlinear dispersive equations: local and global analysis},
  volume 106.
\newblock American Mathematical Soc., 2006.

\bibitem{weinstein1983nonlinear}
M.~I. Weinstein.
\newblock Nonlinear schr{\"o}dinger equations and sharp interpolation
  estimates.
\newblock {\em Communications in Mathematical Physics}, 87(4):567--576, 1983.

\bibitem{zwiers2010standing}
I.~Zwiers.
\newblock Standing ring blowup solutions for cubic nls.
\newblock {\em arXiv preprint arXiv:1002.1267}, 2010.

\end{thebibliography}
\end{document}